\documentclass[11pt]{amsart}
\usepackage{nccmath}

\usepackage{amsmath,amsthm,amssymb}
\usepackage{times}
\usepackage{enumerate}
\usepackage{color}
\usepackage{accents}
\newcommand{\ubar}[1]{\underaccent{\bar}{#1}}

\def\dfrac{\displaystyle\frac}
\def\dsum{\displaystyle\sum}

\newtheorem{prop}{Proposition}
\newtheorem{theo}[prop]{Theorem}
\newtheorem{lemm}[prop]{Lemma}
\newtheorem{coro}[prop]{Corollary}
\newtheorem{rmk}[prop]{Remark}

\newtheorem{defi}[prop]{Definition}

\newtheorem{claim}{Claim}

\newcommand{\be}{\begin{equation}}
\newcommand{\ee}{\end{equation}}
\newcommand{\lt}{\left}
\newcommand{\rt}{\right}

\newcommand{\ju}[2]{\begin{array}{#1}#2\end{array}}
\renewcommand{\leq}{\leqslant}
\renewcommand{\geq}{\geqslant}
\newcommand{\td}{\tilde}
\newcommand{\p}{\partial}

\newcommand{\la}{\kappa}
\newcommand{\R}{\mathbb{R}}
\newcommand{\M}{\mathcal{M}}
\newcommand{\e}{e^*}
\newcommand{\us}{u^*}
\newcommand{\ur}{u^r}
\newcommand{\urs}{u^{r*}}
\newcommand{\uri}{u^{r_i}}
\newcommand{\uus}{\bar{u}_0^*}
\newcommand{\lus}{\ubar{u}^*}
\newcommand{\lu}{\ubar{u}}
\newcommand{\uu}{\bar{u}_0}
\newcommand{\ga}{\gamma}
\newcommand{\gas}{\gamma^*}
\newcommand{\vp}{\varphi}
\newcommand{\T}{\partial}

\newcommand{\ta}{\tilde{a}}
\newcommand{\ba}{\bar{a}}
\newcommand{\ha}{\hat{a}}
\newcommand{\w}{w^*}
\numberwithin{equation}{section}

\begin{document}
\setlength{\baselineskip}{1.2\baselineskip}

\title[Constant Hessian curvature hypersurface in the Minkowski space]
{entire spacelike hypersurfaces with constant $\sigma_{n-1}$ curvature in Minkowski space}

\author{Changyu Ren}
\address{School of Mathematical Science, Jilin University, Changchun, China}
\email{rency@jlu.edu.cn}
\author{Zhizhang Wang}
\address{School of Mathematical Science, Fudan University, Shanghai, China}
\email{zzwang@fudan.edu.cn}
\author{Ling Xiao}
\address{Department of Mathematics, University of Connecticut,
Storrs, Connecticut 06269}
\email{ling.2.xiao@uconn.edu}
\thanks{Research of the first author is supported by NSFC Grant No. 11871243 and the second author is supported  by NSFC Grant No.11871161 and 11771103.}

\begin{abstract}
We prove that, in Minkowski space, if a spacelike, $(n-1)$-convex hypersurface $\M$ with constant $\sigma_{n-1}$ curvature has bounded
principal curvatures, then $\M$ is convex. Moreover, if $\M$ is not strictly convex, after an $\R^{n,1}$ rigid motion, $\M$ splits as a product $\M^{n-1}\times\R.$ We also construct nontrivial examples of strictly convex, spacelike hypersurface $\M$ with constant $\sigma_{n-1}$ curvature and bounded
principal curvatures.
\end{abstract}

\maketitle

\section{Introduction}
\label{int}

Let $\R^{n, 1}$ be the Minkowski space with the Lorentzian metric
\[ds^2=\sum_{i=1}^{n}dx_{i}^2-dx_{n+1}^2.\]
In this paper, we will study spacelike hypersurfaces with positive constant
$\sigma_{n-1}$ curvature in Minkowski space $\R^{n, 1}$. Any such hypersurface can be written locally as the graph of a function
$x_{n+1}=u(x), x\in\R^n$ satisfying the spacelike condition
\be\label{int1.1}
|Du|<1.
\ee
Before stating our main results, we need the following definition:
\begin{defi}
\label{intdef1}
A $C^2$ regular hypersurface $\M\subset \R^{n, 1}$ is $k$-convex,
if the principal curvatures of $\M$ at $X\in\M$ satisfy $\kappa[X]\in\Gamma_k$ for all $X\in\M$, where $\Gamma_k$
is the G\r{a}rding cone
\[\Gamma_k=\{\la\in\R^n|\sigma_m(\la)>0, m=1, \cdots, k\}.\]
\end{defi}

We will investigate the $(n-1)$-convex, spacelike hypersurface $\mathcal{M}_u:=\{(x, u(x))| x\in\R^{n}\}$ satisfying
\be\label{2.1}
\sigma_{n-1}(\kappa[\M_u])=\sigma_{n-1}(\kappa_1,\cdots,\kappa_n)=1,
\ee
where $\kappa[\M_u]=(\kappa_1, \cdots, \kappa_n)$ are the principal curvatures of $\M_u$ and
$\sigma_{n-1}$ is the $(n-1)$-th elementary symmetric polynomial, i.e.,
\[\sigma_{n-1}(\kappa)=\sum\limits_{1\leq i_1<\cdots< i_{n-1}\leq n}\la_{i_1}\cdots\la_{i_{n-1}}.\]

In contrast to the Euclidean case, where the existence of an entire zero mean curvature graph implies that the graph is a hyperplane only for dimensions $n\leq7,$   Cheng-Yau \cite{CY} showed that an entire spacelike maximal hypersurface (zero mean curvature) in Minkowski space is a hyperplane for all dimensions. This raised the question of whether the only entire spacelike hypersurface of constant mean curvature (CMC) in Minkowski space is the hyperboloid. In \cite{T}, Treibergs answered this question by showing that for an arbitrary $C^2$ perturbation of the light cone in Minkowski space, one can construct a spacelike CMC hypersurface which is asymptotic to this perturbation. Moreover, Treibergs also showed that every entire spacelike CMC hypersurface is convex and has bounded principal curvatures. Later, Choi-Treibergs \cite{CT} further studied the Guass map of the CMC hypersurfaces.
They proved that the Gauss map of a spacelike CMC hypersurface in Minkowski space is a harmonic map to hyperbolic space. Furthermore, they showed, given an arbitrary closed set in the ideal boundary at infinity of hyperbolic space, there are many complete entire spacelike CMC hypersurfaces whose Gauss maps are diffeomorphisms onto the interior of the convex hull of the corresponding set in the unit ball.

It may be traced back to Liebmann \cite{Lie} who showed that every compact embedded 2-dimensional surface with constant Gauss curvature is a sphere. Hsiung \cite{Hsi} extended this result to all dimensions. Later, a classical result of Aleksandrov \cite{Ale} states that a compact embedded hypersurface with constant mean curvature is a sphere.  One can investigate what happens if we replace the `` constant mean curvature'' by `` constant $\sigma_k$ curvature.''
In case $k=2,$ Cheng-Yau \cite{CY3} showed the Aleksandrov result still holds. In \cite{Ros}, A. Ros extended Cheng-Yau's result, showing that the sphere is the only embedded compact  $k$-convex hypersurface with constant $\sigma_k$ curvature. Around the same time, Ecker-Huisken \cite{EH} proved a similar result using different approaches.

Inspired by the similarity of constant $\sigma_k$ curvature hypersurfaces and CMC hypersurfaces in Euclidean space, it is natural to ask, do constant $\sigma_k$ curvature hypersurfaces also share similar properties as CMC hypersurfaces in Minkowski space? More specifically, can we show that every entire spacelike constant $\sigma_k$ curvature hypersurface is convex and has bounded principal curvatures?

In \cite{GJS}, Guan-Jian-Schoen considered the existence and regularity of entire spacelike constant Gauss curvature hypersurfaces with prescribed tangent cone at infinity. In addition, they showed that there do exist entire spacelike constant Gauss curvature hypersurfaces with unbounded principal curvatures. Despite so, there still are some nice properties for entire spacelike constant Gauss curvature hypersurfaces. In \cite{LA}, by studying the Legendre transform of entire spacelike constant Gauss curvature hypersurfaces, Anmin Li showed that one can construct spacelike constant Gauss curvature hypersurfaces with bounded principal curvatures
whose Guass map image is the unit ball.

Due to technical reasons, the study of entire spacelike constant $\sigma_k$ curvature hypersurface remains wide open.
The main difficulties are the following: first, we can't show that the entire spacelike constant
$\sigma_k$ curvature hypersurfaces are convex; second, the principal curvatures of the entire spacelike constant $\sigma_k$ curvature hypersurfaces
are not necessarily bounded; lastly, in the process of constructing entire spacelike constant $\sigma_k$ curvature hypersurfaces, we need to solve
a corresponding Dirichlet problem (see \cite{BS} Appendix B for example). Unfortunately, we do not have the existence result for such Dirichlet problems in general. In this paper, we will overcome these difficulties and study the convexity and existence of the entire spacelike constant $\sigma_{n-1}$ curvature hypersurfaces.

We will divide this paper into two parts. In the first part, we will prove that every entire spacelike constant $\sigma_{n-1}$ curvature hypersurface with bounded principal curvatures is convex. In the second part, we will construct nontrivial examples of strictly convex, spacelike hypersurfaces that have bounded principal curvatures and satisfy equation \eqref{2.1}. In particular, we will prove
\begin{theo}
\label{intth1.1}
Let $\M$ be an $(n-1)$-convex, spacelike hypersurface with bounded principal curvatures, and $\M$ satisfies equation
\eqref{2.1}. Then $\M$ is convex. Moreover, if $\M$ is not strictly convex, then after an $\R^{n,1}$ rigid motion, $\R^{n, 1}$ splits as a product
$\R^{n-1,1}\times\R$ such that $\M$ also splits as a product $\M^{n-1}\times\R.$ Here $\M^{n-1}\subset\R^{n-1,1}$ is a strictly convex, $(n-1)$-dimensional graph whose Gauss curvature is equal to $1$.
\end{theo}

\begin{rmk}
One may compare this theorem with constant rank theorems in Euclidean space (see \cite{CGM} and \cite{GM} for example). Recall that constant rank theorems
in Euclidean space only study compact hypersurfaces, which are essentially assuming the principal curvatures are bounded.
In this sense, our theorem is similar to constant rank theorems in Euclidean space. On the other hand, we don't need to assume the convexity of $\M$ at any point. Therefore, our theorem is much stronger than constant rank theorems.
\end{rmk}

In order to construct entire, spacelike, constant $\sigma_{n-1}$ curvature hypersurfaces with bounded principal curvatures, we will
use Anmin Li's idea (see \cite{LA}) and consider the Legendre transform of the solution to equation \eqref{2.1}. We will show in Section \ref{gg} that the study of complete, spacelike, convex hypersurfaces $\mathcal{M}_u$ with bounded principal
curvatures and satisfying $\sigma_{n-1}(\la[\mathcal{M}_u])=1$ can be reduced to the study of the following equation:
\be\label{int1.2}
\left\{
\begin{aligned}
F(\w\gas_{ik}\us_{kl}\gas_{lj})&=1,\,\,\mbox{in $B_1$}\\
\us&=\vp,\,\,\mbox{on $\partial B_1$,}
\end{aligned}
\right.
\ee
where $B_1=\{\xi\in\R^n| |\xi|<1\},$ $\us$ is the Legendre transform of $u,$ $\vp\in C^2(\partial B_1),$
$$\w=\sqrt{1-|\xi|^2},\ \ \gas_{ik}=\delta_{ik}-\frac{\xi_i\xi_k}{1+\w},\ \ \us_{kl}=\frac{\T^2\us}{\T\xi_k\T\xi_l},$$ and $$F(\w\gas_{ik}\us_{kl}\gas_{lj})=\lt(\frac{\sigma_n}{\sigma_1}(\la^*[\w\gas_{ik}\us_{kl}\gas_{lj}])\rt)^{\frac{1}{n-1}}.$$
Here, $\la^*[\w\gas_{ik}\us_{kl}\gas_{lj}]=(\la^*_1, \cdots, \la^*_n)$ are the eigenvalues of the matrix $(\w\gas_{ik}\us_{kl}\gas_{lj})$.

\begin{theo}
\label{intth1.2}
Given a $C^2$ function $\varphi$ on $B_1$, there is a unique strictly convex solution $\us\in C^{\infty}(B_1)\cap C^0(\bar{B}_1)$ to the equation
\eqref{int1.2}. Moreover, the Legendre transform of $\us,$ which we will denote by $u$ satisfies
\[\sigma_{n-1}(\la[\M_u])=1\,\,\mbox{and $\la[\M_u]\leq C$}.\]
Here, $\M_u=\{(x, u(x)) |\, x\in\R^n\}$ is the spacelike graph of $u,$ $\la[\M_u]$ denotes the principal curvatures of $\M_u,$ and the constant $C$ only depends on $|\varphi|_{C^2}$.
 \end{theo}

We can also state the above theorem using geometric terminologies:
\begin{coro}
For any given $C^2$ function $\varphi$ defined on the $(n-1)$-dimensional sphere, there exits an unique entire graphical hypersurface $\M=\{(x, u(x))| x\in\R^n\}$ with constant $\sigma_{n-1}$ curvature and bounded principal curvatures. Moreover, the ideal boundary of $\M$ is the sphere, and on the ideal boundary, the Legendre transform of $u$ is equal to $\varphi$.
 \end{coro}

\begin{rmk}
We can generalize the result of Theorem \ref{intth1.2} to spacelike constant $\sigma_k$ curvature hypersurfaces for all $1\leq k\leq n$.
We will include this result in an upcoming paper, where we will focus on studying properties of spacelike constant $\sigma_k$ curvature hypersurfaces for all
$1\leq k\leq n$.
\end{rmk}
An outline of the paper is as follows. In Section 2, we introduce some basic formulas, notations, as well as properties of
the $k$-th elementary symmetric function that will be used in later sections. Sections 3, 4, 5, and 6 are devoted to proving Theorem \ref{intth1.1}. We will see (for details see Section 6 Lemma \ref{bplem1.1}) that the key step in proving Theorem \ref{intth1.1} is to prove Theorem \ref{ceth1.1} (see Section 3), which is carried out in Sections 3, 4, and 5. More specifically, in Section 3, we reduce the proof of Theorem \ref{ceth1.1} to the proof of the semi-positivity of a symmetric matrix $S$ (see the last 2 paragraphs of Section 3 for the definition of $S$). In Sections 4 and 5, we show that $S$ is indeed semi-positive. Since these two sections involve very delicate and complicated calculations, first-time readers may want to skip this part. We prove the splitting theorem and complete the proof of Theorem \ref{intth1.1} in Section 6. Sections 7, 8, 9 10, and 11 are devoted to proving Theorem \ref{intth1.2}. In this part, we use Legendre transform to construct many examples of strictly convex solutions with bounded principal curvatures to equation \eqref{2.1}. In particular, in Section 7, we investigate spacelike hypersurfaces under the Gauss map and the Legendre transform respectively. The reason we look at two models in Section 7 is that each model has its own advantages in the study of the corresponding Dirichlet problem (see Section 8 and 9) and convergence result (see Section 10). We prove that the solution to equation \eqref{int1.2} exists in Section 8, 9 and 10. In Section 11, we show that the Legendre transform of this solution satisfies equation \eqref{2.1} and has bounded principal curvatures.
This completes the proof of Theorem \ref{intth1.2}.

\section*{Acknowledgements}
Part of this work was done while Z.W. was visiting Jilin University, and he would like to thank their hospitality. Part of this work was done while L.X. was visiting the School of Mathematical Science at Fudan University, and she gratefully acknowledges their
hospitality. L.X. would also like to thank Matthew McGonagle for helpful conversations on various aspects of this work.
\bigskip

\section{Preliminaries}
\label{pre}
We first recall some basic formulas for the geometric quantities of spacelike hypersurfaces
in Minkowski space $\R^{n,1},$ which is $\R^{n+1}$ endowed with the Lorentzian metric
$$ds^2=dx_1^2+\cdots dx_{n}^2-dx_{n+1}^2.$$
Throughout this paper, $\lt<\cdot, \cdot\rt>$ denotes the inner product in $\R^{n,1}$. The corresponding Levi-Civita
connection is denoted by $\bar{\nabla}.$

A spacelike hypersurface $\M$ in $\R^{n, 1}$ is a codimension-one submanifold whose induced metric
is Riemannian. Locally $\M$ can be written as a graph
\[\M_u=\{X=(x, u(x))| x\in\R^n\}\]
satisfying the spacelike condition \eqref{int1.1}. Let
$E=(0, \cdots, 0, 1),$ then the height function of $\M$ is $u(x)=-\lt<X, E\rt>.$ It's easy to see that the induced metric and second fundamental form of $\M$ are given
by
$$g_{ij}=\delta_{ij}-D_{x_i}uD_{x_j}u, \ \  1\leq i,j\leq n,$$
and
\[h_{ij}=\frac{u_{x_ix_j}}{\sqrt{1-|Du|^2}},\]
while the timelike unit normal vector field to $\M$ is
\[\nu=\frac{(Du, 1)}{\sqrt{1-|Du|^2}},\]
where $Du=(u_{x_1}, \cdots, u_{x_n})$ and $D^2u=\lt(u_{x_ix_j}\rt)$ denote the ordinary gradient and Hessian of $u$,
respectively.

One important example of the spacelike hypersurface with constant mean curvature is the hyperboloid
\[u(x)=\lt(\frac{n^2}{H^2}+\sum_{i=1}^{n}x_i^2\rt)^{1/2},\]
which is umbilic, i.e., it satisfies  $\la_1=\la_2=\cdots=\la_n=\frac{H}{n}.$ Other examples of spacelike CMC hypersurfaces include hypersurfaces of revolution,
in which case the graph takes the form $u(x)=\sqrt{f(x_1)^2+|\bar{x}|^2},\,\,x=(x_1, \bar{x})=(x_1, \cdots, x_n)\in\R^n,$ where $f$ is a function only depending on $x_1$.

Now, let $\{\tau_1,\tau_2,\cdots,\tau_n\}$ be a local orthonormal frame on $T\M$. We will use $\nabla$ to denote
the induced Levi-Civita connection on $\M.$ For a function $v$ on $\M$, we denote $v_i=\nabla_{\tau_i}v,$ $v_{ij}=\nabla_{\tau_i}\nabla_{\tau_j}v,$ etc.
In particular, we have
\[|\nabla u|=\sqrt{g^{ij}u_{x_i}u_{x_j}}=\frac{|Du|}{\sqrt{1-|Du|^2}}.\]

We also need the following well known fundamental equations for a hypersurface $\M$ in $\R^{n, 1}:$
\begin{equation}\label{Gauss}
\begin{array}{rll}
X_{ij}=& h_{ij}\nu\quad {\rm (Gauss\ formula)}\\
(\nu)_i=&h_{ij}\tau_j\quad {\rm (Weigarten\ formula)}\\
h_{ijk}=& h_{ikj}\quad {\rm (Codazzi\ equation)}\\
R_{ijkl}=&-(h_{ik}h_{jl}-h_{il}h_{jk})\quad {\rm (Gauss\ equation)},\\
\end{array}
\end{equation}
where $R_{ijkl}$ is the $(4,0)$-Riemannian curvature tensor of $\M$, and the derivative here is covariant derivative with respect to the metric on $\M$.
It is clear that the Gauss formula and the Gauss equation in \eqref{Gauss} are different from those in Euclidean space.
Therefore, the Ricci identity becomes,
\begin{equation}\label{a1.2}
\begin{array}{rll}
h_{ijkl}=& h_{ijlk}+h_{mj}R_{imlk}+h_{im}R_{jmlk}\\
=& h_{klij}-(h_{mj}h_{il}-h_{ml}h_{ij})h_{mk}-(h_{mj}h_{kl}-h_{ml}h_{kj})h_{mi}.\\
\end{array}
\end{equation}

Although in this paper we only study constant $\sigma_{n-1}$ curvature hypersurface, we will need to use other elementary sysmetric polynomials in the process.
In the following, we will introduce notations and properties for general curvature functions.

Recall that the $k$-th elementary symmetric polynomial is
defined by,
$$
\sigma_k(\kappa)=\dsum_{1\leq i_1<\cdots<i_k\leq
n}\kappa_{i_1}\cdots \kappa_{i_k},
$$
where $\kappa=(\kappa_1,\kappa_2,\cdots,\kappa_n)\in\mathbb{R}^n$ and $1\leq k\leq n.$
We also set  $\sigma_0(\kappa)=1$ and $\sigma_k(\kappa)=0$ for
$k>n$. It's well known that, a suitable domain of definition for
$\sigma_k$ is the G\r{a}rding cone $\Gamma_k$ (see \cite{CNS3}).
By the definition of $\Gamma_k$, we can see that
$$
\Gamma_n\subset\cdots\subset\Gamma_k\cdots\subset\Gamma_1.
$$
Moreover, Korevaar \cite{Kor} showed that the G\r{a}rding cone $\Gamma_k$ can also be
characterized as
\begin{eqnarray}\label{Gamma}
&\\
&\left\{\kappa\in
\R^n;\sigma_k(\kappa)>0,\frac{\partial\sigma_k(\kappa)}{\partial\kappa_{i_1}}>0,
\cdots,\frac{\partial^k\sigma_k(\kappa)}{\partial\kappa_{i_1}\cdots\partial
\kappa_{i_k}}>0, \text{ for all } 1\leq i_{1}<\cdots<i_{k}\leq n\right\}.\nonumber
\end{eqnarray}
This characterization will be used throughout this paper. It's particularly useful in analyzing equation \eqref{3.9}.

Let $\mathcal{S}$ be the vector space of $n \times n$ symmetric matrices and
\[\mathcal{S}_K=\{W\in\mathcal{S}: \la[W]\in K\},\]
where $\la[W] = (\la_1, \cdots, \la_n)$ denotes the eigenvalues of $W,$ and $K$ is the admissable set, for example, $\Gamma_k$.
We let $\kappa[W]$ represent the
eigenvalues of the matrix $W=(w_{ij})$. Define a function $F$ by
$$f\left(\kappa [W]\right)=F(W).$$
Throughout this paper we denote,
$$F^{pq}=\frac{\p F}{\p w_{pq}}, \text{ and  } F^{pq,rs}=\frac{\p^2 F}{\p w_{pq}\p w_{rs}}.$$
The matrix $\lt(F^{ij}(W)\rt)$ is symmetric and has eigenvalues $f_1, \cdots, f_n,$
where $f_i=\frac{\partial f}{\partial\la_i},$ $1\leq i\leq n.$
Moreover, if $f$ is a concave function in $K,$ then $F$ is concave as well. That is,
\[F^{ij,kl}(W)\xi_{ij}\xi_{kl}\leq 0, \forall \lt(\xi_{ij}\rt)\in\mathcal{S}, W\in\mathcal{S}_K.\]
In particular, we should keep in mind that both
$\sigma_k^{1/k}(\kappa)$  and $\left(\frac{\sigma_k(\kappa)}{\sigma_l(\kappa)}\right)^{1/(k-l)},$ $l<k,$ are concave functions
in $\Gamma_k,$ for $1\leq k\leq n.$ Let's recall the following well known Lemma (see \cite{Ball}) which will be needed in the proof of Theorem \ref{intth1.2}.

\begin{lemm}\label{prelm1.1}
Denote $Sym(n)$ the set of all $n \times n$ symmetric matrices. Let $F$ be a $C^2$
symmetric function defined in some open subset $\Psi\subset Sym(n).$ At any diagonal matrix
$W\in\Psi$ with distinct eigenvalues, let $\ddot{F}(B, B)$ be the second derivative of $C^2$
symmetric function $F$ in direction $B\in Sym(n)$, then
\[\ddot{F}(B, B)=\sum_{j, k=1}^nf_{jk}B_{jj}B_{kk}+2\sum_{j<k}\frac{f_j-f_k}{\la_j-\la_k}B_{jk}^2,\]
where $f_j=\frac{\partial f}{\partial\la_j}$ and $f_{jk}=\frac{\partial^2 f}{\partial\la_j\partial\la_k}.$
\end{lemm}
From the discussion above, we can see that the definition of the $k$-th elementary symmetric polynomial can
be extended to symmetric matrices. Suppose $W$ is an $n\times n$
symmetric matrix and $\kappa[W]\subset\Gamma_k$. We
define
$$\sigma_k(W)=\sigma_k(\kappa[W]).$$

In the following, we list some algebraic identities and properties of $\sigma_k$ that will be used later.
For $1\leq l\leq n$, we define
$\sigma_l(\kappa|a)$ the $l$-th elementary
symmetric polynomial of $\kappa_1,\kappa_2,\cdots,\kappa_n$ with
$\kappa_a=0$, $\sigma_l(\kappa|ab)$ the $l$-th elementary
symmetric polynomial of $\kappa_1,\kappa_2,\cdots,\kappa_n$ with
$\kappa_a=\kappa_b=0,$ and similarly, we can define $\sigma_l(\kappa|abc\cdots).$   Thus, we have

\par
\noindent (i) $\sigma^{pp}_k(\kappa):=\dfrac{\partial \sigma_k(\kappa)}{\partial
\kappa_p}=\sigma_{k-1}(\kappa|p)$ for any $p=1,\cdots,n$;
\par
\noindent (ii)
$\sigma^{pp,qq}_k(\kappa):=\dfrac{\partial^2 \sigma_k(\kappa)}{\partial
\kappa_p\partial \kappa_q}=\sigma_{k-2}(\kappa|pq)$  for any
$p,q=1,\cdots,n$ and $\sigma^{pp,pp}_k(\kappa)=0$;

\par
\noindent (iii)
$\sigma_k(\kappa)=\kappa_i\sigma_{k-1}(\kappa|i)+\sigma_k(\kappa|i)$ for any fixed $1\leq i\leq n$;
\par
\noindent (iv)
$\dsum_{i=1}^n\kappa_i\sigma_{k-1}(\kappa|i)=k\sigma_k(\kappa)$.\\

\noindent Moreover, for a Codazzi  tensor $W=(w_{ij})$, if $W$ is
diagonal, then we have

\noindent (v)
$-\sum_{p,q,r,s}\sigma^{pq,rs}_kw_{pql}w_{rsl}=\sum_{p,q}\sigma_k^{pp,qq}w_{pql}^2-\sum_{p,q}\sigma^{pp,qq}_kw_{ppl}w_{qql}$,\\
\noindent where $w_{pql}$ is the covariant derivative of $w_{pq}$
and $\sigma_k^{pq,rs}=\frac{\p^2 \sigma_k(W)}{\p w_{pq}\p w_{rs}}$.
The definition of the Codazzi tensor can be found in \cite{GRW}.

\noindent For $\kappa\in \Gamma_k$, if we assume
$\kappa_1\geq\cdots\geq\kappa_n$, then we have
\par
\noindent (vi) $\sigma_{k-1}(\kappa|n)\geq \cdots\geq
\sigma_{k-1}(\kappa|1)>0$;
\par
\noindent (vii) $\kappa_1\sigma_{k-1}(\kappa|1)\geq C_{n,k}\sigma_k(\kappa)$,
\par\noindent
where $C_{n,k}$ is a positive constant depending only on $n,k$.
Details of the proof of these formulas can be found in \cite{HS} and
\cite{Wxj}.

\bigskip

\section{Convexity estimates of the hypersurface}
\label{ce}

In this section, we will start to study the convexity of the spacelike hypersurface $\M_u=\{(x, u(x))| x\in\R^n, |Du|<1\}$ that satisfies
the following conditions:
\begin{itemize}
\item $\kappa[\M_u]\in\Gamma_{n-1};$
\item $\sigma_{n-1}(\kappa[\M_u])=1.$
\end{itemize}
If a spacelike hypersurface $\M_u$ satisfies these conditions, then we say $\M_u$ is \textit{admissible}.
Note that here we don't require the principal curvatures of $\M_u$ are bounded.

Next, we will state one of our main theorems. This theorem plays a key role in the proof
of Theorem \ref{intth1.1}.

\begin{theo}
\label{ceth1.1}
Let $\M_u$ be an admissible hypersurface and $\kappa[\M_u]$ be its principal curvatures, then we have
\be\label{ce1.1}
\sigma_{n-1}^{ij}(\sigma_n(\kappa[\M_u]))_{ij}\leq\sigma_1\sigma_{n-1}\sigma_n-n^2\sigma_n^2.
\ee
\end{theo}

Since the proof of Theorem \ref{ceth1.1} is very complicated, we will split it into 3 sections. In this section, we will simplify our
equation and reduce the proof of this theorem into the proof of the semi-positivity of a matrix. In the next two sections, we will confirm that the matrix we obtain here is indeed semi-positive.

\begin{lemm}
\label{celem1.1}
Inequality \eqref{ce1.1} holds if the following inequality holds on $\M_u$:
\be\label{ce1.2}
\begin{aligned}
&\sum_{j\neq1}\sigma^2_{n-2}(\la|1j)\lt[2\sigma_{n-2}(\la|1)+2\sigma_{n-2}(\la|j)\rt]h^2_{jj1}\\
+&\sum_{p,q\neq1;p\neq q}\sigma_{n-1}(\la|1)\sigma_{n-3}(\la|1pq)
\lt[\sigma_{n-2}(\la|1)+(\la_p+\la_q-2\la_1)\sigma_{n-3}(\la|1pq)\rt]h_{pp1}h_{qq1}\geq 0.
\end{aligned}
\ee
where $h_{ijk}$ is the covariant derivative of the second fundamental form $h_{ij}$ .
\end{lemm}
\begin{proof}
For an arbitrary $X_0\in\M_u,$ we can choose an orthonormal local frame $\tau_1, \cdots, \tau_n$ around $X_0$ on $T\M_u,$
such that at $X_0,$
\[h_{ij}=\kappa_i\delta_{ij},\]
where $\la_1, \cdots , \la_n$ are the principal curvatures of $\M_u$ at $X_0$. Our calculation below is done at the point $X_0.$
We will consider the test function
$$\phi=\sigma_n(h).$$
Differentiating $\phi$ twice, we get
\begin{equation}
\phi_i=\sigma_n^{jj}h_{jji}
\end{equation}
and
\begin{equation}
\phi_{ii}=\sigma_n^{jj}h_{jjii}+\sigma_n^{pq,rs}h_{pqi}h_{rsi}.
\end{equation}
Contracting with $\sigma_{n-1}^{ii}$ on both sides we have,
\begin{equation}
\sigma_{n-1}^{ii}\phi_{ii}=\sigma_{n-1}^{ii}\sigma_n^{jj}h_{jjii}+\sigma_{n-1}^{ii}\sigma_n^{pq,rs}h_{pqi}h_{rsi}.
\end{equation}

Now, let's differentiate equation \eqref{2.1} twice, then we obtain
\begin{equation}\label{3.5}
\sigma_{n-1}^{ii}h_{iij}=0,
\end{equation}
and
\begin{equation}
\sigma_{n-1}^{ii}h_{iijj}+\sigma_{n-1}^{pq,rs}h_{pqj}h_{rsj}= 0.
\end{equation}
By \eqref{a1.2}, we can see that at $X_0$ we have
$$h_{jjii}=h_{iijj}+h_{ii}^2h_{jj}-h_{ii}h_{jj}^2.$$
Thus, we get
\begin{eqnarray}\label{3.7}
&&\sigma_{n-1}^{ii}\phi_{ii}\\
&=&-\sigma_n^{jj}\sigma_{n-1}^{pq,rs}h_{pqj}h_{rsj}+\sigma_{n-1}^{ii}\sigma_n^{pq,rs}h_{pqi}h_{rsi}+\sigma_{n-1}^{ii}\sigma_n^{jj}(h^2_{ii}h_{jj}-h_{ii}h^2_{jj})\nonumber\\
&=&(\sigma_{n-1}^{ii}\sigma_{n}^{pp,qq}-\sigma_{n}^{ii}\sigma_{n-1}^{pp,qq})h_{ppi}h_{qqi}-(\sigma_{n-1}^{ii}\sigma_{n}^{pp,qq}-\sigma_{n}^{ii}\sigma_{n-1}^{pp,qq})h_{pqi}^2\nonumber\\
&&+\sigma_1\sigma_{n-1}\sigma_n-n^2\sigma_n^2,\nonumber
\end{eqnarray}
where we used $\sigma_{n-1}^{pq,rs}h_{pqi}h_{rsi}=\sigma_{n-1}^{pp, qq}h_{ppi}h_{qqi}-\sigma_{n-1}^{pp, qq}h_{pqi}^2.$

In order to prove inequality \eqref{ce1.1} we only need to prove
\begin{equation}\label{3.8}
(\sigma_{n-1}^{ii}\sigma_{n}^{pp,qq}-\sigma_{n}^{ii}\sigma_{n-1}^{pp,qq})h_{ppi}h_{qqi}
-(\sigma_{n-1}^{ii}\sigma_{n}^{pp,qq}-\sigma_{n}^{ii}\sigma_{n-1}^{pp,qq})h_{pqi}^2\leq 0.
\end{equation}

First, when the indices $i, p, q$ are not equal to each other, by a straightforward calculation we get,
\begin{eqnarray}\label{3.9}
&&\sigma_{n-1}^{ii}\sigma_{n}^{pp,qq}-\sigma_{n}^{ii}\sigma_{n-1}^{pp,qq}\\
&=&\sigma_{n-2}(\la |i)\sigma_{n-2}(\la | pq)-\sigma_{n-1}(\la | i)\sigma_{n-3}(\la |pq)\nonumber\\
&=&[\la_p\la_q\sigma_{n-4}(\la |ipq)+(\la_p+\la_q)\sigma_{n-3}(\la |ipq)]\la_i\sigma_{n-3}(\la | ipq)\nonumber\\
&&-\la_p\la_q\sigma_{n-3}(\la |ipq)[\la_i\sigma_{n-4}(\la | ipq)+\sigma_{n-3}(\la | ipq)]\nonumber\\
&=&\sigma_{n-3}^2(\la|ipq)(\la_i\la_p+\la_i\la_q-\la_p\la_q)\nonumber.
\end{eqnarray}

Since $\la\in\Gamma_{n-1},$ by \eqref{Gamma} we have
\[\sigma_2(\la_i, \la_p, \la_q)=\la_i\la_p+\la_i\la_q+\la_p\la_q>0.\]
Also note that
\[\sigma_{n-3}(\la|ipq)=\sigma_{n-3}(\la|piq)=\sigma_{n-3}(\la|iqp).\]
Therefore, by rotating $i, p, q$ and summing them up we obtain
\begin{equation}\label{3.11}
-\dsum_{i\neq p\neq
q}(\sigma_{n-1}^{ii}\sigma_{n}^{pp,qq}-\sigma_{n}^{ii}\sigma_{n-1}^{pp,qq})h_{pqi}^2\leq 0.
\end{equation}

In view of \eqref{3.8}, now we only need to prove, for any fixed $i,$ $1\leq i\leq n,$
\begin{equation}
\label{ce1.2*}
L_i:=2\dsum_{j\neq i}(\sigma_{n-1}^{jj}\sigma_{n}^{ii,jj}-\sigma_{n}^{jj}\sigma_{n-1}^{ii,jj})h_{jji}^2-\dsum_{p\neq q}
(\sigma_{n-1}^{ii}\sigma_{n}^{pp,qq}-\sigma_{n}^{ii}\sigma_{n-1}^{pp,qq})h_{ppi}h_{qqi}\geq
0.
\end{equation}

Next, without loss of generality, we will consider the case when $i=1.$ Namely, we will show that $L_1\geq 0.$

From equation \eqref{3.5},we have
\be\label{ce1.3}
h_{111}=-\dsum_{j\neq1}\frac{\sigma_{n-1}^{jj}}{\sigma_{n-1}^{11}}h_{jj1}.
\ee
Plugging it into equation \eqref{ce1.2*} we get,
\begin{align*}
L_1=&2\dsum_{j\neq 1}(\sigma_{n-1}^{jj}\sigma_{n}^{11,jj}-\sigma_{n}^{jj}\sigma_{n-1}^{11,jj})h_{jj1}^2-\dsum_{p\neq q,p,q\neq 1}(\sigma_{n-1}^{11}\sigma_{n}^{pp,qq}-\sigma_{n}^{11}\sigma_{n-1}^{pp,qq})h_{pp1}h_{qq1}\\
&+\dsum_{q\neq 1}(\sigma_{n-1}^{11}\sigma_{n}^{11,qq}-\sigma_{n}^{11}\sigma_{n-1}^{11,qq})\dsum_{r\neq 1}\frac{\sigma_{n-1}^{rr}}{\sigma_{n-1}^{11}} h_{rr1}h_{qq1}\\
&+\dsum_{p\neq 1}(\sigma_{n-1}^{11}\sigma_{n}^{11,pp}-\sigma_{n}^{11}\sigma_{n-1}^{11,pp})\dsum_{s\neq 1}\frac{\sigma_{n-1}^{ss}}{\sigma_{n-1}^{11}} h_{pp1}h_{ss1}\\
=&\dsum_{j\neq
1}\left(4\sigma_{n-1}^{jj}\sigma_{n}^{11,jj}-2\sigma_{n}^{jj}\sigma_{n-1}^{11,jj}
-2\sigma_{n}^{11}\sigma_{n-1}^{11,jj}\frac{\sigma_{n-1}^{jj}}{\sigma_{n-1}^{11}}\right)h_{jj1}^2\\
&+\dsum_{p\neq q, p,q\neq 1}\Big[\sigma_{n-1}^{pp}\sigma_{n}^{11,qq}+\sigma_{n-1}^{qq}\sigma_{n}^{11,pp}-\sigma_{n-1}^{11}\sigma_{n}^{pp,qq}+\sigma_{n}^{11}\sigma_{n-1}^{pp,qq}\\
&\left.-\frac{\sigma_{n}^{11}}{\sigma_{n-1}^{11}}\left(\sigma_{n-1}^{11,pp}\sigma_{n-1}^{qq}+\sigma_{n-1}^{11,qq}\sigma_{n-1}^{pp}\right)
\right]h_{pp1}h_{qq1}.
\end{align*}
Thus, we have
\be\label{ce1.4}
\begin{aligned}
\sigma_{n-1}^{11}L_1=&\dsum_{j\neq
1}\left(4\sigma_{n-1}^{11}\sigma_{n-1}^{jj}\sigma_{n}^{11,jj}-2\sigma_{n-1}^{11}\sigma_{n}^{jj}\sigma_{n-1}^{11,jj}
-2\sigma_{n}^{11}\sigma_{n-1}^{jj}\sigma_{n-1}^{11,jj}\right)h_{jj1}^2\\
&+\dsum_{p\neq q,p,q\neq 1}\left(\sigma_{n-1}^{11}\sigma_{n-1}^{pp}\sigma_{n}^{11,qq}+\sigma_{n-1}^{11}\sigma_{n-1}^{qq}\sigma_{n}^{11,pp}
+\sigma_{n-1}^{11}\sigma_{n}^{11}\sigma_{n-1}^{pp,qq}\right.\\
&\left.-\sigma_{n-1}^{11}\sigma_{n-1}^{11}\sigma_{n}^{pp,qq}
-\sigma_{n}^{11}\sigma_{n-1}^{11,pp}\sigma_{n-1}^{qq}-\sigma_{n}^{11}\sigma_{n-1}^{11,qq}\sigma_{n-1}^{pp} \right) h_{pp1}h_{qq1}\\
\end{aligned}
\ee
Finally, we want to simplify equation \eqref{ce1.4}. It is straightforward to verify

\begin{align*}
&\sigma_{n-1}^{jj}\sigma_{n}^{11,jj}-\sigma_{n}^{jj}\sigma_{n-1}^{11,jj}\\
=&\sigma_{n-2}(\la |j)\sigma_{n-2}(\la |1j)-\sigma_{n-1}(\la |j)\sigma_{n-3}(\la |1j)\\
=&\big(\la_1\sigma_{n-3}(\la |1j)+\sigma_{n-2}(\la |1j)\big)\sigma_{n-2}(\la |1j)-\la_1\sigma_{n-2}(\la |1j)\sigma_{n-3}(\la |1j)\\
=&\sigma_{n-2}^2(\la |1j).
\end{align*}
Similarly, we get
\be\label{a3.14}
\begin{aligned}
&\sigma_{n-1}^{11}\sigma_{n}^{11,jj}-\sigma_{n}^{11}\sigma_{n-1}^{11,jj}=\sigma_{n-2}^2(\la
|1j).
\end{aligned}
\ee
Therefore,
\be\label{a3.15}
\begin{aligned}
&4\sigma_{n-1}^{11}\sigma_{n-1}^{jj}\sigma_{n}^{11,jj}-2\sigma_{n-1}^{11}\sigma_{n}^{jj}\sigma_{n-1}^{11,jj}
-2\sigma_{n}^{11}\sigma_{n-1}^{jj}\sigma_{n-1}^{11,jj}\\
&=2\sigma^2_{n-2}(\la |1j)\big(\sigma_{n-2}(\la |1)+\sigma_{n-2}(\la
|j)\big).
\end{aligned}
\ee
Moreover, by \eqref{a3.14} we obtain
\begin{eqnarray}\label{ce1.5}
&&\sigma_{n-1}^{11}\sigma_{n-1}^{pp}\sigma_{n}^{11,qq}-\sigma_{n}^{11}\sigma_{n-1}^{11,qq}\sigma_{n-1}^{pp}\\
&=&\sigma_{n-1}^{pp}\lt(\sigma_{n-1}^{11}\sigma_n^{11, qq}-\sigma_n^{11}\sigma_{n-1}^{11, qq}\rt)\nonumber\\
&=&\sigma_{n-2}(\la |p)\la_p^2\sigma_{n-3}^2(\la |1pq)\nonumber\\
&=&[\sigma_{n-1}-\sigma_{n-1}(\la |p)]\la_p\sigma_{n-3}^2(\la |1pq)\nonumber\\
&=&\la_p\sigma_{n-1}\sigma_{n-3}^2(\la
|1pq)-\sigma_{n}\sigma_{n-3}^2(\la |1pq).\nonumber
\end{eqnarray}
Similarly, we have
\begin{eqnarray}\label{ce1.5.1}
\sigma_{n-1}^{11}\sigma_{n-1}^{qq}\sigma_{n}^{11,pp}-\sigma_{n}^{11}\sigma_{n-1}^{11,pp}\sigma_{n-1}^{qq}=&\la_q\sigma_{n-1}\sigma_{n-3}^2(\la
|1pq)-\sigma_{n}\sigma_{n-3}^2(\la |1pq).
\end{eqnarray}
Finally, we compute
\begin{eqnarray}\label{ce1.5.2}
&&\sigma_{n-1}^{11}\sigma_{n}^{11}\sigma_{n-1}^{pp,qq}
-\sigma_{n-1}^{11}\sigma_{n-1}^{11}\sigma_{n}^{pp,qq}\\
&=&\sigma_{n-2}(\la |1)\big(\sigma_{n-1}(\la |1)\sigma_{n-3}(\la |pq)-\sigma_{n-2}(\la |1)\sigma_{n-2}(\la |pq)\big)\nonumber\\
&=&\sigma_{n-2}(\la |1)\big\{\la_p\la_q\sigma_{n-3}(\la
|1pq)\big[\la_1\sigma_{n-4}(\la |1pq)+\sigma_{n-3}(\la
|1pq)\big]\nonumber\\
&&-\big[\la_p\la_q\sigma_{n-4}(\la
|1pq)+(\la_p+\la_q)\sigma_{n-3}(\la |1pq)\big]\la_1\sigma_{n-3}(\la
|1pq)\big\}\nonumber\\
&=&\sigma_{n-2}(\la |1)\big[\la_p\la_q\sigma_{n-3}^2(\la
|1pq)-\la_1(\la_p+\la_q)\sigma_{n-3}^2(\la |1pq)\big]\nonumber\\
&=&\sigma_{n-2}(\la |1)\sigma_{n-1}(\la|1)\sigma_{n-3}(\la
|1pq)-\big(\sigma_{n-1}-\sigma_{n-1}(\la|1)\big)(\la_p+\la_q)\sigma_{n-3}^2(\la
|1pq)\nonumber.
\end{eqnarray}

Thus, we conclude
\begin{eqnarray}\label{ce1.5}
&&\sigma_{n-1}^{11}\sigma_{n-1}^{pp}\sigma_{n}^{11,qq}+\sigma_{n-1}^{11}\sigma_{n-1}^{qq}\sigma_{n}^{11,pp}+\sigma_{n-1}^{11}\sigma_{n}^{11}\sigma_{n-1}^{pp,qq}
-\sigma_{n-1}^{11}\sigma_{n-1}^{11}\sigma_{n}^{pp,qq}\\
&&-\sigma_{n}^{11}\sigma_{n-1}^{11,pp}\sigma_{n-1}^{qq}-\sigma_{n}^{11}\sigma_{n-1}^{11,qq}\sigma_{n-1}^{pp}\nonumber\\
&=&\sigma_{n-1}(\la |1)\sigma_{n-3}(\la |1pq)\big(\sigma_{n-2}(\la
|1)+(\la_p+\la_q-2\la_1)\sigma_{n-3}(\la |1pq)\big)\nonumber.
\end{eqnarray}

Substituting equation \eqref{a3.15} and \eqref{ce1.5} into
\eqref{ce1.4} we obtain,
\begin{eqnarray}\label{3.15}
&&\sigma_{n-1}^{11}L_1\nonumber\\
&=&\dsum_{j\neq 1}\sigma_{n-2}^2(\la|1j)\Big[2\sigma_{n-2}(\la|1)+2\sigma_{n-2}(\la|j)\Big] h_{jj1}^2\nonumber\\
&&+\dsum_{p\neq q,p,q\neq 1}\sigma_{n-1}(\la|1)\sigma_{n-3}(\la
|1pq)\Big[\sigma_{n-2}(\la|1)+(\la_p+\la_q-2\la_1)\sigma_{n-3}(\la
|1pq)\Big]h_{pp1}h_{qq1}\nonumber.
\end{eqnarray}
This completes the proof of Lemma \ref{celem1.1}.
\end{proof}

Now notice that if we can prove the $(n-1)\times(n-1)$ matrix $R=(r_{pq}),$ where
$$r_{pq}=\left\{\begin{matrix}\sigma_{n-2}^2(\la|1p)[2\sigma_{n-2}(\la|1)+2\sigma_{n-2}(\la|p)]& \text{ for } p=q\\
\sigma_{n-1}(\la|1)\sigma_{n-3}(\la
|1pq)[\sigma_{n-2}(\la|1)+(\la_p+\la_q-2\la_1)\sigma_{n-3}(\la
|1pq)]& \text{ for }p\neq q\end{matrix}\right.$$
is semi-positive definite, then we will be done.

Observe that $R$ can be written as the Hadamard product of matrix $T=(t_{pq})$ and $S=(s_{pq}),$ where
$$t_{pq}=\left\{\begin{matrix}\sigma_{n-2}^2(\la | 1p)& \text{ for } p=q\\
\sigma_{n-1}(\la |1)\sigma_{n-3}(\la |1pq)& \text{ for }p\neq q\end{matrix}\right.$$
and
$$s_{pq}=\left\{\begin{matrix}2\sigma_{n-2}(\la|1)+2\sigma_{n-2}(\la|p)&\text{ for } p=q \\
\sigma_{n-2}(\la|1)+(\la_p+\la_q-2\la_1)\sigma_{n-3}(\la |1pq)&\text{ for } p \neq q. \end{matrix}\right.$$
Since when $p\neq q$ we have,
$$\sigma_{n-1}(\la |1)\sigma_{n-3}(\la |1pq)=\sigma_{n-2}(\la |1p)\sigma_{n-2}(\la |1q).$$ This implies the nonnegativity
of the quadratic form that matrix $T$ corresponding to,
\begin{eqnarray}
&&\sum_{j\neq 1}\sigma^2_{n-2}(\la |1j)h_{jj1}^2+\sum_{p\neq q, p,q\neq 1}\sigma_{n-1}(\la |1)\sigma_{n-3}(\la |1pq)h_{pp1}h_{qq1}\\
&=&\left[\sum_{j\neq 1}\sigma_{n-2}(\la |1j)h_{jj1}\right]^2\geq
0\nonumber.
\end{eqnarray}
Thus, the matrix $T$ is a semi-positive definite matrix.
By the Schur product Theorem, if we can prove the matrix $S$ is also a semi-positive definite matrix, then we would obtain $R$ is semi-positive definite.
This would complete the proof of Theorem \ref{ceth1.1}.

We will devote the next two sections to proving the matrix $S$ is semi-positive. In particular,
in Section 4, we will show $S$ is semi-positive by assuming $\la_1\leq 0;$
while in Section 5, we will prove the case when $\la_1>0.$

\bigskip

\section{ The case when $\la_1\leq 0$}
\label{kp}
Let
\be\label{kp1.1}
\begin{aligned}
Q_S:&=\sum_{j\neq 1}\lt[2\sigma_{n-2}(\la|1)+2\sigma_{n-2}(\la|j)\rt]\xi_j^2\\
&+\sum_{p\neq q, p, q\neq1}\lt[\sigma_{n-2}(\la|1)+(\la_p+\la_q-2\la_1)\sigma_{n-3}(\la|1pq)\rt]\xi_p\xi_q\\
\end{aligned}
\ee
be the quadratic form of $S.$ In this section, we will prove $Q_S\geq 0$ for $\la_1\leq 0.$
\begin{lemm}
\label{kplem1.1}
If $\la_1\leq0$, then we have $Q_S\geq 0.$ Therefore, the matrix $S$ is a semi-positive definite matrix.
\end{lemm}
\begin{proof}
Since $\la\in \Gamma_{n-1}$ and we assumed $\la_1\leq 0,$ it is clear that in this case we have $\la_2,\cdots, \la_n>0$. Thus,
we can define
$$\mu_2=\frac{1}{\la_2},\mu_3=\frac{1}{\la_3},\cdots,\mu_n=\frac{1}{\la_n}.$$

We can rewrite equation \eqref{2.1} as follows:
$$1=\la_1\sigma_{n-2}(\la |1)+\sigma_{n-1}(\la |1).$$
This gives
\be\label{kp1.2}
\la_1=\frac{1-\sigma_{n-1}(\la|1)}{\sigma_{n-2}(\la|1)}.
\ee
Moreover, for any given $j\neq 1$ we have,
$$\sigma_{n-2}(\la |1)=\la_j\sigma_{n-3}(\la |1j)+\sigma_{n-2}(\la |1j),$$
and
$$\sigma_{n-1}(\la| 1)=\la_j\sigma_{n-2}(\la |1j).$$
Therefore, we get
\be\label{kp1.3}
\sigma_{n-3}(\la|1j)=\frac{1}{\la_j}\sigma_{n-2}(\la|1)-\frac{1}{\la_j^2}\sigma_{n-1}(\la|1).
\ee
Applying \eqref{kp1.2} and \eqref{kp1.3} we can derive the following equalities,
\be\label{kp1.4}
\begin{aligned}
&2\sigma_{n-2}(\la|1)+2\sigma_{n-2}(\la|j)\\
=&2\sigma_{n-2}(\la|1)+2\la_1\sigma_{n-3}(\la|1j)+2\sigma_{n-2}(\la|1j)\\
=&2\sigma_{n-2}(\la|1)-2\frac{\sigma_{n-1}(\la|1)}{\sigma_{n-2}(\la|1)}\sigma_{n-3}(\la|1j)+2\sigma_{n-2}(\la|1j)+\dfrac{2\sigma_{n-3}(\la|1j)}{\sigma_{n-2}(\la|1)}\\
=&2\sigma_{n-2}(\la|1)-2\frac{\sigma_{n-1}(\la|1)}{\sigma_{n-2}(\la|1)}\left(\frac{1}{\la_j}\sigma_{n-2}(\la |1)-\frac{1}{\la_j^2}\sigma_{n-1}(\la |1)\right)+2\sigma_{n-2}(\la|1j)+\dfrac{2\sigma_{n-3}(\la|1j)}{\sigma_{n-2}(\la|1)}\\
=&2\sigma_{n-2}(\la|1)+2\frac{\sigma_{n-1}^2(\la|1)}{\la_j^2\sigma_{n-2}(\la|1)}-2\frac{\sigma_{n-1}(\la |1)}{\la_j}+2\sigma_{n-2}(\la |1j)+\dfrac{2\sigma_{n-3}(\la|1j)}{\sigma_{n-2}(\la|1)}\\
=&2\frac{\sigma_1(\mu)}{\sigma_{n-1}(\mu)}+2\frac{\mu_j^2}{\sigma_{1}(\mu)\sigma_{n-1}(\mu)}+\dfrac{2\sigma_{n-3}(\la|1j)}{\sigma_{n-2}(\la|1)}\\
=&\frac{2\sigma_{1}^2(\mu)+2\mu_j^2}{\sigma_1(\mu)\sigma_{n-1}(\mu)}+\dfrac{2\sigma_{n-3}(\la|1j)}{\sigma_{n-2}(\la|1)},
\end{aligned}
\ee
and
\be\label{kp1.5}
\begin{aligned}
&\sigma_{n-2}(\la|1)+(\la_p+\la_q-2\la_1)\sigma_{n-3}(\la|1pq)\\
=&\sigma_{n-2}(\la|1)+\left(\la_p+\la_q+2\frac{\sigma_{n-1}(\la|1)}{\sigma_{n-2}(\la|1)}\right)\sigma_{n-3}(\la|1pq)-\dfrac{2\sigma_{n-3}(\la|1pq)}{\sigma_{n-2}(\la|1)}\\
=&\frac{\sigma_{1}(\mu)}{\sigma_{n-1}(\mu)}+\left(\frac{1}{\mu_p}+\frac{1}{\mu_q}+2\frac{1}{\sigma_1(\mu)}\right)\frac{\sigma_{n-1}(\la |1)}{\la_p\la_q}-\dfrac{2\sigma_{n-3}(\la|1pq)}{\sigma_{n-2}(\la|1)}\\
=&\frac{\sigma_{1}(\mu)}{\sigma_{n-1}(\mu)}+\frac{(\mu_p+\mu_q)\sigma_1(\mu)+2\mu_p\mu_q}{\mu_p\mu_q\sigma_1(\mu)}\frac{1}{\la_p\la_q\sigma_{n-1}(\mu)}-\dfrac{2\sigma_{n-3}(\la|1pq)}{\sigma_{n-2}(\la|1)}\\
=&\frac{\sigma_1^2(\mu)+(\mu_p+\mu_q)\sigma_1(\mu)+2\mu_p\mu_q}{\sigma_1(\mu)\sigma_{n-1}(\mu)}-\dfrac{2\sigma_{n-3}(\la|1pq)}{\sigma_{n-2}(\la|1)}.
\end{aligned}
\ee
Now, let's consider the quadratic form $Q_S$ that is corresponding to the matrix $S$,
\begin{eqnarray}
Q_S&=&\dsum_{j\neq 1}\Big[2\sigma_{n-2}(\la|1)+2\sigma_{n-2}(\la|j)\Big] \xi_j^2\nonumber\\
&&+\dsum_{p\neq
q, p,q\neq1 }\Big[\sigma_{n-2}(\la|1)+(\la_p+\la_q-2\la_1)\sigma_{n-3}(\la |1pq)\Big]\xi_{p}\xi_{q}\nonumber.
\end{eqnarray}
By equations \eqref{kp1.4}, \eqref{kp1.5}, and the Lemma 10 in \cite{RW}, we
obtain
\begin{eqnarray}\label{4.2}
&&\sigma_1(\mu)\sigma_{n-1}(\mu)Q_S\\
&\geq &\dsum_{j\neq 1}\Big[2\sigma_{1}^2(\mu)+2\mu_j^2\Big] \xi_j^2+\dsum_{p\neq q,p,q\neq 1}\Big[\sigma_1^2(\mu)+(\mu_p+\mu_q)\sigma_1(\mu)+2\mu_p\mu_q\Big]\xi_{p}\xi_{q}\nonumber\\
&=&\sigma_1^2(\mu)\dsum_{j\neq 1}\xi_j^2+\sigma_1^2(\mu)\left(\sum_{j\neq 1}\xi_j\right)^2+2\left(\sum_{j\neq 1}\mu_j\xi_j\right)^2+2\sigma_1(\mu)\sum_{p\neq 1}\mu_p\xi_p\sum_{q\neq
1,p}\xi_q,\nonumber
\end{eqnarray}
where we used Lemma 10 in \cite{RW} to show
$$\sum_{j\neq 1}\sigma_{n-3}(\la|1j)\xi_j^2-\sum_{p,q\neq 1}\sigma_{n-3}(\la|1pq)\xi_p\xi_q\geq 0.$$
It is easy to see that
\be\label{kp1.6}
\begin{aligned}
\dsum_{p\neq 1}\mu_p\xi_p\dsum_{q\neq 1,
p}\xi_q=\dsum_{p\neq 1}\mu_p\xi_p\dsum_{q\neq 1}\xi_q-\dsum_{p\neq 1}\mu_p\xi_p^2.
\end{aligned}
\ee
Moreover, we have
\be\label{kp1.7}
\begin{aligned}
&\sigma_1^2(\mu)\left(\dsum_{j\neq 1}\xi_j\right)^2+\left(\dsum_{j\neq1}\mu_j\xi_j\right)^2+2\sigma_1(\mu)\dsum_{p\neq 1}\mu_p\xi_p\dsum_{q\neq 1}\xi_q\\
=&\left(\sigma_1(\mu)\dsum_{j\neq 1}\xi_j+\dsum_{j\neq
1}\mu_j\xi_j\right)^2.
\end{aligned}
\ee
Combining \eqref{kp1.6} and \eqref{kp1.7} with \eqref{4.2}, we get
\begin{eqnarray}\label{4.3}
&&\sigma_1(\mu)\sigma_{n-1}(\mu)Q_S\\
&\geq&\sigma_1^2(\mu)\dsum_{j\neq 1}\xi_j^2+\left(\dsum_{j\neq 1}\mu_j\xi_j\right)^2-2\sigma_1(\mu)\dsum_{j\neq 1}\mu_j\xi_j^2+\left(\sigma_1(\mu)\dsum_{j\neq 1}\xi_j+\dsum_{j\neq 1}\mu_j\xi_j\right)^2\nonumber\\
&\geq&\sigma_1^2(\mu)\dsum_{j\neq 1}\xi_j^2+\left(\dsum_{j\neq 1}\mu_j\xi_j\right)^2-2\sigma_1(\mu)\dsum_{j\neq 1}\mu_j\xi_j^2\nonumber\\
&=&\dsum_{j\neq 1}\sigma^2_1(\mu)\xi_j^2-2\dsum_{j\neq1}\sigma_1(\mu)\xi_j\mu_j\xi_j+\dsum_{j\neq1}\mu_j^2\xi_j^2+\dsum_{p\neq q,p,q\neq 1}\mu_p\xi_p\mu_q\xi_q\nonumber\\
&=&\dsum_{j\neq 1}\left(\sigma_1(\mu)-\mu_j\right)^2\xi_j^2+\dsum_{p\neq q,p,q\neq 1}\mu_p\xi_p\mu_q\xi_q\nonumber\\
&\geq&\dsum_{j\neq 1}\left(\dsum_{s\neq 1,j}\mu_s^2\right)\xi_j^2+\dsum_{p\neq q,p,q\neq 1}\mu_p\xi_p\mu_q\xi_q\nonumber\\
&=&\sum_{p\neq q,p,q\neq 1}\mu_p^2\xi_q^2+\sum_{p\neq q,p,q\neq 1}\mu_p\xi_q\mu_q\xi_p\nonumber\\
&=&\frac{1}{2}\sum_{p\neq q,p,q\neq 1}\left(\mu_p\xi_q+\mu_q\xi_p\right)^2\nonumber\\
&\geq&0.\nonumber
\end{eqnarray}
Since $\sigma_1(\mu)\sigma_{n-1}(\mu)$ is positive, we have $Q_S\geq 0$. This completes the proof of Lemma \ref{kplem1.1}.
\end{proof}

\section{The case when $\la_1> 0$}
\label{ku}
In this section, we will prove the following Lemma and complete the proof of Theorem \ref{ceth1.1}.
\begin{lemm}
\label{kulem1.1}
If $\la_1>0$, then for any $1\leq m\leq n-1,$ the sum of all $m$-th principal minors of matrix $S$ is nonnegative. Therefore, the matrix $S$ is a semi-positive definite matrix.
\end{lemm}

Before starting the proof of Lemma \ref{kulem1.1}, we want to recall an important Lemma from \cite{RW}. We will use this Lemma many times throughout this section.
\begin{lemm}(Lemma 9 in \cite{RW})
\label{kulem1.2}
Suppose $2\leq i_1<i_2<\cdots<i_m\leq n$ are $m$ ordered indices. Let $D_m(i_1\cdots i_m)$ denote the
$m$-th principal minor of the matrix $(c_{pq})_{2\leq p, q\leq n},$ where
\[c_{pq}=\left\{\begin{matrix}\sigma_{n-3}(\la|1p)&\text{ for } p=q\\ -\sigma_{n-3}(\la|1pq)&\text{ for } p\neq q.\end{matrix}\right.\]
Then we have
\be\label{ku1.1}
\begin{aligned}
D_m(i_1\cdots i_m)&=\det\left[\ju{cccc}{c_{i_1i_1}&c_{i_1i_2}&\cdots&c_{i_1i_m}\\
c_{i_2i_1}&a_{i_2i_2}&\cdots&c_{i_2i_m}\\
\vdots&\vdots&\ddots&\vdots\\
c_{i_mi_1}&c_{i_mi_2}&\cdots&c_{i_mi_m}}\right]\\
&=\sigma_{n-2}^{m-1}(\la|1)\sigma_{n-(m+2)}(\la|1i_1\cdots i_m).
\end{aligned}
\ee
Moreover, after deleting the $l$-th row and $k$-th column, where $l\neq k,$ we get,
\be\label{ku1.2}
\begin{aligned}
B_{m-1}=&\det\left[\ju{ccccccc}{c_{i_1i_1}&c_{i_1i_2}&\cdots&c_{i_2i_{k-1}}&c_{i_2i_{k+1}}&\cdots&c_{i_2i_m}\\
c_{i_2i_1}&b_{i_2i_2}&\cdots&c_{i_2i_{k-1}}&c_{i_2i_{k+1}}&\cdots&c_{i_2i_m}\\
\vdots&\vdots&\ddots&\vdots&\vdots&\ddots&\vdots\\
c_{i_{l-1}i_1}&c_{i_{l-1}i_2}&\cdots&c_{i_{l-1}i_{k-1}}&c_{i_{l-1}i_{k+1}}&\cdots&c_{i_{l-1}i_m}\\
c_{i_{l+1}i_1}&c_{i_{l+1}i_2}&\cdots&c_{i_{l+1}i_{k-1}}&c_{i_{l+1}i_{k+1}}&\cdots&c_{i_{l+1}i_m}\\\vdots&\vdots&\ddots&\vdots&\vdots&\ddots&\vdots\\
c_{i_mi_1}&c_{i_mi_2}&\cdots&c_{i_mi_{k-1}}&c_{i_mi_{k+1}}&\cdots&c_{i_mi_m}}\right]\\
&=(-1)^{l+k}\sigma_{n-2}^{m-2}(\la|1)\sigma_{n-m-1}(\la|1i_1\cdots i_m).
\end{aligned}
\ee
\end{lemm}

Next, in order to simplify our calculations, we will decompose $S$ into three $(n-1)\times (n-1)$ matrices: $A=(a_{pq})_{2\leq p, q\leq n},$
$B=(b_{pq})_{2\leq p, q\leq n},$ and $C=\sigma_{n-2}(\la|1)Id_{n-1}.$ Here
$$
a_{pq}=\left\{\begin{matrix}2\sigma_{n-2}(\la|1p)+\sigma_{n-2}(\la|1)&\text{ for } p=q\\
(\la_p+\la_q)\sigma_{n-3}(\la|1pq)+\sigma_{n-2}(\la|1)&\text{ for } p\neq q\end{matrix}\right.,
$$

$$
b_{pq}=\left\{\begin{matrix}2\la_1\sigma_{n-3}(\la|1p)&\text{ for } p=q\\ -2\la_1\sigma_{n-3}(\la|1pq)&\text{ for } p\neq q\end{matrix}\right.,
$$
and $Id_{n-1}$ is the $(n-1)\times (n-1)$ identity matrix.
By the equality $$\sigma_{n-2}(\la |p)=\la_1\sigma_{n-3}(\la |1p)+\sigma_{n-2}(\la |1p),$$ we can see that
\be\label{ku1.3}
S=A+B+\sigma_{n-2}(\la|1)Id_{n-1}.
\ee

One of the key reasons that the above decomposition \eqref{ku1.3} can simplify our calculation is the following.
\begin{lemm}\label{lem2} The rank of the matrix $A$ is at most two.
\end{lemm}
\begin{proof}
For any $2\leq j\leq n$, let
$$m_j=\sigma_{n-2}(\la |1j)+\frac{\sigma_{n-2}(\la |1)}{2}.$$ Then, it is easy to see that $a_{pq}=m_p+m_q$.
Therefore, we have
$$A=\left[\begin{matrix}m_2&1\\ m_3&1\\ \cdots  \\ m_{n}&1 \end{matrix}\right]\left[\begin{matrix}1&1&\cdots&1\\ m_2&m_3&\cdots&m_n\end{matrix}\right].$$
Hence, the rank of $A$ is at most two.
\end{proof}

For our convenience, we will introduce some notations. We will denote the set of multiple indices
$$\mathcal{I}_k=\{(i_2,\cdots,i_k)|2\leq i_2<\cdots<i_k\leq n\},$$
and we will use $I_k$ to denote an element in $\mathcal{I}_k$. Moreover, if $I_k=(i_2,i_3,\cdots,i_k),$ then $|I_k|$ denotes the set $\{i_2,i_3,\cdots,i_k\}$. For example, we have $I_n=(2,\cdots,n)\in \mathcal{I}_n$ and $|I_n|=\{2, 3, \cdots, n\}.$
We also need the following definition.

\begin{defi} Suppose $A=(a_{pq}), B=(b_{pq})$ are two $(n-1)\times(n-1)$ matrices and $I_k=(i_2,i_3\cdots,i_k)\in\mathcal{I}_k$ is a multiple index. We define the following principal minors $D_A(I_k), D_B(I_k)$ of matrices $A$ and $B$ respectively:
\begin{align*}
D_A(I_k)=\det\left[\ju{cccc}{a_{i_2i_2}&a_{i_2i_3}&\cdots&a_{i_2i_k}\\
a_{i_3i_2}&a_{i_3i_3}&\cdots&a_{i_3i_k}\\
\vdots&\vdots&\ddots&\vdots\\
a_{i_ki_2}&a_{i_ki_3}&\cdots&a_{i_ki_k}}\right], \quad
D_B(I_k)=\det\left[\ju{cccc}{b_{i_2i_2}&b_{i_2i_3}&\cdots&b_{i_2i_k}\\
b_{i_3i_2}&b_{i_3i_3}&\cdots&b_{i_3i_k}\\
\vdots&\vdots&\ddots&\vdots\\
b_{i_ki_2}&b_{i_ki_3}&\cdots&b_{i_ki_k}}\right].
\end{align*}
For indices $i_l,i_p,i_q\in |I_k|$ and $i_p<i_q$, we also define the following ``mixed'' principal minors $D_{B,A}(I_k ;i_l)$, $D_{B,A}(I_k; i_pi_q)$ of $A,B$:
\begin{align*}
D_{B,A}(I_k;i_l)=&\det\left[\ju{cccccc}{b_{i_2i_2}&b_{i_2i_3}&\cdots&b_{i_2i_l}&\cdots&b_{i_2i_k}\\
b_{i_3i_2}&b_{i_3i_3}&\cdots&b_{i_3i_l}&\cdots&b_{i_3i_k}\\
\vdots&\vdots&\ddots&\vdots&\ddots&\vdots\\
a_{i_li_2}&a_{i_li_3}&\cdots&a_{i_li_l}&\cdots&a_{i_li_k}\\
\vdots&\vdots&\ddots&\vdots&\ddots&\vdots\\
b_{i_ki_2}&b_{i_ki_3}&\cdots&b_{i_ki_l}&\cdots&b_{i_ki_k}}\right].
\end{align*}
and
\begin{align*}
D_{B,A}(I_k; i_pi_q)&=\det\left[\ju{cccccccc}{b_{i_2i_2}&b_{i_2i_3}&\cdots&b_{i_2i_p}&\cdots &b_{i_2i_q}&\cdots&b_{i_2i_k}\\
b_{i_3i_2}&b_{i_3i_3}&\cdots&b_{i_3i_p}&\cdots&b_{i_3i_q}&\cdots&b_{i_3i_k}\\
\vdots&\vdots&\ddots&\vdots&\ddots&\vdots&\ddots&\vdots\\
a_{i_pi_2}&a_{i_pi_3}&\cdots&a_{i_pi_p}&\cdots&a_{i_pi_q}&\cdots&a_{i_pi_k}\\
\vdots&\vdots&\ddots&\vdots&\ddots&\vdots&\ddots&\vdots\\
a_{i_qi_2}&a_{i_qi_3}&\cdots&a_{i_qi_p}&\cdots&a_{i_qi_q}&\cdots&a_{i_qi_k}\\
\vdots&\vdots&\ddots&\vdots&\ddots&\vdots&\ddots&\vdots\\
b_{i_ki_2}&b_{i_ki_3}&\cdots&b_{i_ki_p}&\cdots&b_{i_ki_q}&\cdots&b_{i_ki_k}}\right].
\end{align*}
\end{defi}

Finally, we are ready to prove Lemma \ref{kulem1.1}. In the following, we will start with computing the principal minors of the matrix $A+B$. Then, we will compute the sum of all $m$-th principal minors of $S$ for any $1\leq m\leq n-1$.

By Lemma \ref{lem2} we know that the rank of $A$ is at most two. Thus, for $k\geq 4$ and any multiple index
$I_k=(i_2,\cdots,i_k)\in \mathcal{I}_k$, we have
\begin{eqnarray}\label{5.1}
D_{A+B}(I_k)=D_B(I_k)+\sum_{i_l\in |I_k|} D_{B,A}(I_k;i_l)+\sum_{i_p,i_q\in |I_k|,i_p<i_q}D_{B,A}(I_k;i_pi_q).
\end{eqnarray}
For a given $I_k=(i_2,\cdots,i_k)\in \mathcal{I}_k$ and any integer $2\leq s\leq k$, let's denote
$$\mathcal{J}_s(I_k)=\{(j_2,j_3,\cdots,j_s)| j_2<j_3<\cdots<j_s,\,\,\mbox{where}\,\,j_l\in |I_k| \text{ for } l=2,3,\cdots,s\}.$$
Then by a straightforward calculation we get,
\be\label{ku1.4}
\begin{aligned}
&\sum_{I_k\in\mathcal{I}_k}D_S(I_k)\\
=&\sum_{I_k\in\mathcal{I}_k}D_{A+B}(I_k)+\sum_{I_k\in\mathcal{I}_k}\dsum_{J_{k-1}\in \mathcal{J}_{k-1}(I_k)}
D_{A+B}(J_{k-1})\sigma_{n-2}^{1}(\la|1)\\
&+\sum_{I_k\in\mathcal{I}_k}\dsum_{J_{k-2}\in \mathcal{J}_{k-2}(I_k)}
D_{A+B}(J_{k-2})\sigma_{n-2}^{2}(\la|1)+\cdots+\sum_{I_k\in\mathcal{I}_k}\dsum_{J_{s}\in \mathcal{J}_{s}(I_k)}
D_{A+B}(J_{s})\sigma_{n-2}^{k-s}(\la|1)\\
&+\cdots+\sum_{I_k\in\mathcal{I}_k}\dsum_{J_{2}\in \mathcal{J}_{2}(I_k)}
D_{A+B}(J_{2})\sigma_{n-2}^{k-2}(\la|1)+\sum_{I_k\in\mathcal{I}_k}\sigma_{n-2}^{k-1}(\la|1)
\end{aligned}
\ee

From Lemma \ref{lemm4} to Lemma \ref{lemm10}, we will compute terms involved in
equations \eqref{5.1} and \eqref{ku1.4}.

\begin{lemm}\label{lemm4} For a given $I_k=(i_2,\cdots, i_k)\in \mathcal{I}_k$, we have
$$
D_B(I_k)=(2\la_1)^{k-1}\sigma_{n-2}^{k-2}(\la|1)\sigma_{n-k-1}(\la|1i_2\cdots
i_k).
$$
\end{lemm}
\begin{proof}
By our definition of matrix $B,$ we can see that for any $2\leq p, q\leq n,$
\[\frac{b_{pq}}{2\la_1}=c_{pq}.\]
Thus the result follows from Lemma \ref{kulem1.2} directly.
\end{proof}

In next Lemma, we will compute the summation of principal minors of matrix $B.$
\begin{lemm}\label{lemm5}
For any integer $2\leq s\leq k$,  we have
\begin{align*}
\dsum_{I_k\in \mathcal{I}_k}\dsum_{J_s\in \mathcal{J}_s(I_k)}D_B(J_s)
=&\dfrac{s(n-s)!}{(k-s)!(n-k)!}\left(2\la_1\right)^{s-1}\sigma_{n-2}^{s-2}(\la|1)\sigma_{n-(s+1)}(\la|1).
\end{align*}

\end{lemm}

\begin{proof} Applying Lemma \ref{lemm4}, we get
\begin{align*}
&\dsum_{I_k\in \mathcal{I}_k}\dsum_{J_s\in \mathcal{J}_s(I_k)}D_B(J_s)\\
=&\dsum_{I_k\in \mathcal{I}_k}\dsum_{J_s\in \mathcal{J}_s(I_k)}(2\la_1)^{s-1}\sigma_{n-2}^{s-2}(\la|1)\sigma_{n-(s+1)}(\la|1 j_2\cdots j_s)\\
=&\dfrac{(n-s)C_{k-1}^{s-1}C_{n-1}^{k-1}}{C_{n-1}^{n-(s+1)}}(2\la_1)^{s-1}\sigma_{n-2}^{s-2}(\la|1)\sigma_{n-(s+1)}(\la|1)\\
=&\dfrac{s(n-s)!}{(k-s)!(n-k)!}(2\la_1)^{s-1}\sigma_{n-2}^{s-2}(\la|1)\sigma_{n-(s+1)}(\la|1).
\end{align*}

\end{proof}

\par
Before we continue with the calculation of principal minors, we need the following Lemma.
\begin{lemm}\label{B}
For any ordered indices $2\leq i_2<i_3<\cdots<i_k\leq n$, we have
\begin{eqnarray}
&&\sigma_{n-1}(\la|1)\sigma_{n-k-1}(\la|1i_2\cdots
i_k)+\sigma_{n-k}(\la|1i_2\cdots
i_k)\dsum_{l=2}^k\sigma_{n-2}(\la|1i_l)\\
&=&\sigma_{n-2}(\la|1)\sigma_{n-k}(\la|1i_2\cdots
i_k).\nonumber
\end{eqnarray}
\end{lemm}
\par
\begin{proof} An inductive calculation shows,
\begin{align*}
&\sigma_{n-k}(\la|1i_2\cdots i_k)\dsum_{s=2}^{k}\sigma_{n-2}(\la|1i_s)+\sigma_{n-k-1}(\la|1i_2\cdots i_k)\sigma_{n-1}(\la|1)\\
=&\sigma_{n-k}(\la|1i_2\cdots i_k)\dsum_{s=2}^{k}\sigma_{n-2}(\la|1i_s)+\la_{i_k}\sigma_{n-k-1}(\la|1i_2\cdots i_k)\sigma_{n-2}(\la|1i_k)\\
=&\sigma_{n-k}(\la|1i_2\cdots
i_k)\dsum_{s=2}^{k}\sigma_{n-2}(\la|1i_s)+\sigma_{n-k}(\la|1i_2\cdots
i_{k-1})\sigma_{n-2}(\la|1i_k)\\
&-\sigma_{n-k}(\la|1i_2\cdots i_k)\sigma_{n-2}(\la|1i_k)\\
=&\sigma_{n-k}(\la|1i_2\cdots i_k)\dsum_{s=2}^{k-1}\sigma_{n-2}(\la|1i_s)+\sigma_{n-k}(\la|1i_2\cdots i_{k-1})\sigma_{n-2}(\la|1i_k)\\
=&\la_{i_k}\sigma_{n-k}(\la|1i_2\cdots i_k)\dsum_{s=2}^{k-1}\sigma_{n-3}(\la|1i_si_k)+\la_{i_{k-1}}\sigma_{n-k}(\la|1i_2\cdots i_{k-1})\sigma_{n-3}(\la|1i_{k-1}i_k)\\
=&\sigma_{n-k+1}(\la|1i_2\cdots i_{k-1})\dsum_{s=2}^{k-1}\sigma_{n-3}(\la|1i_si_k)+\sigma_{n-k+1}(\la|1i_2\cdots i_{k-2})\sigma_{n-3}(\la|1i_{k-1}i_k)\\
&-\sigma_{n-k+1}(\la|1i_2\cdots i_{k-1})\sigma_{n-3}(\la|1i_{k-1}i_k)
\end{align*}
\begin{align*}
=&\sigma_{n-k+1}(\la|1i_2\cdots i_{k-1})\dsum_{s=2}^{k-2}\sigma_{n-3}(\la|1i_si_k)+\sigma_{n-k+1}(\la|1i_2\cdots i_{k-2})\sigma_{n-3}(\la|1i_{k-1}i_k)\\
=&\cdots\\
=&\sigma_{n-2}(\la|1)\sigma_{n-k}(\la|1i_2\cdots i_k).
\end{align*}

\end{proof}

Now, let's get back to compute some more complicated principal minors.
\begin{lemm} \label{lemm7}For any multiple index $I_k=(i_2,i_3,\cdots,i_k)\in\mathcal{I}_k$, we have
\begin{align*}
&\dsum_{i_l\in |I_k|}D_{B,A}(I_k;i_l)\\
=&(2\la_1)^{k-2}\sigma_{n-2}^{k-2}(\la|1)\left[k(k-1)\sigma_{n-k}(\la|1i_2\cdots i_k) +\dsum_{i_s\in |I_k|}\sigma_{n-k}(\la|1i_2\cdots\hat{i}_s\cdots
i_k)\right].
\end{align*}
Here $\hat{i}_s$ means that the index $i_s$ does not appear.
\end{lemm}

\begin{proof} Given a index $i_l\in |I_k|$, for computing the ``mixed'' principal minor $D_{B,A}(I_k; i_l)$, we expand the determinant according to its $i_l$-th row,
\begin{eqnarray}\label{defi}
D_{B,A}(I_k;i_l)
&=&a_{i_li_l}M_l+\sum_{s\neq l,s=2}^k(-1)^{s-1+l-1}a_{i_li_s}M_s.
\end{eqnarray}
Here the minors $M_s$ for $s=2,3\cdots,k$ are defined by
\begin{align*}
M_s=&\det\left[\ju{ccccccc}{b_{i_2i_2}&b_{i_2i_3}&\cdots&b_{i_2i_{s-1}}&b_{i_2i_{s+1}}&\cdots&b_{i_2i_k}\\
b_{i_3i_2}&b_{i_3i_3}&\cdots&b_{i_3i_{s-1}}&b_{i_3i_{s+1}}&\cdots&b_{i_3i_k}\\
\vdots&\vdots&\ddots&\vdots&\vdots&\ddots&\vdots\\
b_{i_{l-1}i_2}&b_{i_{l-1}i_3}&\cdots&b_{i_{l-1}i_{s-1}}&b_{i_{l-1}i_{s+1}}&\cdots&b_{i_{l-1}i_k}\\
b_{i_{l+1}i_2}&b_{i_{l+1}i_3}&\cdots&b_{i_{l+1}i_{s-1}}&b_{i_{l+1}i_{s+1}}&\cdots&b_{i_{l+1}i_k}\\\vdots&\vdots&\ddots&\vdots&\vdots&\ddots&\vdots\\
b_{i_ki_2}&b_{i_ki_3}&\cdots&b_{i_ki_{s-1}}&b_{i_ki_{s+1}}&\cdots&b_{i_ki_k}}\right].
\end{align*}
By Lemma \ref{kulem1.2} we have,
\be\label{ku1.5}
M_l=(2\la_1)^{k-2}\sigma_{n-2}^{k-3}(\la|1)\sigma_{n-k}(\la|1i_2\cdots\hat{i}_l \cdots i_k);
\ee
and when $s\neq l$
\be\label{ku1.6}
M_s=(-1)^{l-1+s-1}(2\la_1)^{k-2}\sigma_{n-2}^{k-3}(\la|1)\sigma_{n-k}(\la|1i_2\cdots i_k).
\ee

Combing \eqref{ku1.5} and \eqref{ku1.6} with \eqref{defi} we obtain,
\begin{align}\label{5.3}
&\dsum_{i_l\in |I_k|}D_{B,A}(I_k;i_l)\\
=&(2\la_1)^{k-2}\sigma_{n-2}^{k-3}(\la|1)\dsum_{l=2}^k\left(a_{i_li_l}\sigma_{n-k}(\la|1i_2\cdots\hat{i}_l \cdots i_k)+\sigma_{n-k}(\la|1i_2\cdots i_k)\dsum_{s=2,s\neq
l}^ka_{i_li_s}\right).\nonumber
\end{align}
A straightforward calculation gives
\be\label{ku1.7}
\sigma_{n-k}(\la|1i_2\cdots\hat{i}_l \cdots
i_k)=\la_{i_l}\sigma_{n-k-1}(\la|1i_2\cdots
i_k)+\sigma_{n-k}(\la|1i_2\cdots i_k).
\ee
Using the definition of $a_{pq}$, we get
\begin{align}\label{5.4}
&a_{i_li_l}\sigma_{n-k}(\la|1i_2\cdots\hat{i}_l \cdots i_k)
+\sigma_{n-k}(\la|1i_2\cdots i_k)\dsum_{s=2,s\neq l}^ka_{i_li_s} \\
=&a_{i_li_l}\la_{i_l}\sigma_{n-k-1}(\la|1i_2\cdots i_k)+\sigma_{n-k}(\la|1i_2\cdots i_k)\dsum_{s=2}^ka_{i_li_s}\nonumber\\
=&\Big(2\sigma_{n-2}(\la |1i_l)+\sigma_{n-2}(\la |1)\Big)\la_{i_l}\sigma_{n-k-1}(\la |1i_2\cdots i_k)\nonumber\\
&+\sigma_{n-k}(\la|1i_2\cdots i_k)\sum_{s=2}^k\Big((\la_{i_l}+\la_{i_s})\sigma_{n-3}(\la|1i_li_s)+\sigma_{n-2}(\la|1)\Big)\nonumber\\
=&\Big(2\sigma_{n-1}(\la|1)+\la_{i_l}\sigma_{n-2}(\la|1)\Big)\sigma_{n-k-1}(\la|1i_2\cdots i_k)\nonumber\\
&+(k-1)\sigma_{n-2}(\la|1)\sigma_{n-k}(\la|1i_2\cdots i_k)\nonumber\\
&+\sigma_{n-k}(\la|1i_2\cdots
i_k)\dsum_{s=2}^k\Big[\sigma_{n-2}(\la|1i_s)+\sigma_{n-2}(\la|1i_l)\Big].\nonumber
\end{align}
Since
\begin{eqnarray}\label{5.5}
\dsum_{l=2}^k\dsum_{s=2}^k\Big[\sigma_{n-2}(\la|1i_s)+\sigma_{n-2}(\la|1i_l)\Big]=2(k-1)\dsum_{s=2}^k\sigma_{n-2}(\la|1i_s),
\end{eqnarray}
equations \eqref{5.3}, \eqref{5.4}, and \eqref{5.5} yield
\begin{align}
&\dsum_{i_l\in |I_k|}D_{B,A}(I_k;i_l)\\
=&(2\la_1)^{k-2}\sigma_{n-2}^{k-3}(\la|1)\left\{(k-1)^2\sigma_{n-2}(\la|1)\sigma_{n-k}(\la|1i_2\cdots i_k)\right. \nonumber\\
&+\left[2(k-1)\sigma_{n-1}(\la|1)+\sigma_{n-2}(\la|1)\dsum_{l=2}^k\la_{i_l}\right]\sigma_{n-k-1}(\la|1i_2\cdots i_k)\nonumber\\
&\left.+2(k-1)\sigma_{n-k}(\la|1i_2\cdots
i_k)\dsum_{s=2}^k\sigma_{n-2}(\la|1i_s)\right\}.\nonumber
\end{align}
Using Lemma \ref{B} we obtain,
\begin{align*}
&\dsum_{i_l\in |I_k|}D_{B,A}(I_k;i_l)\\
=&(2\la_{1})^{k-2}\sigma_{n-2}^{k-3}(\la|1)\Big\{(k-1)(k+1)\sigma_{n-2}(\la|1)\sigma_{n-k}(\la|1i_2\cdots i_k) \\
&\left.+\sigma_{n-2}(\la|1)\dsum_{l=2}^k\la_{i_l}\sigma_{n-k-1}(\la|1i_2\cdots i_k)\right\}\\
=&(2\la_1)^{k-2}\sigma_{n-2}^{k-2}(\la|1)\left[k(k-1)\sigma_{n-k}(\la|1i_2\cdots i_k)+\dsum_{s=2}^k\sigma_{n-k}(\la|1i_2\cdots\hat{i}_s\cdots i_k)\right].
\end{align*}
\end{proof}

\begin{lemm}\label{lemm8}
For any integer  $2\leq s\leq k$, we have
\begin{align*}
&\sum_{I_k\in \mathcal{I}_k}\dsum_{J_s\in
\mathcal{J}_s(I_k)}\sum_{l=2}^sD_{B,A}(J_s;j_l)\\
=&\frac{(n+1)(s-1)(n-s)!}{(k-s)!(n-k)!}(2\la_1)^{s-2}\sigma_{n-2}^{s-2}(\la|1)\sigma_{n-s}(\la|1).
\end{align*}

\end{lemm}

\begin{proof}
By Lemma \ref{lemm7}, we have

\begin{align*}
&\dsum_{I_k\in \mathcal{I}_k}\sum_{J_s\in
\mathcal{J}_s(I_k)}\sum_{l=2}^sD_{B,A}(J_s;j_l)\\
=&\dsum_{I_k\in \mathcal{I}_k}\dsum_{J_s\in\mathcal{J}_s
(I_k)}(2\la_1)^{s-2}\sigma_{n-2}^{s-2}(\la|1)\\
&\times \left[s(s-1)\sigma_{n-s}(\la|1j_2\cdots j_s)+\dsum_{l=2}^s\sigma_{n-s}(\la|1j_2\cdots\hat{j}_l\cdots
j_s)\right]\\
=&\frac{(n+1)(s-1)(n-s)!}{(k-s)!(n-k)!}(2\la_1)^{s-2}\sigma_{n-2}^{s-2}(\la|1)\sigma_{n-s}(\la|1).
\end{align*}

\end{proof}

In Lemma \ref{lemm9} and \ref{lemm10}, we are going to compute the ``mixed'' principal minors of type $D_{B, A}(I_k; i_pi_q).$
\begin{lemm}\label{lemm9}
For the multiple index $I_k=(i_2,i_3,\cdots,i_k)\in \mathcal{I}_k$ and $k\geq 3$, we have
\begin{align*}
&D_{B,A}(I_k;i_{p}i_q)\\
=&-(2\la_1)^{k-3}\sigma_{n-2}^{k-3}(\la|1)\left(\la_{i_p}-\la_{i_{q}}\right)^2\sigma_{n-3}(\la|1i_{p}i_q)\sigma_{n-k}(\la|1i_2\cdots
i_k),\nonumber
\end{align*}
where $i_p,i_q\in |I_k|$ and $i_p<i_q$.
\end{lemm}

\begin{proof}
For any $2\leq s<t\leq k$, let's denote
\begin{align*}
&\det M_{st}=\\
&\det\left[\ju{cccccccccc}{b_{i_2i_2}&b_{i_2i_3}&\cdots&b_{i_2i_{s-1}}&b_{i_2i_{s+1}}&\cdots&b_{i_2i_{t-1}}&b_{i_2i_{t+1}}&\cdots&b_{i_2i_k}\\
b_{i_3i_2}&b_{i_3i_3}&\cdots&b_{i_3i_{s-1}}&b_{i_3i_{s+1}}&\cdots&b_{i_3i_{t-1}}&b_{i_3i_{t+1}}&\cdots&b_{i_3i_k}\\
\vdots&\vdots&\ddots&\vdots&\vdots&\ddots&\vdots&\vdots&\ddots&\vdots\\
b_{i_{p-1}i_2}&b_{i_{p-1}i_3}&\cdots&b_{i_{p-1}i_{s-1}}&b_{i_{p-1}i_{s+1}}&\cdots&b_{i_{p-1}i_{t-1}}&b_{i_{p-1}i_{t+1}}&\cdots&b_{i_{p-1}i_k}\\
b_{i_{p+1}i_2}&b_{i_{p+1}i_3}&\cdots&b_{i_{p+1}i_{s-1}}&b_{i_{p+1}i_{s+1}}&\cdots&b_{i_{p+1}i_{t-1}}&b_{i_{p+1}i_{t+1}}&\cdots&b_{i_{p+1}i_k}\\
\vdots&\vdots&\ddots&\vdots&\vdots&\ddots&\vdots&\vdots&\ddots&\vdots\\
b_{i_{q-1}i_2}&b_{i_{q-1}i_3}&\cdots&b_{i_{q-1}i_{s-1}}&b_{i_{q-1}i_{s+1}}&\cdots&b_{i_{q-1}i_{t-1}}&b_{i_{q-1}i_{t+1}}&\cdots&b_{i_{q-1}i_k}\\
b_{i_{q+1}i_2}&b_{i_{q+1}i_3}&\cdots&b_{i_{q+1}i_{s-1}}&b_{i_{q+1}i_{s+1}}&\cdots&b_{i_{q+1}i_{t-1}}&b_{i_{q+1}i_{t+1}}&\cdots&b_{i_{q+1}i_k}\\
\vdots&\vdots&\ddots&\vdots&\vdots&\ddots&\vdots&\vdots&\ddots&\vdots\\
b_{i_ki_2}&b_{i_ki_3}&\cdots&b_{i_ki_{s-1}}&b_{i_ki_{s+1}}&\cdots&b_{i_ki_{t-1}}&b_{i_ki_{t+1}}&\cdots&b_{i_ki_k}}\right].
\end{align*}
Then we have
\begin{align*}
&D_{B,A}(I_k;i_pi_q)\\
=&\sum_{s=2}^{k}(-1)^{q-1+s-1}a_{i_qi_s}\left(\sum_{t< s,2\leq t\leq k}(-1)^{p-1+t-1}a_{i_pi_t}\det M_{ts}\right.\\
&+\left.\sum_{t>s,2\leq t\leq k}(-1)^{p-1+t-2}a_{i_pi_t}\det M_{st}\right).
\end{align*}
Therefore, in order to calculate $D_{B,A}(I_k;i_pi_q),$ we need to figure out the value of $\det M_{st}$ first.

In the following, we will calculate $\det M_{st}$ for different values of $s,t$.

\noindent (1) If $t<p$, we can see that the $(s-1)$-th row and $(t-1)$-th row of $M_{st}$ are
\be\label{5.8}
\begin{aligned}
&(b_{i_si_2},b_{i_si_3},\cdots,b_{i_si_{s-1}},b_{i_si_{s+1}},\cdots,b_{i_si_{t-1}},b_{i_si_{t+1}},\cdots,b_{i_si_k})\\
=&-2\la_1\Big(\sigma_{n-3}(\la|1i_si_2),\sigma_{n-3}(\la|1i_si_3),\cdots,\sigma_{n-3}(\la |1i_s,i_{s-1}),\sigma_{n-3}(\la |1i_si_{s+1}),\cdots,\\
&\sigma_{n-3}(\la|1i_si_{t-1}),\sigma_{n-3}(\la|1i_si_{t+1}),\cdots,\sigma_{n-3}(\la |1i_si_k)\Big),\\
\end{aligned}
\ee
and
\be\label{5.9}
\begin{aligned}
&(b_{i_ti_2},b_{i_ti_3},\cdots,b_{i_ti_{s-1}},b_{i_ti_{s+1}},\cdots,b_{i_ti_{t-1}},b_{i_ti_{t+1}},\cdots,b_{i_ti_k})\\
=&-2\la_1\Big(\sigma_{n-3}(\la|1i_ti_2),\sigma_{n-3}(\la|1i_ti_3),\cdots,\sigma_{n-3}(\la |1i_ti_{s-1}),\sigma_{n-3}(\la |1i_ti_{s+1}),\cdots,\\
&\sigma_{n-3}(\la|1i_ti_{t-1}),\sigma_{n-3}(\la|1i_ti_{t+1}),\cdots,\sigma_{n-3}(\la |1i_ti_k)\Big).
\end{aligned}
\ee
We note that the vector in \eqref{5.8} multiplying by $\la_s$ is equal to the vector in \eqref{5.9} multiplying by $\la_t$.
Thus, the $(s-1)$-th row and the $(t-1)$-th row of $M_{st}$ are linearly dependent, which implies $\det M_{st}=0$.

\noindent (2) If $t=p$, by Lemma \ref{kulem1.2} we have
$$
\det M_{st}=(-1)^{q-2+s-1}(2\la_1)^{k-3}\sigma_{n-2}^{k-4}(\la|1)\sigma_{n-k+1}(\la|1i_2\cdots \hat{i}_p \cdots i_k)\nonumber.
$$

\noindent (3) If $p<t<q$ and $s\neq p$, similar to the case (1), we have $\det M_{st}=0$.

\noindent (4) If $p<t<q$ and $s=p$, similar to case (2), we have
$$
\det M_{st}=(-1)^{q-2+t-2}(2\la_1)^{k-3}\sigma_{n-2}^{k-4}(\la|1)\sigma_{n-k+1}(\la|1i_2\cdots \hat{i}_p \cdots i_k)\nonumber.
$$

\noindent (5) If $t=q$ and $s\neq p$, by Lemma \ref{kulem1.2}, we have
$$
\det M_{st}=(-1)^{p-1+s-1}(2\la_1)^{k-3}\sigma_{n-2}^{k-4}(\la|1)\sigma_{n-k+1}(\la|1i_2\cdots \hat{i}_q \cdots i_k)\nonumber.
$$

\noindent (6) If $t=q$ and $s=p$, by Lemma \ref{kulem1.2}, we have
$$
\det M_{st}=(2\la_1)^{k-3}\sigma_{n-2}^{k-4}(\la|1)\sigma_{n-k+1}(\la|1i_2\cdots\hat{i}_p\cdots \hat{i}_q \cdots i_k),
$$

\noindent (7) If $t>q$ and $s\neq p$ or $q$, similar to the case (1), we have $\det M_{st}=0$.

\noindent (8) If $t>q$ and $s=p$, similar to the case (2), we have
$$
\det M_{st}=(-1)^{q-2+t-2}(2\la_1)^{k-3}\sigma_{n-2}^{k-4}(\la|1)\sigma_{n-k+1}(\la|1i_2\cdots \hat{i}_p \cdots i_k)\nonumber.
$$

\noindent (9) If $t>q$ and $s=q$, similar to the case (5), we have
$$
\det M_{st}=(-1)^{p-1+t-2}(2\la_1)^{k-3}\sigma_{n-2}^{k-4}(\la|1)\sigma_{n-k+1}(\la|1i_2\cdots \hat{i}_q \cdots i_k)\nonumber.
$$

In view of the above calculation, if $\{s,t\}\cap\{i_p, i_q\}=\emptyset,$ we have $\det M_{st}=0$. Therefore, the expansion of $D_{B,A}(I_k; i_pi_q)$ becomes

\begin{eqnarray}\label{5.10}
\\
&&D_{B,A}(I_k;i_pi_q)\nonumber\\
&=&\sum_{s=2}^{k}(-1)^{q-1+s-1}a_{i_qi_s}\left(\sum_{t< s,2\leq t\leq k}(-1)^{p-1+t-1}a_{i_pi_t}\det M_{ts}\right.\nonumber\\
&&\ \ \ \ \ \ \ \ \ \ \ \ \ \ \ \ \ \ \ \ \ \ \ \ \ \ +\left.\sum_{t>s,2\leq t\leq k}(-1)^{p-1+t-2}a_{i_pi_t}\det M_{st}\right)\nonumber\\
&=&(-1)^{p+q}\sum_{s=2}^{k}\left(\sum_{t< s}(-1)^{s+t}a_{i_qi_s}a_{i_pi_t}M_{ts}-\sum_{t>s}(-1)^{s+t}a_{i_qi_s}a_{i_pi_t}\det M_{st}\right)\nonumber\\
&=&(-1)^{p+q}\sum_{s< t, 2\leq s,t\leq k}(-1)^{s+t}\left(a_{i_pi_s}a_{i_qi_t}-a_{i_qi_s}a_{i_pi_t}\right)\det M_{st}\nonumber\\
&=&-\sum_{s< t=p}\left(a_{i_pi_s}a_{i_qi_p}-a_{i_qi_s}a_{i_pi_p}\right)(2\la_1)^{k-3}\sigma_{n-2}^{k-4}(\la|1)\sigma_{n-k+1}(\la|1i_2\cdots \hat{i}_p\cdots i_k)\nonumber\\
&&+\sum_{s\neq p,s<t=q}\left(a_{i_pi_s}a_{i_qi_q}-a_{i_qi_s}a_{i_pi_q}\right)(2\la_1)^{k-3}\sigma_{n-2}^{k-4}(\la|1)\sigma_{n-k+1}(\la|1i_2\cdots \hat{i}_q \cdots i_k)\nonumber\\
&&+\sum_{s=p,t=q}\left(a_{i_pi_p}a_{i_qi_q}-a_{i_qi_p}a_{i_pi_q}\right)(2\la_1)^{k-3}\sigma_{n-2}^{k-4}(\la|1)\sigma_{n-k+1}(\la|1i_2\cdots\hat{i}_p\cdots \hat{i}_q \cdots i_k)\nonumber\\
&&+\sum_{s=p< t,t\neq q}\left(a_{i_pi_p}a_{i_qi_t}-a_{i_qi_p}a_{i_pi_t}\right)(2\la_1)^{k-3}\sigma_{n-2}^{k-4}(\la|1)\sigma_{n-k+1}(\la|1i_2\cdots \hat{i}_p \cdots i_k)\nonumber\\
&&-\sum_{s=q,t>q}\left(a_{i_pi_q}a_{i_qi_t}-a_{i_qi_q}a_{i_pi_t}\right)(2\la_1)^{k-3}\sigma_{n-2}^{k-4}(\la|1)\sigma_{n-k+1}(\la|1i_2\cdots \hat{i}_q \cdots i_k)\nonumber\\
&=&(2\la_1)^{k-3}\sigma_{n-2}^{k-4}(\la|1)\Big\{\sum_{s\neq p,q}\left(a_{i_qi_s}a_{i_pi_p}-a_{i_pi_s}a_{i_qi_p}\right)\sigma_{n-k+1}(\la|1i_2\cdots \hat{i}_p\cdots i_k)\nonumber\\
&&+\sum_{s\neq p,q}\left(a_{i_pi_s}a_{i_qi_q}-a_{i_qi_s}a_{i_pi_q}\right)\sigma_{n-k+1}(\la|1i_2\cdots \hat{i}_q \cdots i_k)\nonumber\\
&&+\left(a_{i_pi_p}a_{i_qi_q}-a_{i_qi_p}a_{i_pi_q}\right)\sigma_{n-k+1}(\la|1i_2\cdots\hat{i}_p\cdots \hat{i}_q \cdots i_k)\Big\}\nonumber.
\end{eqnarray}
Using the definition of $a_{pq}$, for $s\neq p,q$, we have
\be\label{ku1.8}
\begin{aligned}
&a_{i_qi_s}a_{i_pi_p}-a_{i_pi_s}a_{i_qi_p}\\
=&\Big(\sigma_{n-2}(\la|1i_q)+\sigma_{n-2}(\la|1i_s)+\sigma_{n-2}(\la|1)\Big)\Big(2\sigma_{n-2}(\la|1i_p)+\sigma_{n-2}(\la|1)\Big)\\
&-\Big(\sigma_{n-2}(\la|1i_p)+\sigma_{n-2}(\la|1i_s)+\sigma_{n-2}(\la|1)\Big)\\
&\times\Big(\sigma_{n-2}(\la|1i_p)+\sigma_{n-2}(\la|1i_q)+\sigma_{n-2}(\la|1)\Big)\\
=&\Big[\sigma_{n-2}(\la|1i_q)-\sigma_{n-2}(\la|1i_p)\Big]\Big[\sigma_{n-2}(\la|1i_p)-\sigma_{n-2}(\la|1i_s)\Big].
\end{aligned}
\ee
Similarly we can compute,
\be\label{ku1.9}
\begin{aligned}
&a_{i_pi_s}a_{i_qi_q}-a_{i_qi_s}a_{i_pi_q}\\
=&\Big[\sigma_{n-2}(\la|1i_p)-\sigma_{n-2}(\la|1i_q)\Big]\Big[\sigma_{n-2}(\la|1i_q)-\sigma_{n-2}(\la|1i_s)\Big].
\end{aligned}
\ee
Combining equation \eqref{ku1.8} and \eqref{ku1.9} we have,
\begin{eqnarray}\label{5.11}
&&\la_{i_p}\Big[a_{i_qi_s}a_{i_pi_p}-a_{i_pi_s}a_{i_qi_p}\Big]+\la_{i_q}\Big[a_{i_pi_s}a_{i_qi_q}-a_{i_qi_s}a_{i_pi_q}\Big]\\
&=&\sigma_{n-2}(\la|1i_s)(\la_{i_q}-\la_{i_p})\Big[\sigma_{n-2}(\la|1i_q)-\sigma_{n-2}(\la|1i_p)\Big]\nonumber\\
&=&-\sigma_{n-2}(\la|1i_s)(\la_{i_q}-\la_{i_p})^2\sigma_{n-3}(\la|1i_pi_q).\nonumber
\end{eqnarray}
For the case when $s$ is equal to $p$ or $q$ we have
\begin{eqnarray}\label{5.12}
&&a_{i_pi_p}a_{i_qi_q}-a_{i_qi_p}a_{i_pi_q}\\
&=&\Big(2\sigma_{n-2}(\la|1i_p)+\sigma_{n-2}(\la|1)\Big)\Big(2\sigma_{n-2}(\la|1i_q)+\sigma_{n-2}(\la|1)\Big)\nonumber\\
&&-\Big(\sigma_{n-2}(\la|1i_p)+\sigma_{n-2}(\la|1i_q)+\sigma_{n-2}(\la|1)\Big)^2\nonumber\\
&=&-\Big(\sigma_{n-2}(\la|1i_p)-\sigma_{n-2}(\la|1i_q)\Big)^2\nonumber\\
&=&-(\la_{i_q}-\la_{i_p})^2\sigma_{n-3}^2(\la|1i_pi_q)\nonumber.
\end{eqnarray}
Therefore, by \eqref{5.10}, \eqref{5.11} and \eqref{5.12}, we obtain
\begin{align}
&D_{B,A}(I_k;i_pi_q)\nonumber\\
=&(2\la_1)^{k-3}\sigma_{n-2}^{k-4}(\la|1)\Big\{\sigma_{n-k}(\la|1i_2\cdots i_k)\sum_{s\neq p,q}\Big[\la_{i_p}\left(a_{i_qi_s}a_{i_pi_p}-a_{i_pi_s}a_{i_qi_p}\right)\nonumber\\&+\la_{i_q}\left(a_{i_pi_s}a_{i_qi_q}-a_{i_qi_s}a_{i_pi_q}\right)\Big]\nonumber\\
&+\left(a_{i_pi_p}a_{i_qi_q}-a_{i_qi_p}a_{i_pi_q}\right)\left(\la_{i_p}\la_{i_q}\sigma_{n-k-1}(\la|1i_2\cdots i_k)+(\la_{i_p}+\la_{i_q})\sigma_{n-k}(\la|1i_2\cdots i_k)\right)\Big\}\nonumber\\
=&-(2\la_1)^{k-3}\sigma_{n-2}^{k-4}(\la|1)(\la_{i_q}-\la_{i_p})^2\sigma_{n-3}(\la|1i_pi_q)\Big\{\sigma_{n-k}(\la|1i_2\cdots i_k)\dsum_{s\neq p,q}\sigma_{n-2}(\la|1i_s)\nonumber\\
&+\sigma_{n-k-1}(\la|1i_2\cdots i_k)\sigma_{n-1}(\la|1)+\left(\sigma_{n-2}(\la| 1i_p)+\sigma_{n-2}(\la|1i_q)\right)\sigma_{n-k}(\la|1i_2\cdots i_k))\Big\}\nonumber\\
=&-(2\la_1)^{k-3}\sigma_{n-2}^{k-3}(\la|1)(\la_{i_q}-\la_{i_p})^2\sigma_{n-3}(\la|1i_pi_q)\sigma_{n-k}(\la|1i_2\cdots
i_k). \nonumber
\end{align}
Here in the last step, we used  Lemma \ref{B}.
\end{proof}

\begin{lemm}\label{lemm10} For the multiple index $I_k=(i_2,i_3,\cdots,i_k)\in \mathcal{I}_k, \,\,k\geq 3$ and
any integer $3\leq s\leq k$, we have
\begin{align*}
&\dsum_{I_k\in \mathcal{I}_k}\dsum_{J_s\in \mathcal{J}_s(I_k)}\sum_{j_p,j_q\in |J_s|, j_p< j_q}D_{B,A}(J_s;j_pj_q)\\
=&\dfrac{(n-1)(s-1)(n-s)!}{(k-s)!(n-k)!}(2\la_1)^{s-3}\sigma_{n-2}^{s-3}(\la|1)\sigma_{n-1}(\la|1)\sigma_{n-s}(\la|1)\\
&-\dfrac{(n-s+1)!}{(k-s)!(n-k)!}(2\la_1)^{s-3}\sigma_{n-2}^{s-2}(\la|1)\sigma_{n-s+1}(\la|1).
\end{align*}
\end{lemm}

\begin{proof}
By a straightforward calculation we get
\be\label{ku1.10}
\begin{aligned}
&\sum_{j_p,j_q\in |J_s|, j_p< j_q}(\la_{j_p}-\la_{j_q})^2\sigma_{n-3}(\la|1j_pj_q)\\
=&\sum_{j_p,j_q\in |J_s|, j_p< j_q}(\la_{j_p}^2+\la_{j_q}^2-2\la_{j_p}\la_{j_q})\sigma_{n-3}(\la|1j_pj_q)\\
=&\dsum_{j_p\in |J_s|}\la_{j_p}\dsum_{j_q\in
|J_s|, j_q\neq j_p}\sigma_{n-2}(\la|1j_q)-\sum_{j_q\neq j_p}\sigma_{n-1}(\la|1)\\
=&\dsum_{j_p\in |J_s|}\la_{j_p}\dsum_{j_q\in
|J_s|}\sigma_{n-2}(\la|1j_q)-(s-1)^2\sigma_{n-1}(\la|1).
\end{aligned}
\ee
Moreover, we have
\be\label{ku1.11}
\begin{aligned}
&\sigma_{n-s}(\la|1j_2\cdots j_s)\dsum_{j_p\in |J_s|}\la_{j_p}\dsum_{j_q\in |J_s|}\sigma_{n-2}(\la|1j_q)\\
=&\dsum_{j_p\in
|J_s|}\sigma_{n-s+1}(\la|1j_2\cdots\hat{j}_p\cdots j_s)\dsum_{j_q\in
|J_s|}\sigma_{n-2}(\la|1j_q)\\
=&\dsum_{j_p\in |J_s|}\sigma_{n-s+1}(\la|1j_2\cdots\hat{j}_p\cdots
j_s)\Big(\sigma_{n-2}(\la|1)-\dsum_{j_q\in
|I_n|\backslash |J_s|}\sigma_{n-2}(\la|1j_q)\Big)\\
=&\sigma_{n-2}(\la|1)\dsum_{j_p\in
|J_s|}\sigma_{n-s+1}(\la|1j_2\cdots\hat{j}_p\cdots j_s)\\
&-\dsum_{j_p\in |J_s|}\dsum_{j_q\in |I_n|\backslash
|J_s|}\sigma_{n-s}(\la|1j_2\cdots\hat{j}_p\cdots j_s j_q)\sigma_{n-1}(\la|1).
\end{aligned}
\ee

Thus, using equation \eqref{ku1.10} and \eqref{ku1.11} we get
\be\label{ku1.12}
\begin{aligned}
&\dsum_{I_k\in \mathcal{I}_k}\dsum_{J_s\in\mathcal{J}_s(I_k)}\sum_{j_p,j_q\in |J_s|, j_p<
j_q}\sigma_{n-s}(\la|1j_2\cdots j_s)(\la_{j_p}-\la_{j_q})^2\sigma_{n-3}(\la|1j_pj_q)\\
=&\dsum_{I_k\in \mathcal{I}_k}\dsum_{J_s\in\mathcal{J}_s
(I_k)}\sigma_{n-2}(\la|1)\dsum_{j_p\in
|J_s|}\sigma_{n-s+1}(\la|1j_2\cdots\hat{j}_p\cdots j_s)\\
&-\dsum_{I_k\in \mathcal{I}_k}\dsum_{J_s\in \mathcal{J}_s(I_k)}\dsum_{j_p\in
|J_s|}\dsum_{j_q\in |I_n|\backslash
|J_s|}\sigma_{n-s}(\la|1j_2\cdots\hat{j}_p\cdots j_s j_q)\sigma_{n-1}(\la|1)\\
&-\dsum_{I_k\in\mathcal{ I}_k}\dsum_{J_s\in
\mathcal{J}_s(I_k)}(s-1)^2\sigma_{n-1}(\la|1)\sigma_{n-s}(\la|1j_2\cdots j_s)\\
=&\dfrac{(s-1)C_{k-1}^{s-1}C_{n-1}^{k-1}}{C_{n-1}^{n-s+1}}\sigma_{n-2}(\la|1)\sigma_{n-s+1}(\la|1)\\
&-\dfrac{(n-s)(s-1)C_{k-1}^{s-1}C_{n-1}^{k-1}}{C_{n-1}^{n-s}}\sigma_{n-1}(\la|1)\sigma_{n-s}(\la|1)\\
&-\dfrac{(s-1)^2C_{k-1}^{s-1}C_{n-1}^{k-1}}{C_{n-1}^{n-s}}\sigma_{n-1}(\la|1)\sigma_{n-s}(\la|1)\\
=&\dfrac{(n-s+1)!}{(k-s)!(n-k)!}\sigma_{n-2}(\la|1)\sigma_{n-s+1}(\la|1)\\
&-\dfrac{(n-1)(s-1)(n-s)!}{(k-s)!(n-k)!}\sigma_{n-1}(\la|1)\sigma_{n-s}(\la|1).
\end{aligned}
\ee
Lemma \ref{lemm10} follows from equation \eqref{ku1.12} and Lemma \ref{lemm9} directly.
\end{proof}

Now, let's come back to the matrix $S$ and prove Lemma \ref{kulem1.1}.
\begin{proof}(Proof of Lemma \ref{kulem1.1})
By basic Linear Algebra we know, given any multiple index $I_k=(i_2,i_3,\cdots,i_k)\in \mathcal{I}_k$ we have,
\begin{eqnarray}\label{5.13}
&&D_S(I_k)\\
&=&D_{A+B}(I_k)+\dsum_{J_{k-1}\in \mathcal{J}_{k-1}(I_k)}
D_{A+B}(J_{k-1})\sigma_{n-2}^{1}(\la|1)\nonumber\\
&&+\dsum_{J_{k-2}\in \mathcal{J}_{k-2}(I_k)}
D_{A+B}(J_{k-2})\sigma_{n-2}^{2}(\la|1)+\cdots+\dsum_{J_{s}\in \mathcal{J}_{s}(I_k)}
D_{A+B}(J_{s})\sigma_{n-2}^{k-s}(\la|1)\nonumber\\
&&+\cdots+\dsum_{J_{2}\in \mathcal{J}_{2}(I_k)}
D_{A+B}(J_{2})\sigma_{n-2}^{k-2}(\la|1)+\sigma_{n-2}^{k-1}(\la|1).\nonumber
\end{eqnarray}
By \eqref{5.1}, for any multiple index $J_s=(j_2,\cdots,j_s)\in\mathcal{J}_s(I_k)$, we have
\begin{align*}
D_{A+B}(J_s)=D_B(J_s)+\sum_{l=2}^{
s}D_{B,A}(J_s;j_l)+\sum_{j_2\leq j_p< j_q\leq
j_s}D_{B,A}(J_s;j_pj_q).
\end{align*}
Thus, for $3\leq s\leq k$, using Lemma \ref{lemm5}, Lemma \ref{lemm8}, and Lemma \ref{lemm10} we obtain,
\begin{eqnarray}\label{5.14}
&&\dsum_{I_k\in \mathcal{I}_k}\dsum_{J_s\in \mathcal{J}_s(I_k)}D_{A+B}(J_s)\\
&=&\dsum_{I_k\in \mathcal{I}_k}\dsum_{J_s\in \mathcal{J}_s(I_k)}D_B(J_s)+\dsum_{I_k\in
\mathcal{I}_k}\dsum_{J_s\in\mathcal{J}_s(I_k)}\sum_{j_l\in |J_s|}D_{
B,A}(J_s;j_l)\nonumber\\
&&+\dsum_{I_k\in \mathcal{I}_k}\dsum_{J_s\in\mathcal{J}_s(I_k)}\sum_{ j_p< j_q, j_pj_q\in |J_s|}D_{B,A}(J_s;j_pj_q)\nonumber\\
&=&\dfrac{s(n-s)!}{(k-s)!(n-k)!}(2\la_1)^{s-1}\sigma_{n-2}^{s-2}(\la|1)\sigma_{n-s-1}(\la|1)\nonumber\\
&&+\frac{(n+1)(s-1)(n-s)!}{(k-s)!(n-k)!}(2\la_1)^{s-2}\sigma_{n-2}^{s-2}(\la|1)\sigma_{n-s}(\la|1)\nonumber\\
&&+\dfrac{(n-1)(s-1)(n-s)!}{(k-s)!(n-k)!}(2\la_1)^{s-3}\sigma_{n-2}^{s-3}(\la|1)\sigma_{n-1}(\la|1)\sigma_{n-s}(\la|1)\nonumber\\
&&-\dfrac{(n-s+1)!}{(k-s)!(n-k)!}(2\la_1)^{s-3}\sigma_{n-2}^{s-2}(\la|1)\sigma_{n-s+1}(\la|1).\nonumber
\end{eqnarray}
For $s=2$ and $3\leq k\leq n,$ the third term of $D_{A+B}(J_s)$ does not appear. Thus, we have
\begin{eqnarray}\label{5.15}
&&\dsum_{I_k\in \mathcal{I}_k}\dsum_{J_s\in \mathcal{J}_s(I_k)}D_{A+B}(J_s)\\
&=&\dsum_{I_k\in \mathcal{I}_k}\dsum_{J_s\in \mathcal{J}_s(I_k)}D_B(J_s)+\dsum_{I_k\in
\mathcal{I}_k}\dsum_{J_s\in\mathcal{J}_s(I_k)}\sum_{j_l\in |J_s|}D_{
B,A}(J_s;j_l)\nonumber\\
&=&\dfrac{s(n-s)!}{(k-s)!(n-k)!}(2\la_1)^{s-1}\sigma_{n-2}^{s-2}(\la|1)\sigma_{n-s-1}(\la|1)\nonumber\\
&&+\frac{(n+1)(s-1)(n-s)!}{(k-s)!(n-k)!}(2\la_1)^{s-2}\sigma_{n-2}^{s-2}(\la|1)\sigma_{n-s}(\la|1)\nonumber.
\end{eqnarray}

We want to rewrite $\sum_{I_k\in\mathcal{I}_k}D_{S}(I_k)$ as a polynomial of the variable $2\la_1$. Let's calculate the coefficient of $(2\la_1)^i$ for $i=0, 1, \cdots, k-1$.

By equation \eqref{5.14}, the coefficient of $(2\la_1)^{s-3}$ for $3\leq s\leq k$ is
\be\label{ku1.13}
\begin{aligned}
&\dfrac{(s-2)(n-s+2)!}{(k-s+2)!(n-k)!}(2\la_1)^{s-3}\sigma_{n-2}^{s-4}(\la|1)\sigma_{n-s+1}(\la|1)\cdot \sigma_{n-2}^{k-s+2}(\la|1)\\
&+\frac{(n+1)(s-2)(n-s+1)!}{(k-s+1)!(n-k)!}(2\la_1)^{s-3}\sigma_{n-2}^{s-3}(\la|1)\sigma_{n-s+1}(\la|1)\cdot \sigma_{n-2}^{k-s+1}(\la|1)\\
&+\dfrac{(n-1)(s-1)(n-s)!}{(k-s)!(n-k)!}(2\la_1)^{s-3}\sigma_{n-2}^{s-3}(\la|1)\sigma_{n-1}(\la|1)\sigma_{n-s}(\la|1)\cdot \sigma_{n-2}^{k-s}(\la|1)\\
&-\dfrac{(n-s+1)!}{(k-s)!(n-k)!}(2\la_1)^{s-3}\sigma_{n-2}^{s-2}(\la|1)\sigma_{n-s+1}(\la|1)\cdot
\sigma_{n-2}^{k-s}(\la|1)\\
=&(2\la_1)^{s-3}\sigma_{n-2}^{k-3}(\la|1)\Big[P(s-3)\sigma_{n-2}(\la|1)\sigma_{n-s+1}(\la|1)+Q(s-3)\sigma_{n-1}(\la|1)\sigma_{n-s}(\la|1)
\Big],
\end{aligned}
\ee
where the functions $P(s-3)$ and $Q(s-3)$ are defined by
$$
P(s-3)=\dfrac{(s-2)(n-s+2)!}{(k-s+2)!(n-k)!}+\frac{(n+1)(s-2)(n-s+1)!}{(k-s+1)!(n-k)!}-\dfrac{(n-s+1)!}{(k-s)!(n-k)!},
$$
and
$$
Q(s-3)=\dfrac{(n-1)(s-1)(n-s)!}{(k-s)!(n-k)!}.
$$
The coefficient of $(2\la_1)^{k-2}$ is
\be\label{ku1.14}
\begin{aligned}
&(k-1)(n-k+1)(2\la_1)^{k-2}\sigma_{n-2}^{k-3}(\la|1)\sigma_{n-k}(\la|1)\sigma_{n-2}(\la|1)\\
&+(n+1)(k-1)(2\la_1)^{k-2}\sigma_{n-2}^{k-2}(\la|1)\sigma_{n-k}(\la|1)\\
=&(k-1)(2n+2-k)(2\la_1)^{k-2}\sigma_{n-2}^{k-2}(\la|1)\sigma_{n-k}(\la|1).
\end{aligned}
\ee
The coefficient of $(2\la_1)^{k-1}$ is
\be\label{ku1.15}
k(2\la_1)^{k-1}\sigma_{n-2}^{k-2}(\la|1)\sigma_{n-k-1}(\la|1).
\ee
We substitute \eqref{ku1.13}, \eqref{ku1.14}, and \eqref{ku1.15} into \eqref{5.13}, then sum over $I_k\in\mathcal{I}_k, k\geq 3,$ and get,
\be\label{5.16}
\begin{aligned}
&\dsum_{I_k\in
\mathcal{I}_k}D_S(I_k)\\
=&\dsum_{s=0}^{k-3}(2\la_1)^{s}\sigma_{n-2}^{k-3}(\la|1)
\Big[P(s)\sigma_{n-2}(\la|1)\sigma_{n-s-2}(\la|1)+Q(s)\sigma_{n-1}(\la|1)\sigma_{n-s-3}(\la|1)\Big]\\
&+(k-1)(2n+2-k)(2\la_1)^{k-2}\sigma_{n-2}^{k-2}(\la|1)\sigma_{n-k}(\la|1)\\
&+k(2\la_1)^{k-1}\sigma_{n-2}^{k-2}(\la|1)\sigma_{n-k-1}(\la|1).
\end{aligned}
\ee

Since we assume $\la_1>0$, the last two terms are non negative. We only need to analyze the first term.
Note that
\begin{eqnarray}
P(s)&\geq& \frac{(n+1)(s+1)(n-s-2)!}{(k-s-2)!(n-k)!}-\frac{(n-s-2)!}{(k-s-3)!(n-k)!}\nonumber\\
&=&\frac{(n-s-2)!}{(k-s-2)!(n-k)!}[(n+1)(s+1)-k+s+2]\nonumber\\
&\geq&0.\nonumber
\end{eqnarray}
If $\sigma_{n-1}(\la|1)\geq 0,$ then we obtain for $k\geq 3$
$$\dsum_{I_k\in\mathcal{I}_k}D_S(I_k)\geq 0. $$
If $\sigma_{n-1}(\la|1)< 0$, using the identity
\begin{align*}
\la_1\sigma_{n-2}(\la|1)=\sigma_{n-1}-\sigma_{n-1}(\la|1)\geq
-\sigma_{n-1}(\la|1)>0,
\end{align*}
\eqref{5.16} becomes
\begin{eqnarray}\label{5.17}
&&\dsum_{I_k\in
\mathcal{I}_k}D_S(I_k)\\
&=&\sigma_{n-2}^{k-3}(\la|1)\left[\dsum_{s=1}^{k-3}2(2\la_1)^{s-1}P(s)\Big(\la_1\sigma_{n-2}(\la|1)\Big)\sigma_{n-s-2}(\la|1)\right.\nonumber\\
&&\left.+\dsum_{s=0}^{k-3}(2\la_1)^{s}Q(s)\sigma_{n-1}(\la|1)\sigma_{n-s-3}(\la|1)\right]\nonumber\\
&&+2(k-1)(2n+2-k)(2\la_1)^{k-3}\sigma_{n-2}^{k-3}(\la|1)\Big(\la_1\sigma_{n-2}(\la|1)\Big)\sigma_{n-k}(\la|1)\nonumber\\
&&+k(2\la_1)^{k-1}\sigma_{n-2}^{k-2}(\la|1)\sigma_{n-k-1}(\la|1)\nonumber\\
&\geq&\sigma_{n-2}^{k-3}(\la|1)\left[\dsum_{s=1}^{k-3}2(2\la_1)^{s-1}P(s)\Big(-\sigma_{n-1}(\la|1)\Big)\sigma_{n-s-2}(\la|1)\right.\nonumber\\
&&\left.+\dsum_{s=0}^{k-3}(2\la_1)^{s}Q(s)\sigma_{n-1}(\la|1)\sigma_{n-s-3}(\la|1)\right]\nonumber\\
&&+2(k-1)(2n+2-k)(2\la_1)^{k-3}\sigma_{n-2}^{k-3}(\la|1)\Big(-\sigma_{n-1}(\la|1)\Big)\sigma_{n-k}(\la|1)\nonumber\\
&&+k(2\la_1)^{k-1}\sigma_{n-2}^{k-2}(\la|1)\sigma_{n-k-1}(\la|1)\nonumber
\end{eqnarray}
\begin{eqnarray}
&\geq&\sigma_{n-2}^{k-3}(\la|1)\dsum_{s=0}^{k-4}(2\la_1)^{s}\Big(2P(s+1)-Q(s)\Big)\Big(-\sigma_{n-1}(\la|1)\Big)\sigma_{n-s-3}(\la|1)\nonumber\\
&&+\Big(2(k-1)(2n+2-k)-Q(k-3)\Big)(2\la_1)^{k-3}\sigma_{n-2}^{k-3}(\la|1)\nonumber\\
&&\times\Big(-\sigma_{n-1}(\la|1)\Big)\sigma_{n-k}(\la|1)\nonumber.
\end{eqnarray}

It's easy to see that
\begin{eqnarray}\label{5.18}
&&2(k-1)(2n+2-k)-Q(k-3)\\
&=&(k-1)\Big(2(2n+2-k)-(n-1)\Big)\nonumber\\
&\geq& 0.\nonumber
\end{eqnarray}
Moreover, for $0\leq s\leq k-4$, we have
\begin{eqnarray}\label{5.19}
&&2P(s+1)-Q(s)\\
&=&\dfrac{(n-s-3)!(s+2)}{(k-s-3)!(n-k)!}\left[2\left(
\dfrac{n-s-2}{k-s-2}+(n+1)-\frac{k-s-3}{s+2}\right) -(n-1)\right]\nonumber\\
&\geq&\dfrac{(n-s-3)!(s+2)}{(k-s-3)!(n-k)!}\left[2(n+1)-\frac{2(k-s-3)}{s+2} -(n-1)\right]\nonumber\\
&\geq&\dfrac{(n-s-3)!(s+2)}{(k-s-3)!(n-k)!}\Big[(n+1)-(k-s-3)\Big]\nonumber\\
&\geq&0.\nonumber
\end{eqnarray}
Combining \eqref{5.17} with \eqref{5.18} and \eqref{5.19}, we obtain if $\sigma_{n-1}(\la|1)\leq 0$ and $k\geq 3,$
 $$\dsum_{I_k\in\mathcal{I}_k}D_S(I_k)\geq 0.$$
Therefore, we have proved for $2\leq m\leq n-1$ the sum of all $m$-th principal minors of
matrix $S$ is nonnegative. When $m=1,$ by the definition of $S,$ we get $\dsum_{I_2\in\mathcal{I}_2}D_S(I_2)=\sum_{p=2}^n s_{pp}>0$ directly. This completes the proof of Lemma \ref{kulem1.1}.
\end{proof}

Lemma \ref{kplem1.1} and Lemma \ref{kulem1.1} proved that the matrix $S$ is a semi-positive matrix. This together with our analysis in Section
\ref{ce} yields Theorem \ref{ceth1.1}.

\bigskip

\section{Bounded principal curvatures implies convexity}
\label{bp}
In this section, we will study the convexity of the admissible hypersurface $\M_u$ with bounded principal curvatures. More precisely, we will prove that
every spacelike hypersurfaces $\M_u$ that satisfies $\la[\M_u]\in\Gamma_{n-1},$ $\sigma_{n-1}(\la[\M_u])=1,$ and
$|\la[\M_u]|<C$ must be convex.

Following, Cheng-Yau \cite{CY}, we first prove the induced metric on $\M_u$ is complete. Due to our assumption on the principal curvatures, the proof here is much easier than it is in Cheng-Yau \cite{CY}. For readers' convenience, we will include it here.

Recall that the Minkowski distance is defined by
$$2z(X)=\|X\|^2=\lt<X, X\rt>^2=\sum_{i=1}^nx_i^2-x_{n+1}^2,\,\,X\in\R^{n, 1}.$$
Cheng-Yau (see Proposition 1 in \cite{CY}) have shown the following: For a spacelike hypersurface $\M$ in $\R^{n, 1}$ which is closed with respect to the Euclidean topology, if the origin $\mathbf{0}\in \M$, then $z$ is a proper function defined on $\M$. Here being ``proper'' means that for any given constant $c>0$, the set $\{X\in \M\subset \R^{n, 1}| z(X)\leq c\}$ is compact. In general, if $\mathbf{0}\not\in \M,$ without loss of generality, we may assume $P=(0,\xi)\in \M$. Then, we can modify the function $z$ to be
$$2z(X)=\|\tilde{X}\|^2=\left\|X-\xi E \right\|^2, $$ and show the set
$\{X\in \M\subset \R^{n,1}| z(X)\leq c\}$ is compact.
Therefore, in the following, we will always assume $\mathbf{0}\in \M.$
\begin{prop}\label{bppro1.1}
Let $\M\in\R^{n, 1}$ be a spacelike hypersurface with bounded principal curvatures, i.e., $\lt|\la[\M]\rt|\leq C_0$. Then there is a constant $C$ only depending
on $C_0$ such that
\begin{eqnarray}\label{6.1}
|\nabla z(X)|^2\leq C(z(X)+1)^2,\,\,\mbox{for $X\in \M$}.
\end{eqnarray}
\end{prop}
\begin{proof}
In the following, for any $c>0,$ we denote $\M_c:=\{X\in \M| z(X)\leq \dfrac{c}{2}\}$.
Note that by earlier discussion we know that $\M_{2c}$ is compact. Considering an auxiliary function
$$\phi(X)=(c-z)^2\frac{|\nabla z|^2}{(z+1)^2}.$$
It is obvious that $\phi$ achieves its maximum value at some interior point $P_0\in \M_{2c}$.
Let $\{\tau_1, \cdots, \tau_n\}$ be an orthonormal frame at $P_0.$
Now, we differentiate $\log\phi$ at $P_0$ and get,
\begin{eqnarray}
&&2\frac{-z_i}{c-z}+\frac{2\sum_kz_kz_{ki}}{|\nabla z|^2}-2\frac{z_i}{z+1}=0.\label{6.2}
\end{eqnarray}
By a straightforward calculation we have
 \begin{eqnarray}\label{66}
 z_i=\lt<X, \tau_i\rt>,\ \ \  z_{ij}=\delta_{ij}-h_{ij}\lt<X, \nu\rt>.
 \end{eqnarray}
Moreover, since $z\geq 0$, we obtain
\begin{eqnarray}
\label{bp1.0}
\lt<X, \nu\rt>^2\leq \sum_i\lt<X, \tau_i\rt>^2.
\end{eqnarray}
We may choose an orthonormal coordinate at $P_0$ such that
$$z_1=|\nabla z|, \text{ and  } z_i=0 \text{  for } i\neq 1.$$ We may also rotate $\{\tau_2,\cdots,\tau_n\}$ such that
$$h_{ij}=h_{ii}\delta_{ij} \text{ for } i,j\geq 2.$$
Thus, using \eqref{6.2} we get ,
\be\label{bp1.1}
2\frac{-z_1}{c-z}+\frac{2z_1z_{11}}{|\nabla z|^2}-2\frac{z_1}{z+1}=0
\ee
and
\be\label{bp1.2}
\frac{2z_1z_{1i}}{|\nabla z|^2}=0,\,\,\mbox{ for $i\geq 2$}.
\ee
This implies
\begin{eqnarray}\label{6.5}
z_{11}=\frac{|\nabla z|^2}{c-z}+\frac{|\nabla z|^2}{z+1}.
\end{eqnarray}
Without loss of generality we may assume $z_1=|\nabla z|>1$ at $P_0$.
Since we are working on hypersurfaces with bounded curvatures, using \eqref{66}, we have
\begin{eqnarray}
\frac{|\nabla z|^2}{z+1} \leq z_{11}\leq 1+|h_{11}||\lt<X, \nu\rt>|\leq 1+C|\lt<X, \nu\rt>|.\nonumber
\end{eqnarray}
By \eqref{bp1.0} we know $|\lt<X, \nu\rt>|\leq |\nabla z|,$ thus at $P_0$ we have
\begin{eqnarray}\label{6.7}
\frac{|\nabla z|}{z+1} \leq C.
\end{eqnarray}
This yields that
\[(c-z)^2\frac{|\nabla z|^2}{(z+1)^2}\leq c^2C^2.\]
Therefore, on $M_c$ we have
\[|\nabla z|^2\leq 4C^2|z+1|^2.\]
Since $c>0$ is arbitrary, we proved \eqref{6.1}.
\end{proof}

Now by the same argument as in \cite{CY} and \cite{T}, we have
\begin{coro}\label{bpcor1.1}
Let $\M\in\R^{n, 1}$ be a spacelike hypersurface which is closed with respect to the Euclidean topology. Suppose $\M$ has bounded principal curvatures.
Then, $\M$ is complete with respect to the induced metric.
\end{coro}

\begin{rmk} Proposition \ref{bppro1.1} and Corollary \ref{bpcor1.1} give a different proof of the completeness of spacelike hypersurfaces with constant Gauss-Kronecker curvature and bounded principal curvatures (see Proposition 5.2 in \cite{LA}).
\end{rmk}

\begin{lemm}
\label{bplem1.1}
Let $\M$ be an $(n-1)$-convex, spacelike hypersurface with bounded principal curvatures, and $\M$ satisfies equation
\eqref{2.1}. Then $\M$ is convex.
\end{lemm}
\begin{proof}
Recall Theorem \ref{ceth1.1} we have,
\begin{eqnarray}\label{bp1.3}
\sigma_{n-1}^{ij}(\sigma_n)_{ij}&\leq&\sigma_1\sigma_{n-1}\sigma_n-n^2\sigma_n^2.
\end{eqnarray}
Given a point $P\in\M$, we can define the distance function on $\M$
$$r(X)=d(P, X),$$ where $X\in \M$. By Corollary \ref{bpcor1.1} we know that $\M$ is complete. Therefore, for any $a>0,$ let
$\mathcal{B}_a:=\{X\in \M| r(X)<a\}$ be the geodesic ball centered at $P$ with radius $a$, then $\mathcal{B}_a$ is compact.

Now, we define an open subdomain of $\M$
\[\Omega=\{X\in\M| \sigma_n(\la[\M(X)])<0\}.\]
Without loss of generality we assume $\Omega\neq\emptyset,$ otherwise, we would be done.
Considering the  auxiliary function
$$\varphi=-\eta^2(X)\sigma_n(X)$$ on $\Omega,$ where $\eta=a^2-r^2(X)$ is the cutoff function and $\sigma_n(X)=\sigma(\la[\M(X)])$.
It is obvious that the function $\varphi$ achieves its maximum at an interior point $X_0$ in $\Omega\cap \mathcal{B}_a$. Moreover, use the
same argument as \cite{CY2}, we can assume $\eta$ is differentiable near $X_0.$ Now, we choose a local orthonormal frame near $X_0$ such that
at $X_0,$ $h_{ij}=\la_i\delta_{ij}.$ Differentiating
$\log\varphi$ at $X_0$ twice we get,
\begin{eqnarray}
&&\frac{(\sigma_n)_i}{\sigma_n}+2\frac{\eta_i}{\eta}=0;\label{6.8}\\
&&\frac{(\sigma_n)_{ii}}{\sigma_n}-\frac{(\sigma_n)_i^2}{\sigma_n^2}+2\frac{\eta_{ii}}{\eta}-2\frac{\eta_i^2}{\eta^2}\leq 0.\label{6.9}
 \end{eqnarray}
Contracting \eqref{6.9} with $\sigma_{n-1}^{ii}$ and applying \eqref{6.8} yields,
 \begin{eqnarray}
\frac{\sigma_{n-1}^{ii}(\sigma_n)_{ii}}{\sigma_n}\leq -2\frac{\sigma_{n-1}^{ii}\eta_{ii}}{\eta}+6\frac{\sigma_{n-1}^{ii}\eta_i^2}{\eta^2}.\nonumber
 \end{eqnarray}
 Combining with \eqref{bp1.3}, we have
  \begin{eqnarray}
\eta^2\sigma_1\sigma_{n-1}+n^2\varphi&\leq& -2\eta\sigma_{n-1}^{ii}\eta_{ii}+6\sigma_{n-1}^{ii}\eta_i^2\\
&=&4r\eta \sigma^{ii}_{n-1}r_{ii}+(4\eta+24r^2)\sigma_{n-1}^{ii}r_i^2.\nonumber
 \end{eqnarray}
 Since $|\nabla r|=1,$ by our assumption that $\M$ has bounded principal curvatures, we can see in $\mathcal{B}_a$,
 $$(4\eta+24r^2)\sigma_{n-1}^{ii}r_i^2\leq Ca^2,$$
 where the constant $C$ depends on $\la[\M].$
To deal with the term $r_{ii},$ we will use the Hessian comparison theorem. Since the sectional curvature of $\M$ satisfies
$$ R_{ijij}=-h_{ii}h_{jj}\geq -C, $$
we have
$$r_{ii}\leq \frac{n-1}{r}(1+Cr).$$
This implies in $\mathcal{B}_a,$
$$4r\eta \sigma^{ii}_{n-1}r_{ii}\leq Ca^3, $$
where the constant $C$ depends on $\la[\M].$  Thus, we obtain for any $a>0$ large in $\Omega\cap \mathcal{B}_a,$
\be\label{bp1.4}
\varphi(X)\leq Ca^3.
\ee
Now, for any given point $Y\in \Omega\subset \M$, we can take  $a>0$  sufficiently large such that $Y\in\Omega\cap \mathcal{B}_{a/2}$.
Then, by \eqref{bp1.4} we have
$$-\sigma_n(Y)\leq \frac{C}{a}.$$
Let $a$ go to infinity, we obtain
$$\sigma_n(Y)=0.$$
Hence, we conclude that $\Omega$ is an empty set. This proves Lemma \ref{bplem1.1}.
\end{proof}

Now that we have proved the convexity of $\M$, we are in the position to prove Theorem \ref{intth1.1}
of the introduction.
\begin{proof}(proof of Theorem \ref{intth1.1})
In view of the formula \eqref{bp1.3}, we know that for a convex hypersurface $\M$ satisfying $\sigma_{n-1}(\la[\M])=1$,
if there is a degenerate point on $\M,$ i.e., $\sigma_n=0,$ then $\sigma_n\equiv 0$ on $\M.$ We will show in this case $\M=\M^{n-1}\times\R.$

Let $\tau_1$ be the principal direction corresponding
to the minimum principal curvature $\la_1=0$. Then $\tau_1$ is a smooth vector field on $\M.$ Let $\gamma(s)$ be the integral curve of $\tau_1,$
and $\text{Span}\{\tau_1, \cdots, \tau_n\}=T\M.$ Then we have
\[\lt<\bar{\nabla}_{\tau_1}\nu, \tau_i\rt>=0\,\,\mbox{for $1\leq i\leq n$},\]
where $\nu$ is the timelike unite normal of $\M.$ Therefore, $\nu$ is a constant vector along $\gamma(s).$
This implies that $\gamma(s)$ lies in the hyperplane $\mathbb{P}$ that is perpendicular to $\nu.$

Now we can choose a coordinate such that
\[\mathbb{P}=\{x| x_{n+1} =\lt<X, E\rt>=0\}\]
and $-\lt<X, E\rt>\geq 0$ for any $X\in\M,$
where $E=(0, \cdots, 0, 1).$ We claim $\gamma(s)$ is a straight line. If not, we can choose
$p, q\in\gamma(s).$ Since the straight line connects $p, q$ is in the convex hull of $\M,$ we conclude that
the flat region that is enclosed by the straight line connects $p, q$ and $\gamma(s)$ is part of $\M,$
i.e., $\M$ has a flat side. This leads to a contradiction.

Therefore, $\gamma(s)$ is a straight line. By Cheeger-Gromoll splitting
theorem (see Theorem 2 in \cite{CG}), we complete the proof of Theorem \ref{intth1.1}.
\end{proof}

\bigskip

\section{The Gauss map and Legendre transform}
In this section, we will discuss properties of the Gauss map and the Legendre transform. We will use these properties
in later sections.
\label{gg}
\subsection{The Gauss map}
\label{gm}
Let $\M$ be a spacelike hypersurface, $\nu(X)$ be the timelike unit normal vector to $\M$
at $X.$ It's well known that the hyperbolic space $\mathbb{H}^{n}(-1)$ is canonically embedded in $\R^{n, 1}$
as the hypersurface
\[\lt<X, X\rt>=-1,\,\, x_{n+1}>0.\]
By parallel translating to the origin we can regard $\nu(X)$
as a point in $\mathbb{H}^n(-1).$ In this way, we define the Gauss map:
\[G: \M\rightarrow \mathbb{H}^n(-1);\,\, X\mapsto\nu(X).\]
If we take the hyperplane $\mathbb{P}:=\{X=(x_1, \cdots, x_{n}, x_{n+1}) |\, x_{n+1}=1\}$ and consider the projection of
$\mathbb{H}^n(-1)$ from the origin into $\mathbb{P}.$ Then $\mathbb{H}^n(-1)$ is mapped in
a one-to-one fashion onto an open unit ball $B_1:=\{\xi\in\R^n |\, \sum\xi^2_k<1\}.$ The map
$P$ is given by
\[P: \mathbb{H}^n(-1)\rightarrow B_1;\,\,(x_1, \cdots, x_{n+1})\mapsto (\xi_1, \cdots, \xi_n),\]
where $x_{n+1}=\sqrt{1+x_1^2+\cdots+x_n^2},$ $\xi_i=\frac{x_i}{x_{n+1}}.$
We will call the map $P\circ G: \M\rightarrow B_1$ the Gauss map and denote it
by $G$ for the sake of simplicity.

Next, let's consider the support function of $\M.$ We denote
\[v:=\lt<X, \nu\rt>=\frac{1}{\sqrt{1-|Du|^2}}\lt(\sum_ix_i\frac{\partial u}{\partial x_i}-u\rt).\]
Let $\{e_1, \cdots, e_n\}$ be an orthonormal frame on $\mathbb{H}^n.$ We will also denote
$\{\e_1, \cdots, \e_n\}$ the push-forward of $e_i$ by the Gauss map $G.$ Similar to the convex geometry case,
we denote
\[\Lambda_{ij}=v_{ij}-v\delta_{ij}\]
the hyperbolic Hessian. Here $v_{ij}$ denote the covariant derivatives with respect to the hyperbolic metric.

Let $\bar{\nabla}$ be the connection of the ambient space. Then, we have
 $$v_i=\bar{\nabla}_{e^*_i}X\cdot \nu+X\cdot \bar{\nabla}_{e_i}\nu=X\cdot e_i,$$ this implies
 $$X=\sum_iv_ie_i-v\nu.$$ Note that $\lt<\nu, \nu\rt>=-1,$
thus we have,
\begin{eqnarray}
\bar{\nabla}_{e_j^*}X&=&\sum_k(e_j(v_k)e_k+v_k\bar{\nabla}_{e_j}e_k)-v_j\nu-v\bar{\nabla}_{e_j}\nu \\
&=&\sum_k(e_j(v_k)e_k+v_k\nabla_{e_j}e_k+v_k\delta_{kj}\nu)-v_j\nu-ve_j\nonumber\\
&=&\sum_k\Lambda_{kj}e_k\nonumber,\\
g_{ij}&=&\bar{\nabla}_{e^*_i}X\cdot \bar{\nabla}_{e^*_j}X=\sum_k\Lambda_{ik}\Lambda_{kj},\\
h_{ij}&=&\bar{\nabla}_{e^*_i}X\cdot \bar{\nabla}_{e_j}\nu=\Lambda_{ij}.
 \end{eqnarray}
This implies that the eigenvalues of the hyperbolic Hessian are the curvature radius of $\M$. That is,
if the principal curvatures of $\M$ are $(\la_1, \cdots, \la_n),$ then the eigenvalues of the hyperbolic Hessian
are $\lt(\la_1^{-1}, \cdots, \la_n^{-1}\rt).$

Moreover, it is clear that
\be\label{gg1.1}
\bar{\nabla}_{e_j}\bar{\nabla}_{e_i}\nu=\delta_{ij}\nu,
\ee
this yields, for $k=1,2\cdots,n+1$,
\be\label{gg1.2}
\bar{\nabla}_{e_j}\bar{\nabla}_{e_i}x_k=x_k\delta_{ij},
\ee
where $x_k$ is the coordinate function. These properties will be used in Section \ref{es}.

\subsection{Legendre transform}
\label{lt}
Suppose $\M$ is a complete, noncompact, locally stictly convex, spacelike hypersurface.
Then $\M$ is the graph of a convex function
\[x_{n+1}=-\lt<X, E\rt>=u(x_1, \cdots, x_n),\]
where $E=(0, \cdots, 0, 1).$
Introduce the Legendre transform
\[\xi_i=\frac{\T u}{\T x_i},\,\, u^*=\sum x_i\xi_i-u.\]
From the theory of convex bodies we know that
\[\Omega=\lt\{(\xi_1, \cdots, \xi_n)| \xi_i=\frac{\partial u}{\partial x_i}(x), x\in\R^n\rt\}\]
is a convex domain.

In particular, let $u(x)=\sqrt{1+|x|^2},$ $x\in\R^n,$ be a hyperboloid with principal curvatures being equal to $1$, then it's Legendre transform is
$\us(\xi)=-\sqrt{1-|\xi|^2},$ $\xi\in B_1.$

Next, we calculate the first and the second fundamental forms in terms of $\xi_i$. Since
\[x_i=\frac{\T\us}{\T \xi_i},\,\, u=\sum\xi_i\frac{\T\us}{\T\xi_i}-\us,\]
and it is well known that
$$\left(\frac{\T^2 u}{\T x_i\T x_j}\right)=\left(\frac{\T^2 \us}{\T \xi_i\T \xi_j}\right)^{-1}.$$
We have, using the coordinate $\{\xi_1,\xi_2,\cdots,\xi_n\}$, the first and second fundamental forms can be rewritten as:
$$g_{ij}=\delta_{ij}-\xi_i\xi_j, \text{ and\,\,  } h_{ij}=\frac{u^{* ij}}{\sqrt{1-|\xi|^2}},$$
where $\lt(u^{* ij}\rt)$ denotes the inverse matrix of $(\us_{ij})$ and $|\xi|^2=\sum_i\xi_i^2$. Now, let $W$ denote the Weingarten matrix of $\M,$ then
$$(W^{-1})_{ij}=\sqrt{1-|\xi|^2}g_{ik}\us_{kj}.$$

From the discussion above we can see that if $\M_u=\{(x, u(x)) | x\in\R^n\}$ is a complete, strictly convex, spacelike
hypersurface satisfies $\sigma_{n-1}(\la[\M])=1,$ then the Legendre transform of
$u$ denoted by $\us,$ satisfies $\frac{\sigma_n}{\sigma_1}(\la^*[\w\gas_{ik}\us_{kl}\gas_{lj}])=1.$
Here $\w=\sqrt{1-|\xi|^2}$ and $\gas_{ij}=\delta_{ij}-\frac{\xi_i\xi_j}{1+\w}$ is the square root of the matrix $g_{ij}.$

\bigskip

\section{Construction of $\sigma_{n-1}=1$ convex hypersurfaces}
\label{cs}
Sections \ref{cs}, \ref{es}, \ref{cv}, and \ref{ub} will be devoted to the construction of complete, strictly convex,
spacelike $\sigma_{n-1}=1$ hypersurfaces with bounded principal curvatures. There are a few difficulties we need to conquer in
this construction process. First, we need to make sure the hypersurface we construct is strictly convex. Second, we need to show that
the hypersurface we construct has bounded principal curvatures. In order to overcome these difficulties, we will apply Anmin Li's idea
(see \cite{LA}) to study the Legendre transform of the solution.

Let's recall Theorem 3.1 in \cite{LA}.
\begin{theo}
\label{csth1.1}(Theorem 3.1 in \cite{LA})
Let $\M$ be a closed, noncompact, spacelike, strictly convex hypersurface. If there exists a constant
$d>0$ such that $\la_i\geq d$ for all $i=1, 2, \cdots, n$ everywhere on $\M,$ then\\
1. The Gauss map $G: \M\rightarrow B_1$ is a diffeomorphism;\\
2. $\varphi\in C^0(\partial B_1),$ where $\varphi=\lim_{\xi\rightarrow\partial B_1}\us(\xi).$\\
Here $\us$ is the Legendre transform of the height function of $\M.$
\end{theo}

From Theorem \ref{csth1.1} and the discussion in Subsection \ref{lt}, we know that for a closed, noncompact, spacelike, strictly convex hypersurface
$\M$ with principal curvatures bounded from below by a positive constant, and satisfies
\[\sigma_{n-1}(\la[\M])=1,\] its Legendre transform $\us$ must satisfy the following equation:
\be\label{cs1.1*}
\left\{
\begin{aligned}
F(\w\gas_{ik}\us_{kl}\gas_{lj})&=1,\,\,\mbox{in $B_1$}\\
\us&=\vp,\,\,\mbox{on $\partial B_1$,}
\end{aligned}
\right.
\ee
where $\vp\in C^0(\partial B_1),$ $\w=\sqrt{1-|\xi|^2},$ $\gas_{ik}=\delta_{ik}-\frac{\xi_i\xi_k}{1+\w},$
$\us_{kl}=\frac{\T^2\us}{\T\xi_k\T\xi_l},$
and $F(\w\gas_{ik}\us_{kl}\gas_{lj})=\lt(\frac{\sigma_n}{\sigma_1}(\la^*[\w\gas_{ik}\us_{kl}\gas_{lj}])\rt)^{\frac{1}{n-1}}.$

Due to technical issues, we cannot solve the Dirichlet problem with $C^0$ boundary data.
In the following, we will study the existence of solutions to the following equation instead:
\be\label{cs1.1}
\left\{
\begin{aligned}
F(\w\gas_{ik}\us_{kl}\gas_{lj})&=1,\,\,\mbox{in $B_1$}\\
\us&=\vp,\,\,\mbox{on $\partial B_1$,}
\end{aligned}
\right.
\ee
where $\vp\in C^2(\partial B_1).$

Notice that equation \eqref{cs1.1} is degenerate on $\partial B_1.$ Therefore, we will consider
the approximate problem:
\be\label{cs1.2}
\left\{
\begin{aligned}
F(\w\gas_{ik}\urs_{kl}\gas_{lj})&=1,\,\,\mbox{in $B_r$}\\
\urs&=\vp,\,\,\mbox{on $\partial B_r$,}
\end{aligned}
\right.
\ee
where $0<r<1.$

\bigskip

\section{Existence of solutions to equation \eqref{cs1.2}}
\label{es}
In this section, we will show that for each $0<r<1,$ there exists a solution to equation \eqref{cs1.2}.
\subsection{$C^0$ estimates}
\label{es0}
Since $\urs$ is a convex function we have
\[\max_{B_r} \urs\leq \max_{\partial B_r}\vp.\]
In order to show that $\urs$ is bounded from below, similar to \cite{LA}, we consider
a special subsolution of \eqref{cs1.1}
\[\lus=-n^{\frac{1}{n-1}}\sqrt{1-|\xi|^2}+a_1\xi_1+ \cdots +a_n\xi_n+c,\]
where $a_1, \cdots, a_n, c$ are constants such that
\[\lus|_{\partial B_1}<\inf_{\partial B_1}\vp.\]
Note that $\lus$ is the linear translation of the Legendre transform of a standard Hyperboloid
whose principal curvatures are equal to $n^{-\frac{1}{n-1}}$.
Then the maximum principle implies $\urs>\lus$ for any $0<r<1.$

\subsection{$C^1$ estimates}
\label{es1}
By Section 2 of \cite{CNS3}, for any $0<r<1,$ we can construct a subsolution
$\lu^{r*}$ such that
\be\label{es1.1}
\left\{
\begin{aligned}
F(\w\gas_{ik}\lu^{r*}_{kl}\gas_{lj})&\geq 1,\,\,\mbox{in $B_r$}\\
\lu^{r*}&=\vp,\,\,\mbox{on $\partial B_r$.}
\end{aligned}
\right.
\ee
Then by the convexity of $\urs$ we have
\[|D\urs|\leq\max_{\partial B_r}|D\lu^{r*}|.\]

\subsection{$C^2$ boundary estimates}
\label{es2}
For our convenience, in this subsection we will use the hyperbolic model (see Subsection \ref{gm}), and
write equation \eqref{cs1.2} as follows:
\be\label{es1.2}
\left\{
\begin{aligned}
F(v_{ij}-v\delta_{ij})&=1,\,\,\mbox{in $U_r$}\\
v&=\frac{\vp}{\sqrt{1-r^2}},\,\,\mbox{on $\partial U_r.$}
\end{aligned}
\right.
\ee
where $U_r=P^{-1}(B_r)\subset\mathbb{H}^n(-1).$ Here we want to point out that $v=\frac{\us}{\sqrt{1-|\xi|^2}}$
and $\partial U_r\subset \mathbb{P}=\{x_{n+1}=\frac{1}{\sqrt{1-r^2}}\}.$

Equation of this type has been studied by Bo Guan in \cite{Guan}. However, our function $F$ is slightly different from functions in \cite{Guan}.
More precisely, our function $F$ doesn't satisfy the assumption (1.7) in \cite{Guan}. Therefore, in order to obtain the
$C^2$ boundary estimates, we need to give
a different proof of Lemma 6.2 in \cite{Guan}.

\begin{lemm}
\label{eslem1.1}
There exist some uniform positive constant $t, \delta, \epsilon$ such that
\[h=(v-\underline{v})+t\lt(\frac{1}{\sqrt{1-r^2}}-x_{n+1}\rt)\]
satisfies
\[\mathfrak{L}h\leq-a(1+\sum_i F^{ii}),\,\,\mbox{in $U_{r\delta},$}\]
and \[h\geq 0,\,\,\mbox{on $\partial U_{r\delta}.$}\]
Here $a>0$ is some positive constant, $\underline{v}$ is a subsolution, $\mathfrak{L}f:=F^{ij}\nabla_{ij}f-f\sum_iF^{ii},$
and
$U_{r\delta}:=\lt\{x\in U_r\left|\frac{1}{\sqrt{1-r^2}}-x_{n+1}<\delta\right.\rt\}.$
\end{lemm}
\begin{proof}
When $t$ large, $\delta>0$ small, it's easy to see that we have $h\geq 0$ on $\partial U_{r\delta}.$
Moreover, by equation \eqref{gg1.2} we get
\[\mathfrak{L}h\leq-t\frac{1}{\sqrt{1-r^2}}\sum F^{ii}-C,\]
where $C$ depends on $\underline{v}.$
Therefore, we are done.
\end{proof}
We want to point out that the existence of subsolution $\underline{v}$ has been proved in Theorem 1.2 of \cite{Guan}.
The rest of $C^2$ boundary estimates follows from \cite{Guan} directly.

\subsection{Global $C^2$ estimates}
\label{es22}
Just like before, since we don't have the assumption (1.7) of \cite{Guan}, we cannot apply the global $C^2$ estimates there.
We need another approach to prove the $C^2$ estimate of \eqref{cs1.2}.
In particular, we will study the Legendre transform of $\urs,$ which we will denote
by $\ur.$ We will also denote $\Omega_r=D\urs(B_r).$ Then, it's easy to see that $\ur$ satisfies
\be\label{es1.3}
\sigma_{n-1}(\la[a_{ij}])=1,\,\,\mbox{in $\Omega_r$,}
\ee
where $a_{ij}=\frac{\ga^{ik}\ur_{kl}\ga^{lj}}{w},$
$\ga^{ik}=\delta_{ik}+\frac{\ur_i\ur_k}{w(1+w)},$ and $w=\sqrt{1-|D\ur|^2}.$

Since the principal curvature lower bound of $\la[a_{ij}]$ implies the curvature radius upper bound of
$\la^*[\w\gas_{ik}u^{r*}_{kl}\gas_{lj}].$  We will consider
\be\label{es1.4}
\phi=-\log\sigma_n(\la_1, \cdots, \la_n)-N\lt<\nu, E\rt>.
\ee
If $\phi$ achieves its maximum at an interior point $x_0\in\Omega_r.$ Let $\{\tau_1, \cdots, \tau_n\}$ be the orthonormal frame
such that $h_{ij}=\la_i\delta_{ij}$ at $X_0=(x_0, u^r(x_0)).$ Then at this point we have,
\be\label{es1.5}
\phi_i=-\frac{(\sigma_n)_i}{\sigma_n}-Nh_{im}\lt<\tau_m, E\rt>=0,
\ee
and
\be\label{es1.6}
\begin{aligned}
0\geq\sigma^{ii}_{n-1}\phi_{ii}&=-\frac{\sigma_{n-1}^{ii}(\sigma_n)_{ii}}{\sigma_n}
+\frac{\sigma_{n-1}^{ii}(\sigma_n)^2_i}{\sigma_n^2}-N\sigma_{n-1}^{ii}\la_i^2\lt<\nu, E\rt>\\
&=-\frac{\sigma_{n-1}^{ii}(\sigma_n)_{ii}}{\sigma_n}+\sigma_{n-1}^{ii}N^2\la_i^2\lt<\tau_i, E\rt>^2
-N\sigma_{n-1}^{ii}\la_i^2\lt<\nu, E\rt>\\
&\geq n^2\sigma_n-\sigma_1\sigma_{n-1}-N\sigma_1\sigma_{n-1}\lt<\nu, E\rt>\\
&\geq n^2\sigma_n-\sigma_1\sigma_{n-1}+N\sigma_1\sigma_{n-1}\frac{1}{\sqrt{1-|Du|^2}},
\end{aligned}
\ee
where we have used $\sigma_{n-1}^{ii}h_{iik}=0$ and Theorem \ref{ceth1.1}.

Choosing $N=2$  leads to a contradiction. Therefore, we conclude that $\phi$
achieves its maximum at the boundary $\T\Omega_r.$ Combining with the boundary $C^2$ estimates in Subsection \ref{es2}, we obtain
$\sigma_n(\la^*)$ is bounded from above. Since $\frac{\sigma_n(\la^*)}{\sigma_1(\la^*)}=1,$ we have $\sigma_1(\la^*)$ is bounded from above.
Therefore, we obtain the $C^2$ estimates for $|D^2\urs|.$
By the standard continuity argument, we know that equation \eqref{cs1.2} is solvable for any $0<r<1.$

\bigskip
\section{Convergence of solutions to a strictly convex hypersurface}
\label{cv}
In this section we want to construct the solution to equation \eqref{cs1.1}.

\subsection{Barrier function}
\label{cvb}
First, recall section 4 of \cite{LA} we know that there exists a strictly convex solution
$\uus\in C^{\infty}(B_1)\cap C^0(\bar{B}_1)$ satisfies
\be\label{cvb1.1}
\left\{
\begin{aligned}
\det\lt(\w\gas_{ki}\us_{kl}\gas_{lj}\rt)&=n^{\frac{n}{n-1}},\,\,\mbox{in $B_1$}\\
\us&=\vp,\,\,\mbox{on $\T B_1.$}
\end{aligned}
\right.
\ee
By Maclaurin's inequality, we know that $\uus$ is a supersolution of equation \eqref{cs1.1}.
On the other hand, consider the function
\[\lus_0=-A\sqrt{1-|\xi|^2}+\vp,\]
by a straightforward calculation we can see that, when $A>0$ sufficiently large,
$\lus_0$ is a subsolution of \eqref{cs1.1}.

In this section we will consider the convergence of functions $\urs,$ where $\urs$ satisfies
\be\label{cvb1.2}
\left\{
\begin{aligned}
F(\w\gas_{ik}\urs_{kl}\gas_{lj})&=1,\,\,\mbox{in $B_r$}\\
\urs&=\uus,\,\,\mbox{on $\partial B_r$.}
\end{aligned}
\right.
\ee
Note that the existence of the solution to equation \eqref{cvb1.2} has been proved in Section \ref{es}.
In the following, we will denote $u^r$ as the Legendre transform of $\urs,$ $\lu_0$ as the Legendre transform of $\lus_0$,
and $\uu$ as the Legendre transform of $\uus.$ We will also denote $\Omega_r=D\urs(B_r)$ as the domain of $\ur.$
We will show that there exists a subsequence of $\{\ur\}$ which converges locally smoothly to a strictly convex function $u$. Moreover, $u$ satisfies $\sigma_{n-1}(\la[\M_u])=1$ and $\la[\M_u]\leq C.$

\subsection{Local $C^0$ estimates}
\label{cv0}
By the maximum principle we know that for any $r>0$ we have $\lus_0\leq\urs\leq\uus$ everywhere. Therefore,
\[\lt|\urs\rt|\leq C_0.\]
Moreover, since $\urs$ is convex we have, for any $r>1/2$
\[\lt|D\urs(0)\rt|\leq 2\lt(\max\urs-\min\urs\rt)\leq C_1.\]
Note also that at the point where $\min u^r$ is achieved we have $Du^r=0.$
Thus, $\min\ur$ is achieved in $B_{C_1}(0)\subset\R^n.$
On the other hand, when $r>1/2$ we have,
\[|\min\ur|=\lt|0\cdot D\urs(0)-\urs(0)\rt|\leq C_0.\]
These together with the fact that $|Du^r|<1$ yield in a ball of radius $R>C_1$ we have
\[|u^r|\leq C_0+R,\]
for any $r>r_0>1/2.$ Furthermore, from the discussion above, we know that by a coordinate transformation,
we may always assume $2C_0+1\geq u^r(0)\geq1$ and $D\ur(0)=0.$
\subsection{Local $C^1$ estimates}
\label{cv1}
Before we start to work on the derivative estimates, we need the following Lemma.
\begin{lemm}
\label{cvlem1.7}
Let $\urs$ be the solution of \eqref{cvb1.2} and $\ur$ be the Legendre transform of $\urs.$
Then, $\ur|_{\partial\Omega_r}\rightarrow+\infty$ as $r\rightarrow 1.$
\end{lemm}
\begin{proof}
By Lemma 5.6 in \cite{LA} and the maximum principle we have
\be\label{cv1.18}
\begin{aligned}
\ur|_{\T\Omega_r}&=[\xi\cdot D\urs(\xi)-\urs(\xi)]|_{\T B_r}\\
&\geq [\xi\cdot D\uus(\xi)-\uus(\xi)]|_{\T B_r}\\
&\geq\frac{d_1}{\sqrt{1-r^2}}.
\end{aligned}
\ee
Therefore, it's easy to see that $\ur|_{\partial\Omega_r}\rightarrow+\infty$ as $r\rightarrow 1.$
\end{proof}

Next, we will prove the local $C^1$ estimates. In particular, let $\urs$ be the solution of \eqref{cvb1.2} for $r\geq \frac{3}{4},$ and let $\ur$ be the Lengendre transform of $\urs$. we will obtain the interior gradient estimates for $\ur.$

Now let $S^{n-1}(r) = \{\xi\in\R^n |\,\, \sum\xi_i^2=r\},$ where $\frac{3}{4}\leq r\leq 1.$ We denote the angular derivative
$\xi_k\frac{\partial}{\partial\xi_l}-\xi_l\frac{\partial}{\partial\xi_k}$ on
$S^{n-1}(r)$ by $\partial_{k, l},$ or simply by $\partial$, when no confusion arises. Then we have following Lemmas.
\begin{lemm}
\label{cvlem1.1}
Let $\urs$ be the solution of equation \eqref{cvb1.2}. Then,
$|\partial\urs|$ is bounded by a constant depends on $|\uus|_{C^1}$. Moreover, this conclusion also holds for $r=1.$
\end{lemm}
\begin{proof}
Without loss of generality, we assume that $\partial=\xi_1\frac{\partial}{\partial\xi_2}-\xi_2\frac{\partial}{\partial\xi_1}.$ Also, without causing any confusion, we will drop the supscript $r$.
By our equation \eqref{cvb1.2} we have
\[F(\gas_{ik}\us_{kl}\gas_{lj})=\frac{1}{\w}.\]
Differentiating it with respect to $\partial$ we get
\be\label{cv1.1}
F^{kl}\partial(\gas_{ik}\us_{ij}\gas_{lj})=0.
\ee
Since
\be\label{cv1.2}
\T(\gas_{ki}\us_{ij}\gas_{jl})=(\T\gas_{ki})\us_{ij}\gas_{jl}+\gas_{ki}\us_{ij}(\T\gas_{jl})
+\gas_{ki}(\T\us_{ij})\gas_{jl},
\ee
we will compute these terms one by one.

First, we can see that
\be\label{cv1.3}
\begin{aligned}
(\T\us)_{ij}&=(\xi_1\T_2\us-\xi_2\T_1\us)_{ij}\\
&=\delta_{1i}\us_{2j}+\delta_{1j}\us_{2i}-\delta_{2i}\us_{1j}-\delta_{2j}\us_{1i}+\xi_1\us_{2ij}-\xi_2\us_{1ij}.
\end{aligned}
\ee
Thus, we have
\be\label{cv1.4}
\gas_{ki}(\T\us)_{ij}\gas_{jl}=\gas_{k1}\us_{2j}\gas_{jl}+\gas_{ki}\us_{2i}\gas_{1l}
-\gas_{k2}\us_{1j}\gas_{jl}-\gas_{ki}\us_{1i}\gas_{2l}
+\gas_{ki}\T(\us_{ij})\gas_{jl}.
\ee
Next, we differentiate $\gas_{ki}$ and get
\be\label{cv1.5}
\begin{aligned}
\T\gas_{ki}&=\lt(\xi_1\T_2-\xi_2\T_1\rt)\lt(\delta_{ki}-\frac{\xi_k\xi_i}{1+\w}\rt)\\
&=-\frac{1}{1+\w}(\xi_1\delta_{k2}\xi_i+\xi_1\xi_k\delta_{i2}-\xi_2\delta_{k1}\xi_i-\xi_2\xi_k\delta_{1i})\\
&=\delta_{k2}\gas_{1i}+\delta_{i2}\gas_{1k}-\delta_{k1}\gas_{2i}-\delta_{1i}\gas_{2k}.
\end{aligned}
\ee
Hence, we have
\be\label{cv1.6}
\begin{aligned}
(\T\gas_{ki})\us_{ij}\gas_{jl}&=\delta_{k2}\gas_{1i}\us_{ij}\gas_{jl}+\gas_{1k}\us_{2j}\gas_{jl}\\
&-\delta_{k1}\gas_{2i}\us_{ij}\gas_{jl}-\gas_{jl}\us_{1j}\gas_{2k}.
\end{aligned}
\ee
Similarly, we have
\be\label{cv1.7}
\begin{aligned}
\gas_{ki}\us_{ij}(\T\gas_{jl})=\delta_{l2}\gas_{1j}\us_{ij}\gas_{ik}+\gas_{1l}\us_{2i}\gas_{ik}
-\delta_{l1}\gas_{2j}\us_{ij}\gas_{ik}-\gas_{ik}\us_{1i}\gas_{2l}.
\end{aligned}
\ee
Therefore we conclude that,
\be\label{cv1.7*}
\begin{aligned}
\T(\gas_{ki}\us_{ij}\gas_{jl})&=\delta_{k2}\gas_{1i}\us_{ij}\gas_{jl}-\delta_{k1}\gas_{2i}\us_{ij}\gas_{jl}\\
&+\gas_{ki}(\T\us)_{ij}\gas_{jl}+\delta_{l2}\gas_{1j}\us_{ij}\gas_{ik}-\delta_{1l}\gas_{2j}\us_{ij}\gas_{ik}
\end{aligned}
\ee
Plugging equation \eqref{cv1.7*} into \eqref{cv1.1} we get,
\be\label{cv1.8}
F^{kl}\gas_{li}(\T\us)_{ij}\gas_{jk}=0.
\ee
Using the maximum principle we obtain
\[|\partial\us|\leq |\uus|_{C^1}.\]
\end{proof}

\begin{lemm}
\label{cvlem1.2}
Let $\urs$ be the solution of equation \eqref{cvb1.2}.
Then, $\T^2\urs$ is bounded above by a constant depends on $|\uus|_{C^2}$. Moreover, this conclusion also holds for $r=1.$
\end{lemm}
\begin{proof}
In this proof, without causing any confusion, we will drop the supscript $r$.
Suppose $\{\xi_1,\cdots,\xi_n\}$ is a local coordinate and $\xi$ is the position vector. We want to show for any angular derivative operator
$$\T=\langle a, \xi\rangle \langle b, D\rangle-\langle b, \xi\rangle \langle a, D\rangle,$$
we have \be\label{add1}
F^{kl}(\T^2a_{kl})+F^{pq, rs}(\T a_{pq})(\T a_{rs})
=F^{kl}\gas_{ki}(\T^2 u)_{ij}\gas_{jl}+F^{pq, rs}(\gas_{pi}(\T \us)_{ij}\gas_{jq})(\gas_{ri}(\T \us)_{ij}\gas_{js}).
\ee
Here $a,b$ are two unit perpendicular vectors and $D$ is the differential operator.
If we denote
$$\p_{i,j}=\xi_i\frac{\p}{\p\xi_j}-\xi_j\frac{\p}{\p \xi_i}=\langle \frac{\p}{\p \xi_i}, \xi\rangle \langle \frac{\p}{\p \xi_j}, D\rangle-\langle \frac{\p}{\p \xi_j}, \xi\rangle \langle \frac{\p}{\p \xi_i}, D\rangle$$ and let
$$a=\sum_{i}a_i\frac{\p}{\p\xi_i}; b=\sum_j b_j\frac{\p}{\p\xi_j}.$$ Then, we obtain
$$\T=\sum_{i,j}a_ib_j\p_{i,j}$$ and $$\T^2=\sum_{i,j,l,k}a_ib_ja_lb_k\p_{i,j}\p_{l,k}.$$
For our convenience, we will also denote $a_{kl}=\gas_{ki}\us_{ij}\gas_{jl}.$
Note that $F$ depends only on the eigenvalues of the matrix $A=(a_{kl}),$ it is invariant under rotation of
coordinates. Therefore, during our calculations we can always assume $A$ is diagonal, i.e., $a_{kl}=\la_k\delta_{kl}.$
In the following, we will denote $\T_0=\xi_1\frac{\partial}{\partial\xi_2}-\xi_2\frac{\partial}{\partial\xi_1},$
$\T_1=\xi_1\frac{\partial}{\partial\xi_3}-\xi_3\frac{\partial}{\partial\xi_1},$
and $\T_2=\xi_3\frac{\partial}{\partial\xi_4}-\xi_4\frac{\partial}{\partial\xi_3}.$ We will also denote
$\ta_{kl}=\gas_{ki}(\T_0\us)_{ij}\gas_{jl},$ $\ba_{kl}=\gas_{ki}(\T_1\us)_{ij}\gas_{jl},$ and $\ha_{kl}=\gas_{ki}(\T_2\us)_{ij}\gas_{jl}$
It's easy to see that, in order to prove \eqref{add1} we only need to show
\[F^{kl}(\p_0^2a_{kl})+F^{pq, rs}(\p_0 a_{pq})(\p_0 a_{rs})
=F^{kl}\gas_{ki}(\p_0^2 u)_{ij}\gas_{jl}+F^{pq, rs}(\gas_{pi}(\p_0 \us)_{ij}\gas_{jq})(\gas_{ri}(\p_0 \us)_{ij}\gas_{js}),\]
\[F^{kl}(\T_1\p_0 a_{kl})+F^{pq, rs}(\T_1 a_{pq})(\p_0 a_{rs})
=F^{kl}\gas_{ki}(\T_1\p_0 u)_{ij}\gas_{jl}+F^{pq, rs}(\gas_{pi}(\T_1 \us)_{ij}\gas_{jq})(\gas_{ri}(\p_0 \us)_{ij}\gas_{js}),\]
and
\[F^{kl}(\T_2\p_0 a_{kl})+F^{pq, rs}(\T_2 a_{pq})(\p_0 a_{rs})
=F^{kl}\gas_{ki}(\T_2\p_0 u)_{ij}\gas_{jl}+F^{pq, rs}(\gas_{pi}(\T_2 \us)_{ij}\gas_{jq})(\gas_{ri}(\p_0 \us)_{ij}\gas_{js}).\]

Following the notation in section 2, we define
\[f(\la[A])=F(A).\]
First, differentiating \eqref{cv1.1} with respect to $\T_0$ twice we get
\be\label{cv1.9}
F^{kl}(\T^2_0a_{kl})+F^{pq, rs}(\T_0 a_{pq})(\T_0 a_{rs})=0.
\ee
By \eqref{cv1.7*} we get
\be\label{cv1.10}
\T_0 a_{kl}=\delta_{k2}a_{1l}-\delta_{k1}a_{2l}+\delta_{l2}a_{1k}-\delta_{l1}a_{2k}+\ta_{kl},
\ee
and
\be\label{cv1.11}
\begin{aligned}
\T^2_0a_{kl}&=\delta_{k2}(\T_0 a_{1l})-\delta_{k1}(\T_0 a_{2l})+\delta_{l2}(\T_0 a_{1k})-\delta_{l1}(\T_0 a_{2k})+\T_0\ta_{kl}\\
&=\delta_{k2}(\T_0 a_{1l})-\delta_{k1}(\T_0 a_{2l})+\delta_{l2}(\T_0 a_{1k})-\delta_{l1}(\T_0 a_{2k})\\
&+\delta_{k2}\ta_{1l}-\delta_{k1}\ta_{2l}+\delta_{l2}\ta_{1k}-\delta_{l1}\ta_{2k}+\gas_{ki}(\T_0^2\us)_{ij}\gas_{jl}.
\end{aligned}
\ee

Contracting \eqref{cv1.11} with $F^{kl}$ we obtain,
\be\label{cv1.13}
\begin{aligned}
F^{kl}(\T_0^2a_{kl})&=2\lt(-F^{22}a_{22}+F^{22}a_{11}-F^{11}a_{11}+F^{11}a_{22}+F^{22}\ta_{12}-F^{11}\ta_{21}\rt)\\
&+2\lt(F^{22}\ta_{12}-F^{11}\ta_{21}\rt)+F^{kl}\gas_{ki}(\T_0^2\us)_{ij}\gas_{jl}\\
&=-2(f_2-f_1)(\la_2-\la_1)+4(f_2-f_1)\ta_{12}+F^{kl}\gas_{ki}(\T_0^2\us)_{ij}\gas_{jl}.
\end{aligned}
\ee
On the other hand by Lemma \ref{prelm1.1} we have,
\be\label{cv1.14}
\begin{aligned}
&F^{pq, rs}(\T_0 a_{pq})(\T_0 a_{rs})\\
&=F^{pq, rs}\lt(\delta_{p2}a_{1q}-\delta_{p1}a_{2q}+\delta_{q2}a_{1p}-\delta_{q1}a_{2p}+\ta_{pq}\rt)\\
&\times\lt(\delta_{r2}a_{1s}-\delta_{r1}a_{2s}+\delta_{s2}a_{1r}-\delta_{s1}a_{2r}+\ta_{sr}\rt)\\
&=F^{pp, rr}\lt(\delta_{p2}a_{1p}-\delta_{p1}a_{2p}+\delta_{p2}a_{1p}-\delta_{p1}a_{2p}+\ta_{pp}\rt)\\
&\times\lt(\delta_{r2}a_{1r}-\delta_{r1}a_{2r}+\delta_{r2}a_{1r}-\delta_{r1}a_{2r}+\ta_{rr}\rt)\\
&+\sum_{p\neq r}\frac{f_p-f_r}{\la_p-\la_r}\lt(\delta_{p2}a_{1r}-\delta_{p1}a_{2r}+\delta_{r2}a_{1p}-\delta_{r1}a_{2p}+\ta_{pr}\rt)^2\\
&=F^{pq, rs}\ta_{pq}\ta_{rs}+2\frac{f_2-f_1}{\la_2-\la_1}\la_1^2+2\frac{f_1-f_2}{\la_1-\la_2}\la^2_2\\
&-4\frac{f_2-f_1}{\la_2-\la_1}\la_1\la_2+4\frac{f_2-f_1}{\la_2-\la_1}\la_1\ta_{21}-4\frac{f_1-f_2}{\la_1-\la_2}\la_2\ta_{12}\\
&=2(f_2-f_1)(\la_2-\la_1)-4(f_2-f_1)\ta_{12}+F^{pq, rs}\ta_{pq}\ta_{rs}.
\end{aligned}
\ee
Equations \eqref{cv1.13} and \eqref{cv1.14} yields
\be\label{cv1.15}
0=F^{kl}(\T_0^2a_{kl})+F^{pq, rs}(\T_0 a_{pq})(\T_0 a_{rs})=F^{kl}\gas_{ki}(\T^2_0u)_{ij}\gas_{jl}+F^{pq, rs}\ta_{pq}\ta_{rs}.
\ee

Next, we differentiate \eqref{cv1.1} with respect to $\T_1\T_0,$ this gives
\be\label{ad2}
F^{kl}(\T_1\T_0a_{kl})+F^{pq, rs}(\T_1 a_{pq})(\T_0 a_{rs})=0.
\ee
By a straightforward calculation we get
\be\label{ad3}
\begin{aligned}
F^{kl}\T_1\T_0a_{kl}&=F^{kl}\lt[\delta_{k2}(-a_{3l}+\delta_{l3}a_{11}+\ba_{1l})\right.\\
&-\delta_{k1}\ba_{2l}+\delta_{l2}(-a_{3k}+\delta_{k3}a_{11}+\ba_{1k})-\delta_{l1}\ba_{2k}\\
&\left.+\delta_{k3}\ta_{1l}-\delta_{k1}\ta_{3l}+\delta_{l3}\ta_{1k}-\delta_{l1}\ta_{3k}+\gas_{ki}(\T_1\T_0\us)_{ij}\gas_{jl}\rt]\\
&=2(f_2-f_1)\ba_{12}+2(f_3-f_1)\ta_{13}+F^{kl}\gas_{ki}(\T_1\T_0\us)_{ij}\gas_{jl}.
\end{aligned}
\ee
Moreover, we have
\be\label{ad4}
\begin{aligned}
&F^{pq, rs}(\T_1a_{pq})(\T_0a_{rs})\\
&=F^{pq, rs}(\delta_{p3}a_{1q}-\delta_{p1}a_{3q}+\delta_{q3}a_{1p}-\delta_{q1}a_{3p}+\ba_{pq})\\
&\times(\delta_{r2}a_{1s}-\delta_{r1}a_{2s}+\delta_{s2}a_{1r}-\delta_{s1}a_{2r}+\ta_{sr})\\
&=F^{pp, rr}\ba_{pp}\ta_{rr}+\sum\limits_{p\neq r}\frac{f_p-f_r}{\la_p-\la_r}(\delta_{p3}a_{1r}\ta_{pr}-\delta_{p1}a_{3r}\ta_{pr}\\
&+\delta_{r3}a_{1p}\ta_{pr}-\delta_{r1}a_{3p}\ta_{pr}+\delta_{r2}a_{1p}\ba_{pr}-\delta_{r1}a_{2p}\ba_{pr}\\
&+\delta_{p2}a_{1r}\ba_{pr}-\delta_{p1}a_{2r}\ba_{pr}+\ta_{pr}\ba_{pr})\\
&=F^{pp, rr}\ba_{pp}\ta_{rr}+2\frac{f_3-f_1}{\la_3-\la_1}\ta_{13}(\la_1-\la_3)\\
&+2\frac{f_2-f_1}{\la_2-\la_1}\ba_{12}(\la_1-\la_2)+\sum\limits_{p\neq r}\frac{f_p-f_r}{\la_p-\la_r}\ta_{pr}\ba_{pr}.
\end{aligned}
\ee
From \eqref{ad3} and \eqref{ad4} we obtain
\be\label{ad5}
0=F^{kl}(\T_1\T_0a_{kl})+F^{pq, rs}(\T_1 a_{pq})(\T_0 a_{rs})=F^{kl}\gas_{ki}(\T_1\T_0u)_{ij}\gas_{jl}+F^{pq, rs}\ba_{pq}\ta_{rs}.
\ee
Similarly, we can get
\be\label{ad6}
F^{kl}\T_2\T_0a_{kl}=2(f_2-f_1)\ha_{12}+2(f_4-f_3)\ta_{34}+F^{kl}\gas_{ki}(\T_2\T_0\us)_{ij}\gas_{jl},
\ee
and
\be\label{ad7}
F^{pq, rs}(\T_2a_{pq})(\T a_{rs})=F^{pq,rs}\ha_{pq}\ta_{rs}+2\frac{f_4-f_3}{\la_4-\la_3}\ta_{34}(\la_3-\la_4)
+2\frac{f_1-f_2}{\la_1-\la_2}\ha_{12}(\la_1-\la_2).
\ee
This yields
\be\label{ad8}
0=F^{kl}(\T_2\T_0a_{kl})+F^{pq, rs}(\T_2 a_{pq})(\T_0 a_{rs})=F^{kl}\gas_{ki}(\T_2\T_0u)_{ij}\gas_{jl}+F^{pq, rs}\ha_{pq}\ta_{rs}.
\ee
Combining \eqref{cv1.15}, \eqref{ad5}, \eqref{ad8} with the concavity of $F$ we conclude
\be\label{cv1.16}
F^{kl}\gas_{ki}(\T^2u)_{ij}\gas_{jl}\geq 0.
\ee
By the maximum principle, we prove this Lemma.
\end{proof}

Following \cite{LA}, using Lemma \ref{cvlem1.1} and Lemma \ref{cvlem1.2} we can also prove the following Lemmas. Since the proof of these Lemmas are almost identical to \cite{LA} (see Lemma 5.3-Lemma 5.6), we will omit them here.
\begin{lemm}
\label{cvlem1.3}
Let $\urs$ be the solution of equation \eqref{cvb1.2}. Then, there is a positive constant $b$ such that
\[\sqrt{1-|\xi|^2}\lt|\T^2\us\rt|<b.\] Moreover, this conclusion also holds for $r=1.$
\end{lemm}

\begin{lemm}
\label{cvlem1.4}
Let $\urs$ be the solution of equation \eqref{cvb1.2}, $\frac{1}{2}<\td{r}<r,$ and $S^{n-1}(\td{r})=\{\xi\in\R^n|\sum\xi_i^2=\td{r}^2\}.$
For any point $\hat{\xi}\in S^{n-1}(\td{r}),$ there is a function
\[\overline{u}^*=-n^{\frac{1}{n-1}}\sqrt{1-|\xi|^2}+b_1\xi_1+ \cdots +b_n\xi_n+b\]
such that
\[\overline{u}^*(\hat{\xi})=\us(\hat{\xi}),\]
and\[\overline{u}^*(\hat{\xi})>\us(\xi),\,\,\mbox{for any $\xi\in S^{n-1}(\td{r})\setminus\{\hat{\xi}\}$}.\]
Here $b_1, \cdots, b_n$ are constants depending on $\hat{\xi},$ and $b>0$ is a positive constant satisfying $b>C$ for some positive constant
$C$ depending only on $|\bar{u}_0^*|_{C^2}.$ Moreover, this conclusion also holds for $r=1.$
\end{lemm}

Similarly we have
\begin{lemm}
\label{cvlem1.5}
Let $\urs$ be the solution of equation \eqref{cvb1.2}, $\frac{1}{2}<\td{r}<r,$ and $S^{n-1}(\td{r})=\{\xi\in\R^n|\sum\xi_i^2=\td{r}^2\}.$
For any point $\hat{\xi}\in S^{n-1}(\td{r}),$ there is a function
\[\lus=-n^{\frac{1}{n-1}}\sqrt{1-|\xi|^2}+a_1\xi_1+ \cdots +a_n\xi_n-a\]
such that
\[\lus(\hat{\xi})=\us(\hat{\xi}),\]
and\[\lus(\hat{\xi})<\us(\xi),\,\,\mbox{for any $\xi\in S^{n-1}(\td{r})\setminus\{\hat{\xi}\}$}.\]
Here $a_1, \cdots, a_n,$ $a$ are constants depending on $\hat{\xi},$ and $a>0,$ $a\sqrt{1-\sum\hat{\xi}_k^2}<C$
for some positive constant $C$ depending only on $|\uus|_{C^2}.$
  Moreover, this conclusion also holds for $r=1.$
\end{lemm}

Using Lemma \ref{cvlem1.4} and Lemma \ref{cvlem1.5} we can show
\begin{lemm}
\label{cvlem1.6}
Let $\urs$ be the solution of equation \eqref{cvb1.2} and $\ur$ be the Legendre transform of $\urs.$
There are positive constants $d_2>d_1$ that are independent of $r,$ such that
\be\label{cv1.17}
0<d_1\leq \ur\sqrt{1-|D\ur|^2}\leq d_2,
\ee
where $d_2$ depends on $|\ur|_{C^0(\Omega)}$ and $\Omega=\{x\in\R^n | |D\ur(x)|\leq \frac{1}{2}\}.$  
Moreover, this conclusion also holds for $r=1.$
\end{lemm}
Note that, we do not need $d_2$ to obtain the local $C^1$ estimates. We only need the upper bound of $\ur\sqrt{1-|D\ur|^2}$ for $r=1$ in Section
\ref{ub}.

Combining Lemma \ref{cvlem1.7} with Lemma \ref{cvlem1.6} we conclude
\begin{coro}
\label{c2cor}
For any $s>2C_0+1$, there exists $r_s>0$ such that when $r>r_s,$ $\ur|_{\partial\Omega_r}>s.$
Moreover, in the domain $\{x\in\Omega_r |\, \ur(x)<s\}$ we have
\be\label{c2.2}
\sqrt{1-|D\ur|^2}\geq\frac{C_4}{s}.
\ee
\end{coro}

\subsection{Local $C^2$ estimates}
\label{cv2}
\begin{lemm}
\label{cvlem1.8}
Let $\urs$ be the solution of \eqref{cvb1.2} and $\ur$ be the Legendre transform of $\urs.$
For any given $s>2C_0+1,$ let $r_s>0$ be a positive number such that when $r>r_s,$
$\ur|_{\partial\Omega_r}>s.$ Let $\la_{\max}(x)$ be the largest principal curvature of $\M_{\ur}$ at $x,$
where $\M_{\ur}=\{(x, \ur(x)) |\,x\in\Omega_r\}.$ Then, for $r>r_s$ we have
\[\max_{\M_{\ur}}(s-\ur)\la_{\max}\leq C_5.\]
Here, $C_5$ depends on the local $C^1$ estimates of $\ur$ and $s.$
\end{lemm}
\begin{proof}
For our convenience, in this proof we will omit the superscript $r.$ Let's consider
\[\vp=m\log(s-u)+\log P_m-mN\lt<\nu, E\rt>,\]
where $P_m=\sum_j\la_j^m$ and $N, m$ are some constants to be determined.
Suppose that the function $\vp$ achieves its maximum on $\M$ at some point $x_0,$ we may choose a local orthonormal frame
$\{\tau_1, \cdots, \tau_n\}$ such that at $x_0,$ $h_{ij}=\la_i\delta_{ij}$ and $\la_1\geq\la_2\geq\cdots\geq\la_n.$
Differentiating $\vp$ twice at $x_0$ we get
\be\label{cv1.19}
\frac{\sum_j\la_j^{m-1}h_{jji}}{P_m}-Nh_{ii}\lt<X_i, E\rt>+\frac{\lt<X_i, E\rt>}{s-u}=0
\ee
and
\be\label{cv1.20}
\begin{aligned}
0&\geq\frac{1}{P_m}\lt[\sum_j\la_j^{m-1}h_{jjii}+(m-1)\sum_j\la_j^{m-2}h^2_{jji}+\sum_{p\neq q}\frac{\la_p^{m-1}-\la_q^{m-1}}{\la_p-\la_q}h^2_{pqi}\rt]\\
&-\frac{m}{P^2_m}\lt(\sum_j\la_j^{m-1}h_{jji}\rt)^2-Nh_{iji}\lt<X_j, E\rt>-Nh^2_{ii}\lt<\nu, E\rt>\\
&+\frac{h_{ii}\lt<\nu, E\rt>}{s-u}-\frac{u^2_i}{(s-u)^2},
\end{aligned}
\ee
where $X_i=\nabla_{\tau_i}X=\tau_i,$ for $1\leq i\leq n.$
Now, let's differentiate equation \eqref{2.1} twice
\be\label{cv1.21}
\sigma_{n-1}^{ii}h_{iij}=0,\,\,\mbox{and $\sigma_{n-1}^{ii}h_{iijj}+\sigma_{n-1}^{pq, rs}h_{pqj}h_{rsj}=0.$}
\ee
Recall that in Minkowski space we have
\be\label{cv1.21*}
h_{jjii}=h_{iijj}+h_{ii}^2h_{jj}-h_{ii}h^2_{jj}.
\ee
By \eqref{cv1.20} - \eqref{cv1.21*} we find
\be\label{cv1.22}
\begin{aligned}
0&\geq\dfrac{1}{P_m}\lt\{\dsum_j\kappa_j^{m-1}\lt[-(n-1)h_{jj}^2+K(\sigma_{n-1})_j^2-\sigma_{n-1}^{pq,rs}h_{pqj}h_{rsj}\rt]\right. +(m-1)\sigma_{n-1}^{ii}\dsum_j\kappa_j^{m-2}h_{jji}^2\\
&\left.+\sigma_{n-1}^{ii}\dsum_{p\neq q}\dfrac{\kappa_p^{m-1}-\kappa_q^{m-1}}{\kappa_p-\kappa_q}h_{pqi}^2\rt\}
-\dfrac{m\sigma_{n-1}^{ii}}{P_m^2}\lt(\dsum_j \kappa_j^{m-1}h_{jji}\rt)^2\\
&-N \sigma_{n-1}^{ii}\kappa_i^2\langle\nu,E\rangle
+\dfrac{(n-1)\lt< \nu,E\rt>}{s-u}-\dfrac{\sigma_{n-1}^{ii}u_i^2}{(s-u)^2},
\end{aligned}
\ee
where $K$ is a constant.
We denote,
$$A_i=\dfrac{\kappa_i^{m-1}}{P_m}[K(\sigma_{n-1})_i^2-\dsum_{p,q}\sigma_{n-1}^{pp,qq}h_{ppi}h_{qqi}],$$
$$B_i=\dfrac{2\kappa_j^{m-1}}{P_m}\dsum_j\sigma_{n-1}^{jj,ii}h_{jji}^2,$$
$$C_i=\dfrac{m-1}{P_m}\sigma_{n-1}^{ii}\dsum_j\kappa_j^{m-2}h_{jji}^2,$$
$$D_i=\dfrac{2\sigma_{n-1}^{jj}}{P_m}\dsum_{j\neq i}\dfrac{\kappa_j^{m-1}-\kappa_i^{m-1}}{\kappa_j-\kappa_i}h_{jji}^2,$$
and
$$E_i=\dfrac{m\sigma_{n-1}^{ii}}{P_m^2}\lt(\sum_j\kappa_j^{m-1}h_{jji}\rt)^2.$$
Since
$$
-\sigma_{n-1}^{pq,rs}h_{pql}h_{rsl}=-\sigma_{n-1}^{pp,qq}h_{ppl}h_{qql}+\sigma_{n-1}^{pp,qq}h_{pql}^2,
$$
equation \eqref{cv1.22} can be written as
\be\label{cv1.23}
\begin{aligned}
0&\geq \sum_i(A_i+B_i+C_i+D_i-E_i)-\frac{(n-1)\sum\la_j^{m+1}}{P_m}\\
&-N\sigma_{n-1}^{ii}\la_i^2\lt<\nu, E\rt>+\frac{(n-1)\lt<\nu, E\rt>}{s-u}-\frac{\sigma^{ii}_{n-1}u_i^2}{(s-u)^2}.
\end{aligned}
\ee
By Lemma 8 and 9 in \cite{LRW} we can assume the following claim holds.
\begin{claim}
\label{cvlem1.9}
There exists some small positive constant $\delta>0$ such that, if $\kappa_{n-1}\leq \delta\kappa_1,$
we have
\be\label{cv1.24}
\sum_{i}(A_i+B_i+C_i+D_i)-E_1+\lt(1+\frac{1}{m}\rt)\sum_{i\geq 2}E_i\geq 0.
\ee
Here $m>0$ sufficiently large.
\end{claim}
We note that, by Lemma 8 of \cite{LRW}, for $i=2, 3, \cdots, n,$ the inequality
\[A_i+B_i+C_i+D_i-(1+\frac{1}{m})E_i\geq 0\]
always holds.
For $i=1,$ if $A_1+B_1+C_1+D_1-E_1\geq 0$ doesn't hold, there would exist a $\delta>0$ small such that $\la_{n-1}\geq\delta\la_1$. Since
$$\sigma_{n-1}=1\geq\delta^{n-2}\la_1^{n-1},$$ we would obtain an upper bound for
$\la_1$ at $x_0$ directly, then we would be done.

Therefore equation \eqref{cv1.23} and \eqref{cv1.24} implies
\be\label{cv1.25}
\begin{aligned}
0&\geq\frac{1}{P_m}\sum_j\la_j^{m-1}(-C\la_1^2)+\sum_{i=2}^n\frac{\sigma_{n-1}^{ii}}{P_m^2}\lt(\sum_j\la_j^{m-1}h_{jji}\rt)^2\\
&-N\sigma_{n-1}^{ii}\la_i^2\lt<\nu, E\rt>+\frac{(n-1)\lt<\nu, E\rt>}{s-u}-\frac{\sigma_{n-1}^{ii}u_i^2}{(s-u)^2}.
\end{aligned}
\ee
Recall equation \eqref{cv1.19} we obtain for $i\geq 2$,
\be\label{cv1.26}
-\frac{\sigma_{n-1}^{ii}u_i^2}{(s-u)^2}=-\frac{\sigma_{n-1}^{ii}}{P_m^2}\lt(\sum_j\la_j^{m-1}h_{jji}\rt)^2
+N^2\sigma_{n-1}^{ii}\la_i^2u_i^2-2N\frac{\sigma_{n-1}^{ii}\la_iu_i^2}{s-u}.
\ee
Hence, \eqref{cv1.25} becomes
\be\label{cv1.27}
\begin{aligned}
0&\geq-C\la_1+\sum_{i=2}^n\lt(N^2\sigma_{n-1}^{ii}u_i^2\la_i^2-2N\frac{\sigma_{n-1}^{ii}u_i^2\la_i}{s-u}\rt)\\
&-N\sum_{i=1}^n\sigma_{n-1}^{ii}\la_i^2\lt<\nu, E\rt>+\frac{(n-1)\lt<\nu, E\rt>}{s-u}-\frac{\sigma_{n-1}^{11}u_1^2}{(s-u)^2}.
\end{aligned}
\ee
It's easy to see that there exists some positive constant $c_0$ such that $\sigma_{n-1}^{11}\la_1\geq c_0.$
Moreover, for any fixed $i=1, \dots, n$ we have $\sigma_{n-1}\geq\sigma_{n-1}^{ii}\la_i$ (no summation here). Thus, we get
\be\label{cv1.28}
\begin{aligned}
0&\geq\lt(-\frac{c_0N\lt<\nu, E\rt>}{2}-C\rt)\la_1-2N\sum_{i=2}^n\frac{\sigma_{n-1}^{ii}u_i^2\la_i}{s-u}\\
&-\frac{N}{2}\sigma_{n-1}^{11}\la_1^2\lt<\nu, E\rt>+\frac{(n-1)\lt<\nu, E\rt>}{s-u}-\frac{\sigma_{n-1}^{11}u_1^2}{(s-u)^2}.
\end{aligned}
\ee
Note that
\[\sum_{i=1}^nu_i^2=\sum_{i=1}^n\lt<X_i, E\rt>^2=\lt<\nu, E\rt>^2-1\leq \lt<\nu, E\rt>^2 .\]
When we choose $N>0$ large enough such that $-N\lt<\nu, E\rt>\geq\frac{4C}{c_0},$ \eqref{cv1.28} yields
\be\label{cv1.29}
\frac{c_0N}{4}\la_1(-\lt<\nu, E\rt>)+\frac{N}{2}\sigma_{n-1}^{11}\la_1^2(-\lt<\nu, E\rt>)
\leq\frac{3N}{s-u}\lt<\nu, E\rt>^2+\frac{\sigma_{n-1}^{11}}{(s-u)^2}\lt<\nu, E\rt>^2.
\ee
If at $x_0$ we have $s-u\geq \sigma_{n-1}^{11},$ then \eqref{cv1.29} implies
\[(s-u)\la_1\leq C_5.\]
If at $x_0$ we have $s-u<\sigma_{n-1}^{11},$ then from \eqref{cv1.29} we get
\[\frac{N}{2}\sigma_{n-1}^{11}\la_1^2(s-u)^2\leq- 3N(s-u)\lt<\nu, E\rt>-\sigma_{n-1}^{11}\lt<\nu, E\rt>,\]
this gives
\[(s-u)\la_1\leq C_5.\]
Therefore, we obtain the desired Pogorelov type $C^2$ interior estimates.
\end{proof}

\subsection{Convergence of $u^r$ to the solution $u$}
\label{cv3}
Combining estimates in Subsections \ref{cv0} - \ref{cv2} with the classic regularity theorem, we know that
there exists a subsequence of $\{\ur\},$ which we will denote by $\{\uri\},$ $r_i\rightarrow 1$ as $i\rightarrow\infty,$ converging locally smoothly to
a convex function $u$ defined over $\R^n,$ and $u$ satisfies
\[\sigma_{n-1}(\la[\M_u])=1.\]
\begin{lemm}
\label{cvlem1.10}
$u$ is a strictly convex function and $Du(\R^n)=B_1$.
\end{lemm}
\begin{proof}
We will denote $\us=x\cdot Du(x)-u(x).$ By Theorem \ref{intth1.1} we know that
to prove Lemma \ref{cvlem1.10} is equivalent to prove $\us$ is defined on $B_1.$

For any $\xi_0\in B_1$ and $i\geq i_0,$ there exists a point $x_0^{r_i}$
such that
\[D\uri(x_0^{r_i})=\xi_0.\]
In other words, we have $Du^{r_i*}(\xi_0)=x_0^{r_i}.$
Denote $\eta_{r_i}=dist(\xi_0, \T B_{r_i}),$ since $u^{r_i*}$ is convex, we get
\[|Du^{r_i*}(\xi_0)|\leq\frac{1}{\eta_{r_i}}\lt(\max u^{r_i*}-\min u^{r_i*}\rt)\leq \frac{2C_0}{\eta_{r_i}}.\]
This implies, when $i>i_0$ we have $x_0^{r_i}\in \bar{B}_{C_6}(0)\subset\R^n,$ where $C_6=\frac{2C_0}{\eta_{r_{i_0}}}$
is a constant. Therefore, there exists a subsequence of $\{x_0^{r_i}\}$ which we still denote as $\{x_0^{r_{i}}\}$
such that
\[\lim_{i\rightarrow\infty}x_0^{r_i}=x_0\in\bar{B}_{C_6}(0).\]
Moreover, by the local $C^1$ estimates we conclude that
\[Du(x_0)=\xi_0.\]
Since $\xi_0\in B_1$ is arbitrary, we proved this Lemma.
\end{proof}
An immediate consequence of Lemma \ref{cvlem1.10} is the following
\begin{coro}
\label{cvlem1.11}
Let $\{u^{r_i*}\}$ be the Legendre transform of $\{\uri\},$ where $\{\uri\}$ is the same convergent sequence we chose earlier.
Then, $\us=\lim_{i\rightarrow \infty}u^{r_i*}$ solves equation \eqref{cs1.1}. Moreover, $u^*$ is the Legendre transform of
$u=\lim_{i\rightarrow\infty}\uri.$
\end{coro}

\section{A uniform bound for $\la[\M_u]$}
\label{ub}
In this section, we will show that the function $u$ we obtained in Section \ref{cv} satisfies $\la[\M_u]\leq C.$
In other words, the noncompact, strictly convex, spacelike hypersurface we constructed in Section \ref{cv} has bounded principal curvatures.

\begin{prop}
\label{ubprop1}
Let $\us$ be the solution of equation \eqref{cs1.1} and $u$ be the Legendre transform of $\us.$
Then the hypersurface $\M_u=\{(x, u(x)) |\,x\in\R^n\}$ has bounded principal curvatures.
\end{prop}
\begin{proof}
We will use the idea of \cite{LA} to obtain a Pogorelov type interior estimate.
For any $s>0,$ consider
$$\phi=e^{-\frac{ms}{s-u}}\left(\frac{-\langle\nu,E\rangle}{u}\right)^{mN}P_m,$$ where $m, N>0$ are constants to be determined later
and $u\geq 1$ (see Subsection \ref{cv0}).
 It's easy to see that $\phi$ achieves its local maximum at an interior point of
 $U_s=\{x\in\R^n |\, u(x)<s\},$ we will assume this point is $x_0.$ We can choose a local orthonormal frame
 $\{\tau_1, \cdots, \tau_n\}$ such that at $x_0,$
 $h_{ij}=\la_i\delta_{ij}$ and $\la_1\geq\la_2\geq\cdots\geq\la_n.$

 Differentiating $\log\phi$ at $x_0$ we get,
\begin{equation}\label{12.2}
\frac{\phi_i}{\phi}=\dfrac{\dsum_j\kappa_j^{m-1}h_{jji}}{P_m}+Nh_{ii}\frac{\langle \tau_i,E\rangle}{\langle\nu,E\rangle}-N\frac{u_i}{u}-\frac{ su_i}{(s-u)^2}=0,
\end{equation}
and
\begin{eqnarray}\label{ub1.0}
\\
\frac{\phi_{ii}}{\phi}-\frac{\phi_i^2}{\phi^2}&=&\frac{1}{P_m}[\dsum_j\kappa_j^{m-1}h_{jjii}+(m-1)\dsum_j\kappa_j^{m-2}h_{jji}^2
+\dsum_{p\neq q}\dfrac{\kappa_p^{m-1}-\kappa_q^{m-1}}{\kappa_p-\kappa_q}h_{pqi}^2] \nonumber\\
&&-\dfrac{m}{P_m^2}(\dsum_j\kappa_j^{m-1}h_{jji})^2 +Nh_{imi}\frac{\langle \tau_m,E\rangle}{\langle\nu,E\rangle}
+Nh_{ii}^2-Nh_{ii}^2\frac{u_i^2}{\langle\nu,E\rangle^2}\nonumber\\
&&+N\frac{h_{ii}\langle\nu,E\rangle}{u}+N\frac{u_i^2}{u^2}+ s\frac{h_{ii}\langle \nu,E\rangle}{(s-u)^2}-2s\frac{u_i^2}{(S-u)^3}\leq 0.\nonumber
\end{eqnarray}

Following the same arguments as Lemma \ref{cvlem1.8}, we immediately derive
\be\label{ub1.2}
\begin{aligned}
0\geq \sigma_{n-1}^{ii}\frac{\phi_{ii}}{\phi}&=\frac{\sigma_{n-1}^{ii}}{P_m}[\dsum_j\kappa_j^{m-1}h_{jjii}+(m-1)\dsum_j\kappa_j^{m-2}h_{jji}^2
+\dsum_{p\neq q}\dfrac{\kappa_p^{m-1}-\kappa_q^{m-1}}{\kappa_p-\kappa_q}h_{pqi}^2]\\
&-\dfrac{m\sigma_{n-1}^{ii}}{P_m^2}(\dsum_j\kappa_j^{m-1}h_{jji})^2 +N\sigma_{n-1}^{ii}h_{imi}\frac{\langle \tau_m,E\rangle}{\langle\nu,E\rangle}
+N\sigma_{n-1}^{ii}h_{ii}^2-N\sigma_{n-1}^{ii}h_{ii}^2\frac{u_i^2}{\langle\nu,E\rangle^2}\\
&+N\sigma_{n-1}^{ii}\frac{h_{ii}\langle\nu,E\rangle}{u}+N\sigma_{n-1}^{ii}\frac{u_i^2}{u^2}+ s\frac{\sigma_{n-1}^{ii}h_{ii}\langle \nu,E\rangle}{(s-u)^2}-2 s\frac{\sigma_{n-1}^{ii}u_i^2}{(s-u)^3}\\
&\geq -C\kappa_1+\sum_i(A_i+B_i+C_i+D_i-E_i)+N\sigma_k^{ii}\kappa_i^2-N\sigma_{n-1}^{ii}h_{ii}^2\frac{u_i^2}{\langle\nu,E\rangle^2}\\
&+N\sigma_{n-1}^{ii}\frac{h_{ii}\langle\nu,E\rangle}{u}+N\sigma_{n-1}^{ii}\frac{u_i^2}{u^2}+ s\frac{\sigma_{n-1}^{ii}h_{ii}\langle \nu,E\rangle}{(s-u)^2}-2 s\frac{\sigma_{n-1}^{ii}u_i^2}{(s-u)^3},
\end{aligned}
\ee
where $A_i,B_i,C_i,D_i,E_i$ are defined in the argument of Lemma \ref{cvlem1.8}. By Lemma 8, Lemma 9 and Corollary 10 in \cite{LRW} we can assume the following claim holds.
\begin{claim}
\label{cvlem1.900}
There exists two small positive constants $\delta$ and $\eta<1$. If $\kappa_k\leq \delta\kappa_1$,
we have
\be\label{cv1.240}
\sum_{i}(A_i+B_i+C_i+D_i-\lt(1+\frac{\eta}{m}\rt)E_i)\geq 0,
\ee
where $m>0$ sufficiently large.
\end{claim}
If \eqref{cv1.240} doesn't hold, which implies $\kappa_k> \delta\kappa_1$, using the equation $\sigma_{n-1}=1$, we would obtain an upper bound for
$\la_1$ at $x_0$ directly, then we would be done. Otherwise, we assume \eqref{cv1.240} holds. Thus, using \eqref{ub1.2}, we have

\be
\begin{aligned}\label{neweq}
0&\geq-C\kappa_1+\eta\frac{\sigma_{n-1}^{ii}}{P_m^2}
(\dsum_j\kappa_j^{m-1}h_{jji})^2+N\sigma_k^{ii}\kappa_i^2-N\sigma_{n-1}^{ii}h_{ii}^2\frac{u_i^2}{\langle\nu,E\rangle^2}\\
&+N\sigma_{n-1}^{ii}\frac{h_{ii}\langle\nu,E\rangle}{u}+N\sigma_{n-1}^{ii}\frac{u_i^2}{u^2}+ s\frac{\sigma_{n-1}^{ii}h_{ii}\langle \nu,E\rangle}{(s-u)^2}-2 s\frac{\sigma_{n-1}^{ii}u_i^2}{(s-u)^3}.
\end{aligned}
\ee
From equation \eqref{12.2} we obtain
\be\label{ub1.1}
\begin{aligned}
\left(\dfrac{\dsum_j\kappa_j^{m-1}h_{jji}}{P_m}\right)^2&=N^2\frac{\kappa_i^2u_i^2}{\langle\nu,E\rangle^2}+N^2\frac{u_i^2}{u^2}+\frac{s^2u_i^2}{(s-u)^4}\\
&+2N^2\frac{\kappa_iu_i^2}{u\langle \nu,E\rangle}+2N s\frac{\kappa_iu_i^2}{\langle\nu,E\rangle(s-u)^2}+2N s\frac{u_i^2}{u(s-u)^2}.
\end{aligned}
\ee
Inserting \eqref{ub1.1} into \eqref{neweq}, we have
\begin{eqnarray}
0&\geq&-C\kappa_1+N\sigma_{n-1}^{ii}\kappa_i^2+\eta\frac{s^2\sigma_{n-1}^{ii}u_i^2}{(s-u)^4}+N(N\eta-1)\sigma_{n-1}^{ii}\kappa_i^2\frac{u_i^2}{\langle\nu,E\rangle^2}\nonumber\\
&&+2N^2\eta\frac{\sigma_{n-1}^{ii}\kappa_iu_i^2}{u\langle \nu,E\rangle}+2N s\eta\frac{\sigma_{n-1}^{ii}\kappa_iu_i^2}{\langle\nu,E\rangle(s-u)^2}+2N s\eta\frac{\sigma_{n-1}^{ii}u_i^2}{u(s-u)^2}\nonumber\\
&&+N\sigma_{n-1}^{ii}\frac{h_{ii}\langle\nu,E\rangle}{u}+N(\eta N+1)\sigma_{n-1}^{ii}\frac{u_i^2}{u^2}+ s\frac{\sigma_{n-1}^{ii}h_{ii}\langle \nu,E\rangle}{(s-u)^2}-2 s\frac{\sigma_{n-1}^{ii}u_i^2}{(s-u)^3}\nonumber.
\end{eqnarray}
By Lemma \ref{cvlem1.6} we get
$$|\nabla u|=\frac{|Du|}{\sqrt{1-|Du|^2}}<-\langle\nu,E\rangle\leq \frac{u}{d_1}.$$
Moreover, notice that $\sigma_{n-1}^{ii}\kappa_i\leq 1$ (no summation), by a simple calculation we can see
\[\eta\frac{s^2\sigma_{n-1}^{ii}u_i^2}{(s-u)^4}+2Ns\eta\frac{\sigma_{n-1}^{ii}u_i^2}{u(s-u)^2}-2s\frac{\sigma_{n-1}^{ii}u_i^2}{(s-u)^3}\geq 0,\]
if we let $N>1/\eta^2$.
Thus, for $N>1$ large, we have
\begin{eqnarray}
0&\geq&-C\kappa_1+N\sigma_{n-1}^{ii}\kappa_i^2-\frac{3N^2}{d_1}-2N s\frac{u}{d_1(s-u)^2}- s\frac{(n-1)u}{d_1(s-u)^2}\nonumber.
\end{eqnarray}
This yields, at $x_0$
$$\kappa_1\leq C(N,d_1)\frac{s^2}{(s-u)^2}.$$
Therefore, in $U_s$ we have
$$\phi^{\frac{1}{m}}\leq C(N, d_1)e^{-\frac{s}{s-u}}\frac{s^2}{(s-u)^2}.$$
Now note that for $t\in[0,s],$
$$\varphi(t)=e^{-\frac{s}{s-t}}\frac{s^2}{(s-t)^2}\leq 4e^{-2}.$$
Hence, we obtain that at any point $x\in U_s$,
\be\label{ub1.3}
\phi^{\frac{1}{m}}\leq C(N,d_1).
\ee
Now, for any $x\in\R^n$, we can choose $s$ large such that $x\in U_{s/2}$. Then by \eqref{ub1.3} and Lemma \ref{cvlem1.6}, we conclude
\[\la_1(x)\leq C(N, d_1, d_2).\]
Since $x$ is arbitrary, we finish proving Proposition \ref{ubprop1}.
\end{proof}

We conclude that Theorem \ref{intth1.2} is proved. More specifically, in Section \ref{cs} we showed that in order to find a strictly
convex, spacelike hypersurface with constant $\sigma_{n-1}$ curvature, we only need to show there exists a solution to equation \eqref{cs1.1}.
Since equation \eqref{cs1.1} is degenerate, we considered the solvability of the approximate problem \eqref{cs1.2} in Section \ref{es} instead.
In Section \ref{cv}, we proved that there exists a subsequence of solutions to \eqref{cs1.2} that converges to the solution of \eqref{cs1.1}.
Finally, in Section \ref{ub} we proved that the solution we obtained indeed has bounded principal curvatures.


\bigskip


\begin{thebibliography}{99}


\bibitem{Ale}  A. D. Aleksandrov, {\em Uniqueness theorems for surfaces in the large,} Vestnik Leningrad Univ. Math. 13 (1958) 5--8.

%\bibitem {BK}
%I. Bakelman, and B. Kantor,
%{\em Existence of spherically homeomorphic hypersurfaces in Euclidean
%space with prescribed mean curvature}.
%{ Geometry and Topology,} Leningrad, {1},(1974), 3--10.

\bibitem{Ball}
J.M. Ball,
{\em Differentiability properties of symmetric and isotropic functions,}
Duke Math. J. 51 (1984), no. 3, 699 --728.


\bibitem{BS} P. Bayard and O.C. Schn\"urer, {\em Entire spacelike hypersurfaces of constant Gauss curvature in Minkowski space},
J. Reine Angew. Math. 627 (2009), 1--29.

\bibitem{CG}
J. Cheeger and D. Gromoll,
{\em The splitting theorem for manifolds of nonnegative Ricci curvature,}
J. Differential Geometry 6 (1971/72), 119--128.

\bibitem{CGM}
L. Caffarelli, P. Guan, and X.-N. Ma,
{\em A constant rank theorem for solutions of fully nonlinear elliptic equations,}
Comm. Pure Appl. Math. 60 (2007), no. 12, 1769--1791.

\bibitem{CNS3}L. Caffarelli, L. Nirenberg and J. Spruck,
{\em The Dirichlet problem for nonlinear second-order elliptic equations. III. Functions of the eigenvalues of the Hessian,}
Acta Math. 155 (1985), no. 3-4, 261--301.

\bibitem{CY} S.-Y. Cheng and S.-T. Yau, {\em Maximal spacelike hypersurfaces in the Lorentz-Minkowski spaces},  Ann. of Math. 104, 407--419.

\bibitem{CY2} S.-Y. Cheng and S.-T. Yau, {\em Differential equations on Riemannian manifolds and their geometric application},  Comm. Pure Appl. Math., 28, 337--354 (1975).

\bibitem{CY3} S.-Y. Cheng and S.T. Yau, {\em Hypersurfaces with constant scalar curvature}, Math. Ann. 247, 81--93 (1980).

\bibitem{CT} H.I. Choi and A.E. Treibergs, {\em Gauss maps of spacelike constant mean curvature hypersurfaces of Minkowski space}, J. Diff. Geo. 32 (1990), 775--817.

\bibitem{EH} K. Ecker and G. Huisken, {\em Immersed hypersurfaces with constant Weingarten curvature}, Math. Ann. 283 (1989), no. 2, 329--332.



\bibitem{Guan}
B. Guan,
{\em The Dirichlet problem for Hessian equations on Riemannian manifolds,}
Calc. Var. Partial Differential Equations 8 (1999), no. 1, 45--69.

\bibitem{GJS} B. Guan, H.-Y. Jian and R. M. Schoen, {\em Entire spacelike hypersurfaces of prescribed Gauss curvature in Minkowski space}, J. Reine Angew. Math. 595 (2006), 167--188.

\bibitem{GM}
P. Guan and X.-N. Ma,
{\em The Christoffel-Minkowski problem. I. Convexity of solutions of a Hessian equation,}
Invent. Math. 151 (2003), no. 3, 553--577.

\bibitem{HN} J. Hano and K. Nomizu, {\em On isometric immersions of hyperbolic plane into the Lorentz-Minkowski space and the Monge-Ampere equation of a certain type,} Math. Ann., 262(1983), no.1, 245--253.


\bibitem{GRW} P. Guan, C. Ren and Z. Wang, {\em Global $C^2$ estimates
for curvature equation of convex solution,} Comm. Pure Appl. Math. LXVIII(2015) 1287--1325.

\bibitem{Hsi} C. C. Hsiung, {\em Some integral formulas for closed hypersurfaces,} Math. Scand. 2 (1954) 286--294.

\bibitem{HS} G. Huisken  and C. Sinestrari,
{\em Convexity estimates for mean curvature flow and singularities of
mean convex surfaces}, Acta Math. 183 (1999), no. 1, 45--70.

\bibitem{Kor}
N. J. Korevaar, {\em  A priori interior gradient bounds for
solutions to elliptic Weingarten equations}, Ann. Inst. H.
Poincar\'e, Anal. Non Lin\'eaire  4(1987), 405--421.

\bibitem{Lie} H. Liebmann, {\em Eine neue Eigenschaft der Kugel, Nachr. Kg. Ges. Wiss. Gottingen,} Math.-Pys. Kl., 1899,44--55.

\bibitem{LRW} M. Li, C. Ren and Z. Wang,
{\em An interior estimate for convex solutions and a rigidity theorem, }
J. Funct. Anal. 270 (2016), no. 7, 2691--2714.

\bibitem{LA} A.-M. Li, {\em Spacelike hypersurfaces with constant Gauss-Kronecker curvature in the Minkowski space},
Arch. Math., Vol.64, 534--551 (1995).

\bibitem{Ros} A. Ros,
{\em Compact hypersurfaces with constant scalar curvature and a congruence theorem. With an appendix by Nicholas J. Korevaar}, J. Diff. Geom. 27(1988), no. 2, 215--223.

\bibitem{RW} C. Ren and Z. Wang, {\em On the curvature estimates for Hessian equations}, American Journal of Mathematic, Vol.141, No.5, 1281--1315, 2019.

\bibitem{T} A. E. Treibergs, {\em Entire spacelike hypersurfaces of constant mean curvature in Minkowski space}, Invent. Math. 66, 39--56 (1982).

\bibitem{Tru1}
N. S. Trudinger, {\em The Dirichlet problem for the prescribed
curvature equations}, Arch. Rational Mech. Anal., 111(1990),
153--170.


\bibitem{Wxj} X. Wang, {\em The k-Hessian equation. Geometric analysis and PDEs},
177--252, Lecture Notes in Math., (1977), Springer, Dordrecht,
(2009).




%\bibitem{Xin} Y.L. Xin, {\em On the Gauss image of a spacelike hypersurface with constant mean curvature in Minkowski space}, Commentarii Mathematici Helvetici, December 1991, Volume 66, Issue 1, pp 590--598
\end{thebibliography}
\end{document}